\documentclass[a4paper,12pt]{amsart}
\usepackage{amsmath}
\usepackage{mathtools}
\usepackage{amssymb}
\usepackage{amscd}
\usepackage[breaklinks=true]{hyperref}

\usepackage[backend=bibtex]{biblatex}
\usepackage{xspace}

\usepackage{tikz}
\usetikzlibrary{shapes,fit}
\tikzstyle{morphism}=[ellipse, fill=green!20!white, draw=black, inner sep=0pt, minimum height=16pt, minimum width=16pt]
\tikzstyle{block}=[ellipse, fill=blue, opacity=0.3]
\tikzstyle{Xsize}=[minimum size=1.3cm]

\newcommand{\done}[1]{}

\newtheorem{thm}{Theorem}[section]
\newtheorem{conj}[thm]{Conjecture}
\newtheorem{lemma}[thm]{Lemma}
\newtheorem{prop}[thm]{Proposition}
\newtheorem{cor}[thm]{Corollary}
\newtheorem{ass}{Assumption}

\theoremstyle{definition}
\newtheorem{defn}[thm]{Definition}
\newtheorem{ex}[thm]{Example}
\newtheorem{rem}[thm]{Remark}

\newcommand{\Set}[1]{\ensuremath{\mathcal{#1}}} 
\newcommand{\Dfn}[1]{\emph{#1}} 
\newcommand{\size}[1]{\left\lvert #1\right\rvert}
\newcommand{\eval}[2][\right]{\relax
  \ifx#1\right\relax \left.\fi#2#1\rvert}
\newcommand{\qbinom}[2]{\genfrac{[}{]}{0pt}{}{#1}{#2}_q}
\newcommand{\Sp}{\mathrm{Sp}}
\newcommand{\GL}{\mathrm{GL}}
\newcommand{\Hom}{\mathrm{Hom}}
\newcommand{\End}{\mathrm{End}}
\newcommand{\bN}{\mathbb{N}}
\newcommand{\bZ}{\mathbb{Z}}
\newcommand{\Sym}{\mathsf{\Lambda}} 
\newcommand{\cB}{\mathsf{B}} 
\newcommand{\cV}{\mathsf{Vect}} 
\newcommand{\Br}{\mathsf{Br}} 
\newcommand{\Par}{\mathsf{P}} 
\newcommand{\cJ}{\mathsf{J}} 
\newcommand{\T}{\mathsf{T}} 
\newcommand{\fS}{{\mathfrak S}}
\newcommand{\pr}{\mathrm{pr}}

\newcommand{\cb}{\mathsf{C}}
\newcommand{\cP}{\mathsf{P}} 
\newcommand{\cT}{\mathsf{T}} 
\newcommand{\dc}{\mathsf{D}} 

\newcommand{\Pf}{\mathrm{Pf}}

\DeclareMathOperator{\ev}{ev}
\DeclareMathOperator{\bev}{\overline{ev}} 
\newcommand{\hu}{regular}
\newcommand{\sC}{\mathsf{C}}

\DeclareMathOperator{\maj}{maj}
\DeclareMathOperator{\lds}{lds}
\DeclareMathOperator{\id}{id}

\DeclareMathOperator{\Inf}{Inf}

\DeclareMathOperator{\tr}{\mathbf{tr}}   
\DeclareMathOperator{\fix}{\mathbf{fix}} 
\DeclareMathOperator{\ch}{\mathbf{ch}}   
\DeclareMathOperator{\fd}{\mathbf{fd}}   

\title[combinatorial representation theory]{A combinatorial approach to classical representation theory}
\author{Martin Rubey}
\address{Fakult\"at f\"ur Mathematik und Geoinformation, TU Wien,
  Austria}%
\email{martin.rubey@tuwien.ac.at}%
\author{Bruce W. Westbury}%
\address{Department of Mathematics, University of Warwick, Coventry,
  CV4 7AL}%
\email{Bruce.Westbury@warwick.ac.uk}%
\date{May 2015}

\begin{document}

\begin{abstract} 
  A fundamental problem from invariant theory is to describe, given a
  representation $V$ of a group $G$, the algebra $\End_G(\otimes^r
  V)$ of multilinear functions invariant under the action of the
  group.  According to Weyl's classic, a first main (later: \lq
  fundamental\rq) theorem of invariant theory provides a finite
  spanning set for this algebra, whereas a a second main theorem
  describes the linear relations between those basic invariants.

  Here we use diagrammatic methods to carry Weyl's programme a step
  further, providing explicit bases for the subspace of $\otimes^r V$
  invariant under the action of $G$, that are additionally preserved
  by the action of the long cycle of the symmetric group $\fS_r$.

  The representations we study are essentially those that occur in
  Weyl's book, namely the defining representations of the symplectic
  groups, the defining representations of the symmetric groups
  considered as linear representations, and the adjoint
  representations of the linear groups.

  In particular, we present a transparent, combinatorial proof of a
  second fundamental theorem for the defining representation of the
  symplectic groups $\Sp(2n)$.  Our formulation is completely
  explicit and provides a very precise link to $(n+1)$-noncrossing
  perfect matchings, going beyond a dimension count.  Extending our
  argument to the $k$-th symmetric powers of these representations,
  the combinatorial objects involved turn out to be
  $(n+1)$-noncrossing $k$-regular graphs.  As corollaries we obtain
  instances of the cyclic sieving phenomenon for these objects and
  the natural rotation action.

  In general, we are able to derive branching rules for the diagram
  algebras corresponding to the representations in a uniform way.
  Moreover, we compute the Frobenius characteristics of modules of
  the diagram algebras restricted to the action of the symmetric
  group.  Via a general theorem, we also obtain the isotypic
  decomposition of $\otimes^r V$ when $n$ is large enough in
  comparison to $r$.
\end{abstract}

\maketitle
\begin{center} Dedicated to Mia and George
\end{center}

\tableofcontents

\section{Introduction}
The objective of this paper is to exhibit examples of the cyclic sieving phenomenon involving
classical groups and to make this accessible to a reader with a background in combinatorics
rather than representation theory. In order to achieve this we have been led to give a contemporary, combinatorial account of the representation theory in \cite{MR1488158}.

The method used to exhibit the cyclic sieving phenomenon was given in \cite{1512}. The method starts
from a representation $V$ of a reductive group $G$. Then, for $r\ge 0$, we have the subspace
of $\otimes^rV$ invariant under the action of $G$. This subspace has a natural action of the
symmetric group, $\fS_r$. In order to exhibit the cyclic sieving phenomenon we then need to
determine the Frobenius character of this representation and to construct a basis invariant
under the action of the cyclic group generated by the long cycle. In this paper we apply this
method to three families of classical groups; the defining representations of the symplectic groups,
the defining permutation representations of the symmetric groups and the adjoint representations
of the general linear groups.

It is now well-established that an effective method of understanding the tensor powers of these
representations is to use the diagram categories. The primary benefit of the use of diagram categories
is that it gives simple conceptual proofs avoiding detailed calculations. A secondary benefit
is that the methods work directly with modules which gives stronger results than character
calculations. Although in this paper we focus on semisimple algebras this approach lends itself
to the general case. For example, the blocks of the centraliser algebras are determined in
\cite{MR2813567}, \cite{MR2550471}, \cite{MR2417984}. This approach is also important in representation
theory as it is characteristic free. However this is left implicit in this paper for simplicity.

The basic problem of decomposing these tensor powers
is equivalent to understanding the structure of the endomorphism algebras $\End_G(\otimes^rV)$. In our examples this is a direct sum of matrix algebras. Then the starting point for this method is that there are homomorphisms from the endomorphism algebras in the diagram category onto these endomorphism algebras.
This is known as the first fundamental theorem. The basic problem can then be solved in two steps. The first step is to determine the structure of the endomorphism algebras in the diagram category and to construct
the simple modules. This achieved following the method in \cite{MR2658143}.
The second step is to determine the kernel. This second step is known as the second fundamental theorem.

The main problem in this paper is a refinement of this basic problem and is solved following the same approach. There are inclusions $K\fS_r\rightarrow \End_G(\otimes^rV)$, where $K$ is the ground field. The problem now is to consider each simple module of the endomorphism algebra as a $K\fS_r$-module
and to determine the Frobenius character (a homogeneous symmetric function of degree $r$). The significance of this is that it describes the decomposition of $\otimes^rV$ under the action of $G\times\fS_r$.
This is equivalent to describing the branching rules for the homomorphism $G\rightarrow \GL(V)$ and to describing the decomposition of the Schur functors evaluated at $V$.

In particular the invariant tensors have an action of the symmetric group and we give
formulae for their Frobenius characters. For the first two families we describe a basis
which is invariant under the long cycle. This then gives examples of the cyclic sieving phenomenon.
In the example arising from the symplectic group $\Sp(2n)$, the set is the set of $(n+1)$ noncrossing
perfect matchings and the generator of the cyclic group acts by rotation. Our proof of the second
fundamental theorem is simpler than existing proofs and avoids the use of algebraic geometry.

The organisation of this paper is as follows. The background needed to make the paper
self-contained is presented in the first four sections. The next two sections give an account of the
features and methods common to the three examples. The final three sections give the results for each
of the three families. These three sections follow the following template:
\begin{itemize}
 \item The first subsection describes the relevant diagram category.
 \item The second subsection constructs the simple modules for the endomorphism algebras
in the diagram category and determines their Frobenius characters.
 \item The third subsection gives branching rules for the inclusion of an endomorphism
algebra in its successor.
 \item The fourth subsection gives the evaluation functors and in the first two families
describes the kernel and a basis for the image.
 \item The fifth subsection constructs the simple modules for the endomorphism algebras
and relates the parametrisation of simple modules arising
from the diagram category with the standard parametrisation.
 \item The sixth subsection discusses the cyclic sieving phenomenon using the second fundamental
theorem and the Frobenius characters of invariant tensors.
\end{itemize}

\section{Combinatorial species}
In this section we briefly recall the theory of Joyal's species, fix
notation and give some examples we will use throughout.

Let $\cB$ be the category of finite sets with bijections.  A
\Dfn{combinatorial species} (in $d$ variables, also called \lq
sorts\rq) is a functor $F: \cB^d\to \cB$.  By convention, the
application of $F$ is written with square brackets: for a $d$-tuple
of sets $U_1,\dots,U_d$
\[
F[U_1,\dots,U_d]
\]
is the set of $F$-structures on the given tuple of sets, and, for a
$d$-tuple of bijections $\sigma_i:U_i\to V_i$ we write
\[
F[\sigma_1,\dots,\sigma_d]: F[U_1,\dots,U_d]\to F[V_1,\dots,V_d]
\]
for the so-called transport of structures.  For ease of notation let
us abbreviate $F[\{1,\dots,r_1\},\dots,\{1,\dots,r_d\}]$ to
$F[r_1,\dots,r_d]$.  Then, $F$ is the same as a family of permutation
representations
\[
(\fS_{r_1}\times\dots\times \fS_{r_d})\times F[r_1,\dots,r_d]\to
F[r_1,\dots,r_d].
\]
In particular, if $d=1$ and $F[k]=\emptyset$ for $k\neq r$, we
identify $F$ with the corresponding permutation representation of
$\fS_r$.

To every combinatorial species $F$ we associate its \Dfn{cycle index
series}.  This is the symmetric function in $d$ alphabets $\mathbf
x_1,\dots,\mathbf x_d$
\begin{equation*}
  \sum_{r_1,\dots,r_d}\frac{1}{r_1!\cdots r_d!}
  \sum_{\sigma_1\in\fS_{r_1},\dots,\sigma_d\in\fS_{r_d}}%
  \fix{F[\sigma_1,\dots,\sigma_d]}\, %
  p_{\lambda(\sigma_1)}(\mathbf x_1) %
  \cdots p_{\lambda(\sigma_d)}(\mathbf x_d),
\end{equation*}
where $\fix$ denotes the number of fixed points of a permutation,
$\lambda(\sigma)$ is the cycle type of the permutation $\sigma$ and
$p$ denotes the power sum symmetric functions.

We will additionally use the following variant of combinatorial
species: a \Dfn{tensorial species} is a functor $F: \cB^d\to \cV$,
where $\cV$ is the category of finite dimensional vector spaces.
Such a functor $F$ is equivalent to a family of linear
representations of products of $d$ symmetric groups.  Note that we
can transform any combinatorial species into a tensorial species by
declaring the set $F[U_1,\dots,U_d]$ to be the basis of a vector
space.  Conversely, if there is a basis of $F[r_1,\dots,r_d]$ that is
permuted by $F[\sigma_1,\dots,\sigma_d]$ then we can regard $F$ as a
combinatorial species.

For any tensorial species $F$, the associated \Dfn{Frobenius character} is
the symmetric function, $\ch F$, given by
\begin{equation*}
  \sum_{r_1,\dots,r_d}\frac{1}{r_1!\cdots r_d!}
  \sum_{\sigma_1\in\fS_{r_1},\dots,\sigma_d\in\fS_{r_d}}%
  \tr{F[\sigma_1,\dots,\sigma_d]}\, %
  p_{\lambda(\sigma_1)}(\mathbf x_1) %
  \cdots p_{\lambda(\sigma_d)}(\mathbf x_d),
\end{equation*}
where $\tr$ denotes the trace of an endomorphism.  Note that the
Frobenius character coincides with the cycle index series when $F$
can be regarded as a combinatorial species.

Since combinatorial and tensorial species are completely analogous
concepts we will henceforth use the same notation for either concept,
drop the adjectives and simply write of \Dfn{species}.

The only point to keep in mind that a tensorial species is determined
up to isomorphism by its Frobenius character, while this is not true
for combinatorial species.  In particular, to show that two
combinatorial species are isomorphic it is not sufficient to compare
their Frobenius characters, see~\cite[Section 2.6, Equations
(30)--(31)]{BLL}.

For computations with species in several variables it is convenient
to introduce a \lq singleton\rq\ species $X_1,X_2,\dots,X_d$ for each
variable by setting
\[
X_i[U_1,\dots,U_d] =
\begin{cases}
  \{U_i\} &\text{if $\size{U_i} = 1$ and $U_j=\emptyset$ for $j\neq
    i$}\\
  \emptyset &\text{otherwise}.
\end{cases}
\]
Thus, $\ch X_i = p_1(\mathbf x_i)$.  In particular, this definition
allows us to make the number of variables of a species $F$ explicit
by writing
\begin{equation}
  \label{eq:multisort-species}
  F(X_1,\dots, X_d).
\end{equation}

A fundamental example for a (univariate) species is $h_r$, the
species of sets with $r$ elements, which corresponds to the trivial
representation of $\fS_r$.  More formally,
\[
h_r[U] =
\begin{cases}
  \{U\} & \text{if $\size{U} = r$}\\
  \emptyset & \text{otherwise}.
\end{cases}
\]
Its cycle index series or Frobenius character, also denoted $h_r$,
is the $r$-th complete homogeneous symmetric function.

The interest in the theory of species lies in the fact that it
provides combinatorial interpretations for many operations on group
actions.  For the most important operations on species this is
illustrated in great detail in~\cite{BLL}, so we will be brief in the
following.  Since for all operations the definitions of the group
actions is always the obvious one we restrict ourselves to the
definition of the underlying sets.

Let $F$ and $G$ be two species in $d$ variables.  Then $F+G$ is
defined on sets as
\[
(F+G)[U_1,\dots,U_d] = F[U_1,\dots,U_d] \amalg G[U_1,\dots,U_d],
\]
and, for tensorial species,
\[
(F+G)[U_1,\dots,U_d] = F[U_1,\dots,U_d] \oplus G[U_1,\dots,U_d].
\]
Clearly, $\ch(F+G) = \ch F + \ch G$.  Informally, an
$(F+G)$-structure is either an $F$-structure or a $G$-structure.

The product $F\cdot G$ of species is defined on sets as
\[
(F\cdot G)[U_1,\dots,U_d] = \coprod_{V_i\amalg W_i = U_i}
F[V_1,\dots,V_d] \times G[W_1,\dots,W_d].
\]
and, for tensorial species,
\[
(F\cdot G)[U_1,\dots,U_d] = \bigoplus_{V_i\amalg W_i = U_i}
F[V_1,\dots,V_d] \otimes G[W_1,\dots,W_d].
\]
Thus, $\ch(F\cdot G) = (\ch F)\cdot(\ch G)$.  Informally, an $(F\cdot
G)$- structure on a set $U$ is a pair consisting of an $F$-structure
on a set $V$ and a $G$-structure on a set $W$, such that $U$ is the
disjoint union of $V$ and $W$.

When $F_r$ and $G_s$ are representations of $\fS_r$ and $\fS_s$,
their product is precisely the induction product (known also as
external, outer or Cauchy product):
\[
F_r\cdot G_s = (F_r\times
G_s)\uparrow_{\fS_r\times\fS_s}^{\fS_{r+s}}.
\]

We define the next operation, partitional composition only for
univariate species $F$, which is sufficient for our calculations and
simpler than the general case.  Let $G$ a $d$-variate species.  Then
$F\circ G$ is the $d$-variate species corresponding to the plethysm
of representations:
\[
(F\circ G)[U_1,\dots,U_d] = \coprod_{\text{$\pi$ set partition of
    $\coprod_i U_i$}} F[\pi] \times \prod_{(B_1,\dots,B_d)\in\pi}
G[B_1,\dots,B_d],
\]
where we regard each block of the set partition $\pi$ as a $d$-tuple
of sets $(B_1,\dots,B_d)$ such that $B_i\subseteq U_i$ for $1\leq
i\leq d$.  For tensorial species, this is
\[
(F\circ G)[U_1,\dots,U_d] = \bigoplus_{\text{$\pi$ set partition of
    $\coprod_i U_i$}} F[\pi] \otimes
\bigotimes_{(B_1,\dots,B_d)\in\pi} G[B_1,\dots,B_d].
\]
Alternatively, we will also write $F\big(G(X_1,\dots,X_d)\big)$.
Note that the notation introduced in
equation~\ref{eq:multisort-species} is consistent with the definition
of partitional composition.  For the Frobenius character we have
$\ch(F\circ G) = \ch F\circ \ch G$, the plethysm of symmetric
functions.

An important special case of the partitional composition arises when
$F$ is the species of sets $H = \sum_{r\geq 0} h_r$.  An $(H\circ
G)$-structure should indeed be thought of a set of $G$-structures.
In particular, we will consider the following species:
\begin{itemize}
\item the species of perfect matchings $M$, which can be expressed as
  $H\circ h_2$, since a perfect matching is a set of two-element
  sets,
\item the species of set partitions $P$, which can be expressed as
  $H\circ H_+$, where $H_+ = \sum_{r\geq 1} h_r$ is the species of
  non-empty sets, since a set partition is a set of non-empty sets,
\item the species of permutations $S$, which can be expressed as
  $H\circ C_+$, where $C_+$ is the species of cycles, since a
  permutation is a set of cycles,
\item the bivariate species of bijections between a set of sort $X$
  and a set of sort $Y$ is $H(X\cdot Y)$.
\end{itemize}

We will employ two further operations, the diagonal of a bivariate
species and the scalar product of two species. These were both
introduced in \cite{MR1506633} and for a textbook exposition
see \cite{MR0357214}.

The diagonal of a bivariate species $F$ is the univariate species
$\nabla F$ with $\nabla F[U] = F[U, U]$.
The Frobenius character of $\nabla F$ is the
diagonal of the Frobenius character of $F$, that is, the linear extension
of
\[
\nabla \frac{p_\lambda(\mathbf x)}{z_\lambda} \frac{p_\mu(\mathbf
  y)}{z_\mu} =
\begin{cases}
  \frac{p_\lambda(\mathbf x)}{z_\lambda}&\text{if $\lambda=\mu$}\\
  0&\text{otherwise.}
\end{cases}
\]

To define the scalar product of two species two auxiliary definitions
are helpful.  Let $F$ be a species which is homogeneous of degree $r$
in its first variable, that is, $F[U_1,\dots,U_d]$ is the empty set
whenever $\size{U_1}\neq r$.  Then $\eval{F}_{X_1=1}$ is the
$(d-1)$-variate species defined by sending $[U_2,\dots,U_d]$ to the
set of orbits of $F[\{1,\dots,r\}, U_2,\dots,U_d]$ under the action
of $\fS_r$ on $\{1,\dots,r\}$.  As long as the obvious finiteness
conditions are met, this definition can be extended by linearity.

The Cartesian product of a $c$-variate species $F$ and a
$(c+d)$-variate species $G$ with respect to the first $c$ variables
is defined on sets as 
\[
F\ast_{X_1,\dots,X_c} G[U_1,\dots, U_{c+d}] = F[U_1,\dots, U_c]\times
G[U_1,\dots, U_{c+d}],
\]
and, for tensorial species
\[
F\ast_{X_1,\dots,X_c} G[U_1,\dots, U_{c+d}] = F[U_1,\dots,
U_c]\otimes G[U_1,\dots, U_{c+d}].
\]
When $d=0$ we will simply write $F\ast G$.  In this case, when $F$
and $G$ are representations of $\fS_r$, their Cartesian product is
precisely the Kronecker product of representations (known also as
internal product).  We remark that several notations for this product
are in use: for species it seems that \lq$\times$\rq\ is predominant
(eg., Joyal~\cite{MR633783} and Bergeron, Labelle,
Leroux~\cite{BLL}), whereas in the context of representation theory,
\lq$\ast$\rq\ is used more frequently.  Curiously,
Weyl~\cite[pg.~20]{MR1488158} used \lq$\times$\rq\ for
representations.

The scalar product of a $c$-variate species $F$ and a $(c+d)$-variate
species $G$ with respect to the first $c$ variables is then defined
as the $d$-variate species
\[
\langle F, G\rangle_{X_1,\dots,X_c} = \eval{(F\ast_{X_1,\dots,X_c}
  G)}_{X_1=\dots=X_c=1}.
\]
The representation theoretic interpretation of this operation is as
follows: let $F$ be an $\fS_r$-module and let $G$ be an
$\fS_r\times\fS_s$-bimodule.  Consider $F$ and $G$ as two bivariate
species, homogeneous of degree $r$ in the first and homogeneous of
degree $s$ in the second variable. Then the species $\langle F,
G\rangle_X$ is isomorphic (as a linear representation) to the
$\fS_s$-module $F\otimes_{K\fS_r} G$.

The following fact concerning the scalar product of species will be
useful:
\begin{lemma}[\protect{\cite[Section~2.2, Proposition~5]{MR927763}}]
  \label{lem:sets-reproducing-kernel}
  For any species $F$, we have
  \[
  \Big\langle F(X), H\big(X\cdot Y\big)\Big\rangle_X = F(Y).
  \]
  Similarly, we have
  \[
  \big\langle F(X_1,X_2), %
  H(X_1\cdot Y_1 + X_2\cdot Y_2)\big\rangle_{X_1, X_2} %
  = F(Y_1,Y_2)
  \]
  for any bivariate species $F$.
\end{lemma}
\begin{proof}[First proof] The Schur functions are an orthonormal basis
and the Cauchy identity is
\begin{equation*}
 H(X\cdot Y) = \sum_\lambda s_\lambda(X)\cdot s_\lambda(Y)
\end{equation*}
Hence $H(X\cdot Y)$ is a reproducing kernel
  \begin{equation*}
  \Big\langle F(X), H\big(X\cdot Y\big)\Big\rangle_X = 
\sum_\lambda \Big\langle F(X), s_\lambda(X)\Big\rangle_X s_\lambda(Y) = F(Y)
  \end{equation*}
\end{proof}
\begin{proof}[Second proof]
  We only prove the first statement, the second follows by the same
  reasoning.  We first note that $F(X)\ast_X H(X\cdot Y)=F(X\cdot Y)$:
  for any pair of sets $U_1$ and $U_2$, we find that $H(X\cdot
  Y)[U_1,U_2]$ is the set of bijections between $U_1$ and $U_2$.
  Therefore $F(X)\ast_X H(X\cdot Y)[U_1, U_2]$ is the set of pairs
  whose first component is an $F$-structure on $U_1$ and whose second
  component is a bijection between $U_1$ and $U_2$.  This coincides
  with the set $F(X\cdot Y)[U_1,U_2]$.

  By the definition of the scalar product we now have
  \begin{multline*}
  \Big\langle F(X), H\big(X\cdot Y\big)\Big\rangle_X = \\
  \eval{(F(X)\ast_X H(X\cdot Y))}_{X=1} = \eval{F(X\cdot Y)}_{X=1} =
  F(Y).
  \end{multline*}
\end{proof}

\section{Cyclic sieving}\label{section:csp}
The cyclic sieving phenomenon was introduced as a combinatorial theory in \cite{MR2087303}.
Here we review this theory from the perspective of representation theory.

Let $C$ be a finite cyclic group of order $r$ together with a
generator $c\in C$.  Put $\omega=\exp(2\pi i/r)$.  Then we identify
the character ring of $C$ with $\bZ[q]/\langle q^r-1\rangle$. For
$0\le k\le r-1$ the character of the representation $c^p\mapsto
\omega^{kp}$ is $q^k\in\bZ[q]/\langle q^r-1\rangle$.  In general, the
character of a linear representation $\rho\colon C\rightarrow
\GL(U)$, $\chi(\rho)\in\bZ[q]/\langle q^r-1\rangle$, is characterised
by
\begin{equation}\label{eq:CSP-ch}
  \chi(\rho)(\omega^k) = \tr \rho(c^k)
\end{equation}
for $0\le k\le r-1$.

Let $X$ be a finite set and $c\colon X\rightarrow X$ a bijection of
order $r$.  Associated to this is $\chi(X)\in \bZ[q]/\langle
q^r-1\rangle$, the character of the linear representation associated
to the permutation representation $X$.  The characterisation of the
character in equation \eqref{eq:CSP-ch} is now written as
\begin{equation*}
  \chi(X,c)(\omega^k) = \fix(c^k)
\end{equation*}
for $0\le k\le r-1$.

For example, if $k|r$ then there is a transitive permutation
representation of degree $r/k$ whose stabiliser is the cyclic group
of order $k$ generated by $c^{r/k}$.  The character of this
permutation representation is $(1-q^r)/(1-q^k)$.

Put $\chi(X)(q)=\sum_{k=0}^{r-1} a_k q^k$.
Then the combinatorial interpretation of the coefficient, $a_k$, is that it
is the number of orbits of $C$ whose stabiliser order divides $k$.

A triple $(X,c,P)$ with $P\in\bZ[q]$ exhibits the \emph{cyclic
  sieving phenomenon} if $P \equiv \chi(X)\mod{q^r-1}$.

\begin{ex}\label{ex:csp} For each partition $\lambda\vdash r$ put $X_\lambda=%
  \fS_\lambda \backslash \fS_r$ where $\fS_\lambda$ is a Young
  subgroup. The set $X_\lambda$ has an action of $\fS_r$ and we take
  $c_\lambda$ to be the action of long cycle. Then
\begin{equation*}
  \left(X_\lambda, c_\lambda, \qbinom{r}{\lambda}\right)
\end{equation*}
exhibits the cyclic sieving phenomenon, where $\qbinom{r}{\lambda}$
is the $q$-analogue of the multinomial coefficient.
\end{ex}

Our approach to the cyclic sieving phenomenon is based on the fake degree polynomial.
For each $r\ge 0$, let $\Sym_r$ be the $\bZ$-module of homogeneous symmetric functions
of degree $r$. Then, for each $r\ge 0$, $\fd_r\colon\Sym_r\rightarrow\bZ[q]$ is a morphism
of $\bZ$-modules. This is determined by the properties
\begin{equation*}
  \fd_{r+s}(f\cdot g) = \qbinom{r+s}{r} \fd_{r}(f)\fd_s(g)\qquad \fd_r(h_r)=1
\end{equation*}
where $f\in\Sym_r$, $g\in\Sym_s$ and $h_r$ is the complete homogeneous symmetric function.
Alternatively, the fake degree polynomial can be defined in terms of the principal 
specialisation.

On the basis of Schur functions the fake degree is given by
\begin{equation*}
 \fd (s_\lambda) = \sum_T q^{\maj(T)}
\end{equation*}
where the sum is over standard tableaux of shape $T$ and $\maj(T)$ is the
major index of the tableau $T$.

\begin{thm}\label{thm:csp} For any representation, $\fS_r\rightarrow \GL(U)$
the restriction to the cyclic subgroup generated by the long cycle
has character 
\[
\fd_r\big(\ch(U)\big) \mod{q^r-1}.
\]
\end{thm}
This can be proved and generalised using Springer's theory since the symmetric group is a
real reflection group and the long cycle is a regular element. Alternatively, there is an
elementary proof which deduces the result from Example \ref{ex:csp}.
For any partition $\lambda\vdash r$, the characteristic
map of the linear representation of $X_\lambda$ is $h_\lambda$ and
$\fd_r(h_\lambda)=\qbinom{r}{\lambda}$. Hence, by Example \ref{ex:csp}, the two linear
maps $\Sym_r\rightarrow \bZ[q]/\langle q^r-1\rangle$ agree on a basis and so are equal.

This gives a general technique from~\cite{1512} to obtain a cyclic sieving polynomial $P(q)$:
\begin{cor}\label{cor:spr}
Let $U$ be a representation of $\fS_r$ and let $X\subset U$ be a basis which is preserved
by the long cycle $c$. Take $P=\fd_r(\ch(U))$; then $(X,c,P)$ exhibits the cyclic sieving
phenomenon.
\end{cor}

The most straightforward application of this theorem is to permutation representations
of $\fS_r$. Here the representation and the basis are given so it remains to determine
the characteristic map and its fake degree polynomial.

\section{Monoidal categories}
In this section we give a concise account of symmetric monoidal and
pivotal categories.  Here we only discuss the strict versions of
these concepts; this is for brevity and simplicity and is justified
since all the examples of interest are strict.

Monoidal and symmetric monoidal categories were introduced in \cite{MR0170925}
and a coherence theorem is given in \cite{MR593254}. A more recent reference
is \cite{MR1268782}. In this paper we use the diagram calculus for
monoidal and symmetric monoidal categories described in \cite[Sections 1,2]{MR1113284}.

\subsection{Monoidal categories}
\begin{defn}
A \emph{strict monoidal category} consists of the following data:
\begin{itemize}
 \item a category $\sC$,
 \item a functor $\otimes \colon \sC \times \sC\rightarrow\sC$,
 \item an object $I\in\sC$.
\end{itemize}
The functor $\otimes$ is required to be associative. This is the
condition that
\begin{equation*}
\otimes \circ (\otimes \times\id) = \otimes \circ (\id \times \otimes)
\end{equation*}
as functors $\sC \times \sC \times \sC\rightarrow\sC$. Equivalently,
the following diagram commutes:
\begin{equation*}\begin{CD}
\sC \times \sC \times \sC @>{\otimes \times\id}>> \sC \times \sC \\
@V{\id \times \otimes}VV @VV{\otimes}V \\
\sC \times \sC @>>{\otimes}> \sC
\end{CD}\end{equation*}
The object $I$ is a left and right unit for $\otimes$. This means
that $\id_I \otimes \phi = \phi$ and $\phi \otimes \id_I = \phi$ for
all morphisms $\phi$.
\end{defn}

Let $\sC$ and $\sC'$ be strict monoidal categories. A functor
$F\colon\sC\rightarrow\sC'$ is \emph{monoidal}
if $F(I)=I'$ and the following diagram commutes
\begin{equation*}\begin{CD}
\sC \times \sC  @>{F\times F}>> \sC' \times \sC' \\
@V{\otimes}VV @VV{\otimes}V \\
\sC @>>{F}> \sC'
\end{CD}\end{equation*}

In this article we will make extensive use of string diagrams for
morphisms in a monoidal
category. A morphism is represented by a graph embedded in a
rectangle with boundary points
on the top and bottom edges. The edges of the graph are directed and
are labelled by objects.
The vertices are labelled by morphisms. A morphism $f\colon
x\rightarrow y$ is drawn as follows:
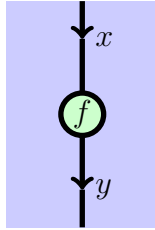
\begin{equation*}
\begin{tikzpicture}[line width=2pt]
\fill[color=blue!20] (-1,-1.5) rectangle (1,1.5);
\draw (0,1.5)[->] -- (0,1) node[anchor=west] {$x$}; \draw[->] (0,1) -- 
(0,-1);
\draw (0,-1) node[anchor=west] {$y$} -- (0,-1.5);
\draw (0,0) node[morphism] {$f$};
\end{tikzpicture}
\end{equation*}

The rules for combining diagrams are that composition is given by
stacking diagrams and the tensor product is given by putting diagrams
side by side.  The corresponding string diagrams are shown in
Figure~\ref{fig:comp}.

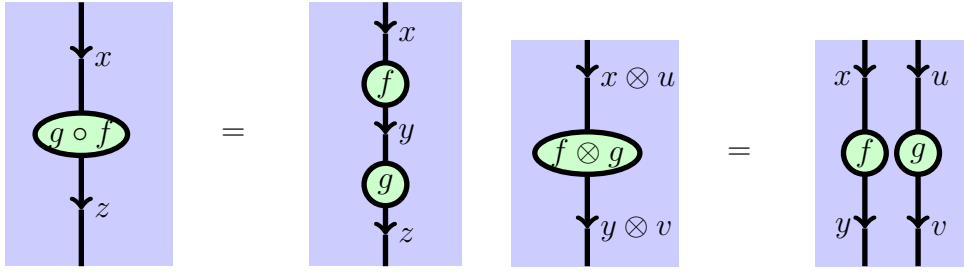
\begin{figure}
\begin{tikzpicture}[line width=2pt]
\fill[color=blue!20] (-1,-1.75) rectangle (1.2,1.75);
\draw (0,1.75)[->] -- (0,1) node[anchor=west] {$x$};
\draw[->] (0,1) -- (0,-1);
\draw (0,-1) node[anchor=west] {$z$} -- (0,-1.75);
\draw (0,0) node[morphism] {$g\circ f$};
\draw (2,0) node {$=$};
\fill[color=blue!20] (3,-1.75) rectangle (5,1.75);
\draw (4,1.75)[->] -- (4,1.333);
\draw (4,1.333)[->] node[anchor=west] {$x$} -- (4,0);
\draw (4,0)[->] node[anchor=west] {$y$} -- (4,-1.333);
\draw (4,-1.333) node[anchor=west] {$z$} -- (4,-1.75);
\draw (4,0.666) node[morphism] {$f$};
\draw (4,-0.666) node[morphism] {$g$};
\end{tikzpicture}
\hfill
\begin{tikzpicture}[line width=2pt]
\fill[color=blue!20] (-1,-1.5) rectangle (1.2,1.5);
\draw (0,1.5)[->] -- (0,1) node[anchor=west] {$x\otimes u$};
\draw[->] (0,1) -- (0,-1);
\draw (0,-1) node[anchor=west] {$y\otimes v$} -- (0,-1.5);
\draw (0,0) node[morphism] {$f\otimes g$};
\draw (2,0) node {$=$};
\fill[color=blue!20] (3,-1.5) rectangle (5,1.5);
\draw (3.65,1.5)[->] -- (3.65,1) node[anchor=east] {$x$};
\draw[->] (3.65,1) -- (3.65,-1);
\draw (3.65,-1) node[anchor=east] {$y$} -- (3.65,-1.5);
\draw (3.65,0) node[morphism] {$f$};
\draw (4.35,1.5)[->] -- (4.35,1) node[anchor=west] {$u$};
\draw[->] (4.35,1) -- (4.35,-1);
\draw (4.35,-1) node[anchor=west] {$v$} -- (4.35,-1.5);
\draw (4.35,0) node[morphism] {$g$};
\end{tikzpicture}
\caption{Composition and tensor product}\label{fig:comp}
\end{figure}
Rules for simplifying diagrams are that vertices labelled by an
identity morphism can be omitted and edges labelled by the trivial
object, $I$, can be omitted.  The string diagram for the rule that an
identity morphism can be omitted is shown in Figure~\ref{fig:id}.
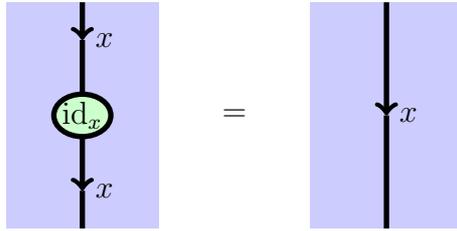
\begin{figure}
\begin{tikzpicture}[line width=2pt]
\fill[color=blue!20] (-1,-1.5) rectangle (1,1.5);
\draw (0,1.5)[->] -- (0,1) node[anchor=west] {$x$}; \draw[->] (0,1) -- 
(0,-1);
\draw (0,-1) node[anchor=west] {$x$} -- (0,-1.5);
\draw (0,0) node[morphism] {$\id_x$};
\draw (2,0) node {$=$};
\fill[color=blue!20] (3,-1.5) rectangle (5,1.5);
\draw (4,1.5)[->] -- (4,0); \draw (4,0) node[anchor=west] {$x$} -- (4,-1.5);
\end{tikzpicture}
\caption{Identity morphism}\label{fig:id}
\end{figure}

\subsection{Symmetric monoidal categories}
\begin{defn} A \emph{strict symmetric monoidal category} is a strict
monoidal category $\sC$ together with
natural isomorphisms $\alpha(x,y)\colon x\otimes y\rightarrow
y\otimes x$ for all objects $x,y$.
These are required to satisfy the following conditions:
\begin{equation}\label{a:square}
 \alpha(x,y)\circ\alpha(y,x) = \id_{y\otimes x}
\end{equation}
for all $x,y$ and
\begin{align}
 \alpha(v,x) \circ (f\otimes \id_x) &= (\id_x\otimes f) \circ
\alpha(u,x) \label{a:slide1} \\
\alpha(x,v) \circ (\id_x\otimes f) &= (f\otimes\id_x) \circ
\alpha(x,u) \label{a:slide2}
\end{align}
for all $f\colon u\rightarrow v$ and all $x$, and
\begin{align}
\alpha(x,y\otimes
z)&=\id_y\otimes\alpha(x,z)\,\circ\,\alpha(x,y)\otimes\id_z
\label{a:tensor1} \\
\alpha(x\otimes
y,z)&=\alpha(x,z)\otimes\id_y\,\circ\,\id_x\otimes\alpha(y, z)
\label{a:tensor2}
\end{align}
for all $x,y,z$.

A \emph{symmetric monoidal functor} is a functor compatible with this
structure.
\end{defn}

The string diagram for $\alpha(x,y)$ is as follows:
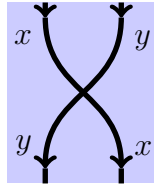
\begin{equation*}
\begin{tikzpicture}[line width=2pt]
\fill[color=blue!20] (-1,-1.2) rectangle (1,1.2);
\draw (-0.5,-1) -- (-0.5,-1.2);
\draw (0.5,-1) -- (0.5,-1.2);
\draw (-0.5,1.2)[->] -- (-0.5,1);
\draw (0.5,1.2)[->] -- (0.5,1);
\draw (-0.5,1)[->] node[anchor=north east] {$x$}  .. controls (-0.5,0) and (0.5,0) .. (0.5,-1) node[anchor=south west] {$x$};
\draw (0.5,1)[->] node[anchor=north west] {$y$}  .. controls (0.5,0) and (-0.5,0) .. (-0.5,-1) node[anchor=south east] {$y$};
\end{tikzpicture}
\end{equation*}
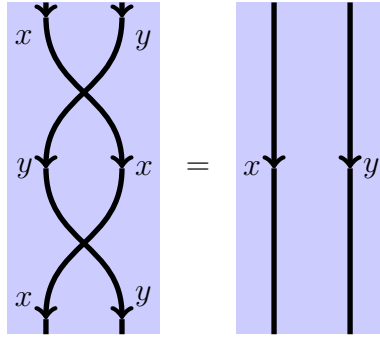
\begin{figure}
\begin{tikzpicture}[line width=2pt]
\fill[color=blue!20] (-1,-1.2) rectangle (1,3.2);
\draw (-0.5,-1) -- (-0.5,-1.2);
\draw (0.5,-1) -- (0.5,-1.2);
\draw (-0.5,3.2)[->] -- (-0.5,3) node[anchor=north east] {$x$};
\draw (0.5,3.2)[->] -- (0.5,3)  node[anchor=north west] {$y$};
\draw (-0.5,1)[->]   .. controls (-0.5,0) and (0.5,0) .. (0.5,-1)
node[anchor=south west] {$y$};
\draw (0.5,1)[->]  .. controls (0.5,0) and (-0.5,0) .. (-0.5,-1)
node[anchor=south east] {$x$};
\draw (-0.5,3)[->]  .. controls (-0.5,2) and (0.5,2) .. (0.5,1)
node[anchor= west] {$x$};
\draw (0.5,3)[->]  .. controls (0.5,2) and (-0.5,2) .. (-0.5,1)
node[anchor= east] {$y$};
\draw (1.5,1) node {$=$};
\fill[color=blue!20] (2,-1.2) rectangle (4,3.2);
\draw (2.5,3.2)[->] -- (2.5,1);
\draw (2.5,1) node[anchor= east] {$x$} -- (2.5,-1.2);
\draw (3.5,3.2)[->] -- (3.5,1);
\draw (3.5,1) node[anchor= west] {$y$} -- (3.5,-1.2);
\end{tikzpicture}
\caption{Symmetry}\label{fig:sym}
\end{figure}
The string diagram for condition~\eqref{a:square} is shown in
Figure~\ref{fig:sym}, the string diagrams for
conditions~\eqref{a:slide1} and~\eqref{a:slide2} are shown in
Figure~\ref{fig:rels-slide}, and those for
conditions~\eqref{a:tensor1} and~\eqref{a:tensor2} are shown in
Figure~\ref{fig:rels-tensor}.
\begin{figure}
\begin{tikzpicture}[line width=2pt]
\fill[color=blue!20] (-1,-1.2) rectangle (1,3.2);
\draw (-0.5,-1) node[anchor=south east] {$x$} -- (-0.5,-1.2);
\draw (0.5,-1) node[anchor=south west] {$v$} -- (0.5,-1.2);
\draw (-0.5,3.2)[->] -- (-0.5,1);
\draw (0.5,3.2)[->] -- (0.5,1);
\draw (-0.5,3.2)[->] -- (-0.5,3) node[anchor=north east] {$u$};
\draw (0.5,3.2)[->] -- (0.5,3)  node[anchor=north west] {$x$};
\draw (-0.5,1)[->] node[anchor=east] {$v$}  .. controls (-0.5,0) and
(0.5,0) .. (0.5,-1) ;
\draw (0.5,1)[->]  node[anchor=west] {$x$} .. controls (0.5,0) and
(-0.5,0) .. (-0.5,-1);
\draw (-0.5,2) node[morphism] {$f$};
\draw (1.5,1) node {$=$};
\fill[color=blue!20] (2,-1.2) rectangle (4,3.2);
\draw (2.5,3.2)[->] -- (2.5,3) node[anchor=north east] {$u$};
\draw (3.5,3.2)[->] -- (3.5,3)  node[anchor=north west] {$x$};
\draw (2.5,1) node[anchor=  east] {$x$} -- (2.5,-1.2);
\draw (2.5,-1) node[anchor= south east] {$x$};
\draw (3.5,1) node[anchor=  west] {$u$} -- (3.5,-1.2);
\draw (3.5,-1) node[anchor= south west] {$v$};
\draw (2.5,3)[->]  .. controls (2.5,2) and (3.5,2) .. (3.5,1) ;
\draw (3.5,3)[->]  .. controls (3.5,2) and (2.5,2) .. (2.5,1) ;
\draw (3.5,0) node[morphism] {$f$};
\end{tikzpicture}
\hfill
\begin{tikzpicture}[line width=2pt]
\fill[color=blue!20] (-1,-1.2) rectangle (1,3.2);
\draw (-0.5,-1) node[anchor=south east] {$v$} -- (-0.5,-1.2);
\draw (0.5,-1) node[anchor=south west] {$x$} -- (0.5,-1.2);
\draw (-0.5,3.2)[->] -- (-0.5,1) node[anchor=  east] {$x$};
\draw (0.5,3.2)[->] -- (0.5,1) node[anchor=  west] {$v$};
\draw (-0.5,3.2)[->] -- (-0.5,3) node[anchor=north east] {$x$};
\draw (0.5,3.2)[->] -- (0.5,3)  node[anchor=north west] {$u$};
\draw (-0.5,1)[->]   .. controls (-0.5,0) and (0.5,0) .. (0.5,-1) ;
\draw (0.5,1)[->]  .. controls (0.5,0) and (-0.5,0) .. (-0.5,-1) ;
\draw (0.5,2) node[morphism] {$f$};
\draw (1.5,1) node {$=$};

\fill[color=blue!20] (2,-1.2) rectangle (4,3.2);
\draw (2.5,3.2)[->] -- (2.5,3) node[anchor=north east] {$x$};
\draw (3.5,3.2)[->] -- (3.5,3)  node[anchor=north west] {$u$};
\draw (2.5,1) -- (2.5,-1.2);
\draw (3.5,1) -- (3.5,-1.2);
\draw (2.5,-1) node[anchor = south east] {$v$};
\draw (3.5,-1) node[anchor = south west] {$x$};
\draw (2.5,3)[->]  .. controls (2.5,2) and (3.5,2) .. (3.5,1)
node[anchor=  west] {$x$};
\draw (3.5,3)[->]  .. controls (3.5,2) and (2.5,2) .. (2.5,1)
node[anchor=  east] {$u$};
\draw (2.5,0) node[morphism] {$f$};
\end{tikzpicture}
\caption{Relations \eqref{a:slide1} and
  \eqref{a:slide2}}\label{fig:rels-slide}
\end{figure}
\begin{figure}
\begin{tikzpicture}[line width=2pt]
\fill[color=blue!20] (-1.75,-1.2) rectangle (1.75,1.2);
\draw (-0.5,-1) -- (-0.5,-1.2);
\draw (0.5,-1) -- (0.5,-1.2);
\draw (-0.5,1.2)[->] -- (-0.5,1);
\draw (0.5,1.2)[->] -- (0.5,1);
\draw (-0.5,1)[->] node[anchor=north east] {$x\otimes y$}
.. controls (-0.5,0) and (0.5,0) .. (0.5,-1) node[anchor=south west]
{$x\otimes y$};
\draw (0.5,1)[->] node[anchor=north west] {$z$}  .. controls (0.5,0)
and (-0.5,0) .. (-0.5,-1) node[anchor=south east] {$z$};
\draw (2.1,0) node {$=$};
\begin{scope}[xshift=3.5cm]
\fill[color=blue!20] (-1,-1.2) rectangle (2,1.2);
\draw (-0.5,-1) -- (-0.5,-1.2);
\draw (0.5,-1) -- (0.5,-1.2);
\draw (1.5,-1) -- (1.5,-1.2);
\draw (-0.5,1.2)[->] -- (-0.5,1);
\draw (0.5,1.2)[->] -- (0.5,1);
\draw (1.5,1.2)[->] -- (1.5,1);
\draw (-0.5,1)[->] node[anchor=north east] {$x$}  .. controls
(-0.5,0) and (0.5,0) .. (0.5,-1) node[anchor=south west] {$x$};
\draw (0.5,1)[->] node[anchor=north west] {$y$}  .. controls (0.5,0)
and (1.5,0) .. (1.5,-1) node[anchor=south west] {$y$};
\draw (1.5,1)[->] node[anchor=north west] {$z$}  .. controls (1.5,0)
and (-0.5,0) .. (-0.5,-1) node[anchor=south east] {$z$};
\end{scope}
\end{tikzpicture}

\vspace{1cm}

\begin{tikzpicture}[line width=2pt]
\fill[color=blue!20] (-1.75,-1.2) rectangle (1.75,1.2);
\draw (-0.5,-1) -- (-0.5,-1.2);
\draw (0.5,-1) -- (0.5,-1.2);
\draw (-0.5,1.2)[->] -- (-0.5,1);
\draw (0.5,1.2)[->] -- (0.5,1);
\draw (-0.5,1)[->] node[anchor=north east] {$x$}  .. controls
(-0.5,0) and (0.5,0) .. (0.5,-1) node[anchor=south west] {$x$};
\draw (0.5,1)[->] node[anchor=north west] {$y\otimes z$}  .. controls
(0.5,0) and (-0.5,0) .. (-0.5,-1) node[anchor=south east] {$y\otimes
z$};
\draw (2.1,0) node {$=$};
\begin{scope}[xshift=3.5cm]
\fill[color=blue!20] (-1,-1.2) rectangle (2,1.2);
\draw (-0.5,-1) -- (-0.5,-1.2);
\draw (0.5,-1) -- (0.5,-1.2);
\draw (1.5,-1) -- (1.5,-1.2);
\draw (-0.5,1.2)[->] -- (-0.5,1);
\draw (0.5,1.2)[->] -- (0.5,1);
\draw (1.5,1.2)[->] -- (1.5,1);
\draw (-0.5,1)[->] node[anchor=north east] {$x$}  .. controls
(-0.5,0) and (1.5,0) .. (1.5,-1) node[anchor=south west] {$x$};
\draw (0.5,1)[->] node[anchor=north east] {$y$}  .. controls (0.5,0)
and (-0.5,0) .. (-0.5,-1) node[anchor=south west] {$y$};
\draw (1.5,1)[->] node[anchor=north west] {$z$}  .. controls (1.5,0)
and (0.5,0) .. (0.5,-1) node[anchor=south east] {$z$};
\end{scope}
\end{tikzpicture}
\caption{Relations \eqref{a:tensor1} and \eqref{a:tensor2}}\label{fig:rels-tensor}
\end{figure}
Note that applying the relations \eqref{a:slide1}, \eqref{a:slide2}
to $f=\alpha$
give the braid relations for the morphisms $\alpha$.

\subsection{Duality}
\begin{defn}
Let $x$ and $y$ be objects in a monoidal category. Then $x$ is a
\emph{left dual}
of $y$ and $y$ is a \emph{right dual} of $x$ means that we are given
evaluation and
coevaluation morphisms
$I\rightarrow x\otimes y$ and $y\otimes x\rightarrow I$ such that
both of the
following composites are identity morphisms
\begin{align*}
 x\rightarrow I\otimes x\rightarrow x\otimes y\otimes x\rightarrow
x\otimes I\rightarrow x \label{s:trl} \\
 y\rightarrow y\otimes I\rightarrow y\otimes x\otimes y\rightarrow
I\otimes y\rightarrow y. 
\end{align*}
The object $x$ is \emph{dual} to $y$ if it is both a left dual and a
right dual.
\end{defn}

In the string diagrams we adopt the convention depicted in
Figure~\ref{fig:dual}.
\begin{figure}
 \begin{tikzpicture}[line width=2pt]
\fill[color=blue!20] (-1,-0.5) rectangle (1,0.5);
\draw (0,0.5)[->] -- (0,0) node[anchor=east] {$y$};
\draw (0,0) -- (0,-0.5);
\draw (1.5,0) node {$=$};
\fill[color=blue!20] (2,-0.5) rectangle (4,0.5);
\draw (3,-0.5)[->] -- (3,0) node[anchor=east] {$x$};
\draw (3,0) -- (3,0.5);
 \end{tikzpicture}
\caption{Convention for $x$ being a dual of $y$}\label{fig:dual}
\end{figure}
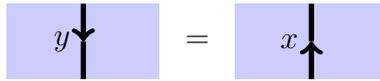
Using this convention and omitting the edge labelled $I$ the string diagrams
for the evaluation and coevaluation morphisms are:
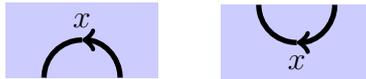
\begin{equation*}
\begin{tikzpicture}[line width=2pt]
\fill[color=blue!20] (-1,-1) rectangle (1,0);
\draw (0.5,-1)[->] arc (0:90:0.5);
\draw (0,-0.5) node[anchor=south] {$x$} arc (90:180:0.5);
\end{tikzpicture}
\qquad
\begin{tikzpicture}[line width=2pt]
\fill[color=blue!20] (-1,-1) rectangle (1,0);
\draw (0.5,0)[->] arc (0:-90:0.5);
\draw (0,-0.5) node[anchor=north] {$x$} arc (-90:-180:0.5);
\end{tikzpicture}
\end{equation*}

The string diagrams for the conditions on these morphisms are shown
in Figure~\ref{fig:cupcap}.
\begin{figure}
\begin{tikzpicture}[line width=2pt]
\fill[color=blue!20] (-1.25,-1) rectangle (1.25,1);
\draw[->] (1,1) -- (1,0) arc(0:-90:0.5);
\draw[->] (0.5,-0.5) node[anchor=north] {$x$} arc(-90:-180:0.5)
arc(0:90:0.5);
\draw (-0.5,0.5) node[anchor=south] {$x$} arc(90:180:0.5) -- (-1,-1);
\draw (1.5,0) node {$=$};
\fill[color=blue!20] (1.75,-1) rectangle (2.5,1);
\draw[->] (2,1) -- (2,0);
\draw (2,0) node[anchor=west] {$x$} -- (2,-1);
\end{tikzpicture}
\qquad\qquad
\begin{tikzpicture}[line width=2pt]
\fill[color=blue!20] (-1.25,-1) rectangle (1.25,1);
\draw[->] (1,-1) -- (1,0) arc(0:90:0.5);
\draw[->] (0.5,0.5) node[anchor=south] {$x$} arc(90:180:0.5) arc(0:-90:0.5);
\draw (-0.5,-0.5) node[anchor=north] {$x$} arc(-90:-180:0.5) -- (-1,1);
\draw (1.5,0) node {$=$};
\fill[color=blue!20] (1.75,-1) rectangle (2.5,1);
\draw[->] (2,-1) -- (2,0);
\draw (2,0) node[anchor=west] {$x$} -- (2,1);
\end{tikzpicture}
\caption{Relations for duals}\label{fig:cupcap}
\end{figure}
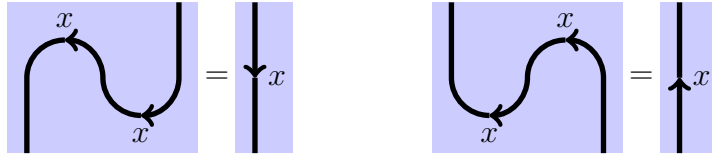

It follows that there are natural homomorphisms
\begin{equation*}
  \Hom(y\otimes u,v)\rightarrow \Hom(u,x\otimes v)
\end{equation*}
sending $\phi\colon y\otimes u\rightarrow v$ to the composite
\begin{equation*}
  u\rightarrow I\otimes u\rightarrow x\otimes y\otimes u\rightarrow
  x\otimes v.
\end{equation*}
Similarly there are natural homomorphisms
\begin{equation*}
  \Hom(u,x\otimes v)\rightarrow \Hom(y\otimes u,v).
\end{equation*}
sending $\phi\colon u\rightarrow x\otimes v$ to the composite
\begin{equation*}
  y\otimes u\rightarrow y\otimes x\otimes v\rightarrow I\otimes
  v\rightarrow  v.
\end{equation*}
It follows from the relations in Figure~\ref{fig:cupcap} that these
are inverse isomorphisms.

\begin{ex} Take a category $\sC$ and consider the category whose objects are
functors $F\colon \sC \rightarrow \sC$ and whose morphisms are natural
transformations. This is a monoidal category with $\otimes$ given by
composition.
Then $F$ is left dual to $G$ in this monoidal category is equivalent to
$F$ is left adjoint to $G$.
\end{ex}

\begin{ex} Let $K$ be a field. The category of vector spaces over $K$ is
a symmetric monoidal category. A vector space has a dual if and only if
it is finite dimensional. More generally, Let $K$ be a commutative ring.
The category of $K$-modules is a symmetric monoidal category.
A $K$-module has a dual if and only if it is finitely generated and
projective.
\end{ex}

\begin{defn}
  A \Dfn{strict pivotal category} is a strict monoidal category with
  an antimonoidal antiinvolution $\ast$ such that $x^\ast$ is a dual of
  $x$ for all objects $x$.
\end{defn}

\subsection{Rotation}
Let $x$ be an object in a monoidal category. Then $\Hom(I,\otimes^r
x)$ is the set of \emph{invariant tensors}.  In this section we
discuss two constructions of an action of the cyclic group on this
set each of which depends on additional structure. One construction,
which we call rotation, requires that $x$ has a dual.  The other
construction assumes that the category is symmetric monoidal. The
main result in this section to compare these two actions in the
situation when both are defined.

Let $x$ be an object of a symmetric monoidal category.  Then
$\otimes^r x$ has a natural action of $\fS_r$ where the action of the
Coxeter generator $s_i$ is given by the isomorphism
$(\otimes^{i-1}\id_x)\otimes\alpha(x,x)\otimes
(\otimes^{r-i-1}\id_x)$.  This induces an action of $\fS_r$ on the
set $\Hom(y,\otimes^r x)$ for any $y$. In particular for $y=I$ this
is an action of $\fS_r$ on the invariant tensors.

Assume $x$ has a dual $x^\ast$. Then, for any $y$, there is a natural
map $\Hom(I,y\otimes x)\rightarrow \Hom(I,x\otimes y)$ sending
$f\colon I\rightarrow y\otimes x$ to the composite
\begin{equation*}
 I\rightarrow x\otimes x^*\rightarrow x\otimes I\otimes x^*\rightarrow
x\otimes y\otimes x\otimes x^\ast\rightarrow x\otimes y\otimes I
\rightarrow x\otimes y
\end{equation*}

Taking $y=\otimes^{r-1}x$ gives the \emph{rotation map}
$\otimes^rx\rightarrow \otimes^rx$. Then, for any $x$, the $r$-th
power of the rotation map is the identity.

\begin{prop}\label{lem:rot} Let $x$ be an object of symmetric monoidal
category with dual $x^\ast$ such that the condition in Figure \ref{fig:con}
is satisfied.
\begin{figure}
\setlength{\tabcolsep}{0.6cm}
\begin{tabular}{cc}
\begin{tikzpicture}[line width=2pt]
\fill[color=blue!20] (-1,-1) rectangle (1,1);
\draw[->] (0.5,-1) .. controls (0.5,-0.5) and (-0.5,-0.5) .. (-0.5,0) arc(180:90:0.5);
\draw (0,0.5) node[anchor=south] {$x$} arc(90:0:0.5) .. controls (0.5,-0.5) and (-0.5,-0.5) .. (-0.5,-1);
\draw (1.5,0) node {$=$};
\fill[color=blue!20] (2,-1) rectangle (4,1);
\draw[->] (2.5,-1) -- (2.5,0) arc(180:90:0.5);
\draw (3,0.5) node[anchor=south] {$x$} arc(90:0:0.5) -- (3.5,-1);
\end{tikzpicture} &
\begin{tikzpicture}[line width=2pt]
\fill[color=blue!20] (-1,-1) rectangle (1,1);
\draw[->] (-0.5,-1) .. controls (-0.5,-0.5) and (0.5,-0.5) .. (0.5,0)
arc(0:90:0.5);
\draw (0,0.5) node[anchor=south] {$x$} arc(90:180:0.5) .. controls
(-0.5,-0.5) and (0.5,-0.5) .. (0.5,-1);
\draw (1.5,0) node {$=$};
\fill[color=blue!20] (2,-1) rectangle (4,1);
\draw[->] (3.5,-1) -- (3.5,0) arc(0:90:0.5);
\draw (3,0.5) node[anchor=south] {$x$} arc(90:180:0.5) -- (2.5,-1);
\end{tikzpicture} \\*[1cm]
\begin{tikzpicture}[line width=2pt]
\fill[color=blue!20] (-1,-1) rectangle (1,1);
\draw[->] (-0.5,1) .. controls (-0.5,0.5) and (0.5,0.5) .. (0.5,0)
arc(0:-90:0.5);
\draw (0,-0.5) node[anchor=north] {$x$} arc(-90:-180:0.5) .. controls
(-0.5,0.5) and (0.5,0.5) .. (0.5,1);
\draw (1.5,0) node {$=$};
\fill[color=blue!20] (2,-1) rectangle (4,1);
\draw[->] (2.5,1) -- (2.5,0) arc(180:270:0.5);
\draw (3,-0.5) node[anchor=north] {$x$} arc(-90:0:0.5) -- (3.5,1);
\end{tikzpicture} &
\begin{tikzpicture}[line width=2pt]
\fill[color=blue!20] (-1,-1) rectangle (1,1);
\draw[->] (0.5,1) .. controls (0.5,0.5) and (-0.5,0.5) .. (-0.5,0)
arc(180:270:0.5);
\draw (0,-0.5) node[anchor=north] {$x$} arc(-90:0:0.5) .. controls
(0.5,0.5) and (-0.5,0.5) .. (-0.5,1);
\draw (1.5,0) node {$=$};
\fill[color=blue!20] (2,-1) rectangle (4,1);
\draw[->] (2.5,1) -- (2.5,0) arc(180:270:0.5);
\draw (3,-0.5) node[anchor=north] {$x$} arc(-90:0:0.5) -- (3.5,1);
\end{tikzpicture}
\end{tabular}
\caption{Conditions on a representation}\label{fig:con}
\end{figure}
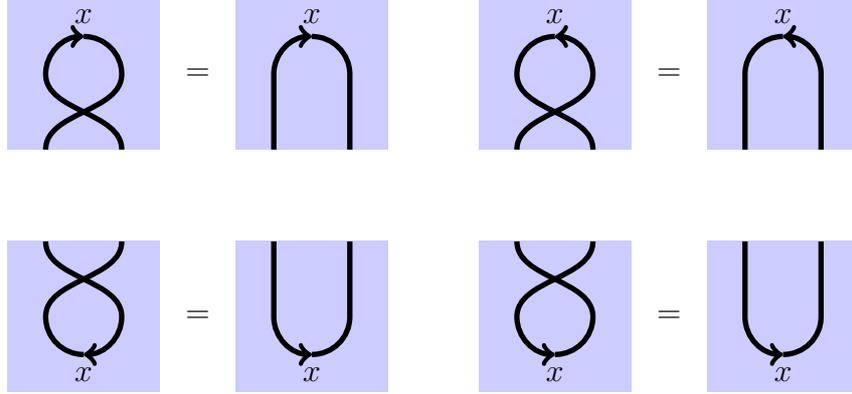
Then the action of the long cycle is given by rotation.
\end{prop}
It is clear that all the conditions on the string diagrams can be
interpreted as sliding the strings in the diagrams.  The converse is
that if the strings in one diagram can be slid to obtain another
diagram then the two morphisms are equal.

\begin{proof}
  In the following diagrams the half circle represents an invariant
  tensor.  The first diagram shows the map given by the symmetry and
  the final diagram shows the rotation map.
  
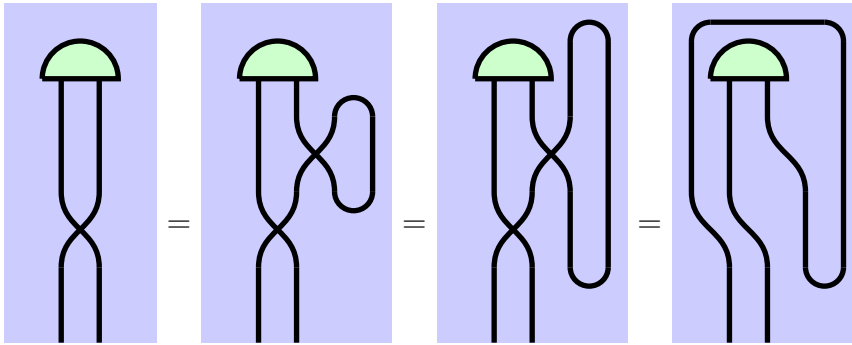
\begin{center}
    \begin{tikzpicture}[line width = 2pt]
      \fill[color=blue!20] (0,-1) rectangle  (2,3.5);
      \fill[morphism] (0.5,2.5) -- (1.5,2.5) arc (0:180:0.5);
      \draw (0.75,2.5) -- (0.75,1) .. controls (0.75,0.5) and (1.25,0.5) .. (1.25,0);
      \draw (1.25,2.5) -- (1.25,1) .. controls (1.25,0.5) and (0.75,0.5) .. (0.75,0);
      \draw (0.75,0) -- (0.75,-1);
      \draw (1.25,0) -- (1.25,-1);
    \end{tikzpicture}
    \raisebox{1.5cm}{$=$}
    \begin{tikzpicture}[line width = 2pt]
      \fill[color=blue!20] (0,-1) rectangle  (2.5,3.5);
      \fill[morphism] (0.5,2.5) -- (1.5,2.5) arc (0:180:0.5);
      \draw (0.75,2.5) -- (0.75,1) .. controls (0.75,0.5) and  (1.25,0.5) .. (1.25,0);
      \draw (1.25,2.5) -- (1.25,2) .. controls (1.25,1.5) and (1.75,1.5) .. (1.75,1);
      \draw (1.75,2) .. controls (1.75,1.5) and (1.25,1.5) .. (1.25,1);
      \draw (1.25,1) .. controls (1.25,0.5) and (0.75,0.5) .. (0.75,0);
      \draw (1.75,2) arc(180:0:0.25);
      \draw (2.25,2) -- (2.25,1);
      \draw (1.75,1) arc(-180:0:0.25);
      \draw (0.75,0) -- (0.75,-1);
      \draw (1.25,0) -- (1.25,-1);
    \end{tikzpicture}
    \raisebox{1.5cm}{$=$}
    \begin{tikzpicture}[line width = 2pt]
      \fill[color=blue!20] (0,-1) rectangle  (2.5,3.5);
      \fill[morphism] (0.5,2.5) -- (1.5,2.5) arc (0:180:0.5);
      \draw (0.75,2.5) -- (0.75,1) .. controls (0.75,0.5) and  (1.25,0.5) .. (1.25,0);
      \draw (1.25,2.5) -- (1.25,2) .. controls (1.25,1.5) and (1.75,1.5) .. (1.75,1);
      \draw (1.75,2) .. controls (1.75,1.5) and (1.25,1.5) .. (1.25,1);
      \draw (1.25,1) .. controls (1.25,0.5) and (0.75,0.5) .. (0.75,0);
      \draw (1.75,3) arc(180:0:0.25);
      \draw (2.25,3) -- (2.25,0);
      \draw (1.75,3) -- (1.75,2);
      \draw (1.75,1) -- (1.75,0);
      \draw (1.75,0) arc(-180:0:0.25);
      \draw (0.75,0) -- (0.75,-1);
      \draw (1.25,0) -- (1.25,-1);
    \end{tikzpicture}
    \raisebox{1.5cm}{$=$}
    \begin{tikzpicture}[line width = 2pt]
      \fill[color=blue!20] (0,-1) rectangle  (2.5,3.5);
      \fill[morphism] (0.5,2.5) -- (1.5,2.5) arc (0:180:0.5);
      \draw (0.75,2.5) -- (0.75,1) .. controls (0.75,0.5) and  (1.25,0.5) .. (1.25,0);
      \draw (1.25,2.5) -- (1.25,2) .. controls (1.25,1.5) and (1.75,1.5) .. (1.75,1);
      \draw (0.25,3) -- (0.25,1) .. controls (0.25,0.5) and (0.75,0.5) .. (0.75,0);
      \draw (2.25,3) arc(0:90:0.25);
      \draw (2,3.25) -- (0.5,3.25);
      \draw (0.5,3.25) arc(90:180:0.25);
      \draw (2.25,3) -- (2.25,0);
      \draw (1.75,1) -- (1.75,0);
      \draw (1.75,0) arc(-180:0:0.25);
      \draw (0.75,0) -- (0.75,-1);
      \draw (1.25,0) -- (1.25,-1);
    \end{tikzpicture}
\end{center}
The first equation is a consequence of the relations for duals in
Figure~\ref{fig:cupcap}, relation~\eqref{a:slide1} in
Figure~\ref{fig:rels-slide} (with $f$ being the coevaluation) and the
condition in Figure~\ref{fig:con}.  The second equation is the
condition that two pairs of tensors commute.  The third equation is
an application of~\eqref{a:slide1} and~\eqref{a:tensor1}.

For illustration, we demonstrate the first equation in more detail:
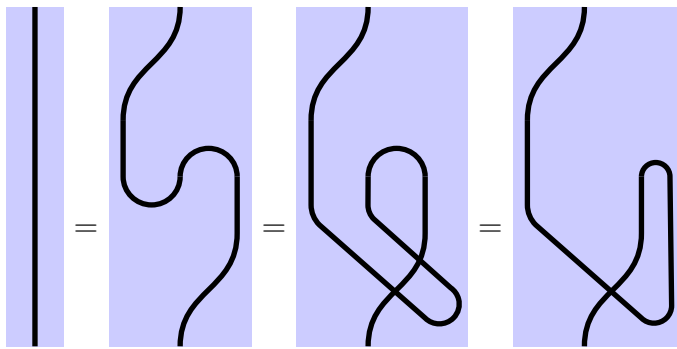
\begin{center}
  \begin{tikzpicture}[line width=2pt, scale=0.75]
    \fill[color=blue!20] (-0.5,-3) rectangle (0.5,3);
    \draw (0,-3) -- (0,3);
  \end{tikzpicture}
  \raisebox{1.5cm}{$=$}
  \begin{tikzpicture}[line width=2pt, scale=0.75]
    \fill[color=blue!20] (-1.25,-3) rectangle (1.25,3);
    \draw (0,3) .. controls (0,2) and (-1,2) .. (-1,1);
    \draw (-1, 1) -- (-1, 0) arc(180:360:0.5);
    \draw (1,0) arc(0:180:0.5);
    \draw (1, 0) -- (1, -1) .. controls (1,-2) and (0,-2) .. (0,-3);
  \end{tikzpicture}
  \raisebox{1.5cm}{$=$}
  \begin{tikzpicture}[line width=2pt, scale=0.75]
    \fill[color=blue!20] (-1.25,-3) rectangle (1.75,3);
    \draw (0,3) .. controls (0,2) and (-1,2) .. (-1,1);
    \draw (-1, 1) -- (-1, -0.5) arc(180:225:0.5) -- (1, -2.5)
    arc(225:415:0.35) -- (0.1, -0.75) arc(225:180:0.35) -- (0, 0);
    \draw (1,0) arc(0:180:0.5);
    \draw (1, 0) -- (1, -1) .. controls (1,-2) and (0,-2) .. (0,-3);
  \end{tikzpicture}
  \raisebox{1.5cm}{$=$}
  \begin{tikzpicture}[line width=2pt, scale=0.75]
    \fill[color=blue!20] (-1.25,-3) rectangle (1.75,3);
    \draw (0,3) .. controls (0,2) and (-1,2) .. (-1,1);
    \draw (-1, 1) -- (-1, -0.5) arc(180:225:0.5) -- (1, -2.5)
    arc(225:360:0.31) -- (1.5, 0) arc(0:180:0.25);
    \draw (1, 0) -- (1, -1) .. controls (1,-2) and (0,-2) .. (0,-3);
  \end{tikzpicture}
\end{center}
\end{proof}

The strategy then for constructing examples of the cyclic sieving phenomenon
is the following. Let $x$ be an object in a symmetric monoidal category
with a dual $x^\ast$. The combinatorial structure is a basis of $\Hom(I,\otimes^n x)$
which is preserved by rotation. Then the aim is to apply Theorem \ref{thm:csp}
to obtain an example of the cyclic sieving phenomenon.

\section{Diagram categories}\label{section:dc}

In this section we set up a categorical framework that contains the
Brauer category and the partition category as special cases.  In the
following sections we will employ this framework to compute the
Frobenius character of tensor powers of the defining representations
of the symplectic group and of the symmetric groups when the
dimension of the group is large enough.  A slight variation allows us
to do the same for the adjoint representations of the general linear
groups.

Before we give the definitions, we present two fundamental examples.

\subsection{The permutation category}\label{sec:perm}
The permutation category, $\dc_\cP$, has as objects the natural
numbers $\bN$ and, for two objects $r$ and $s$, morphisms
\begin{equation*}
  \Hom_{\dc_\cP}(r,s) %
  =\begin{cases}
    K\fS_r &\text{if $r=s$} \\
    \emptyset &\text{if $r\ne s$}
  \end{cases}
\end{equation*}

The composition of two morphisms is given by group multiplication.
The permutation category is a strict symmetric monoidal category, the
tensor product is the standard homomorphism
$\fS_r\times\fS_s\rightarrow \fS_{r+s}$.  There is also an
antiinvolution $\ast$ of $\fS_r$ which on the standard generators is
given by $s_i\mapsto s_{n-i}$, but the category is not pivotal, since
$x^\ast$ is not a dual of $x$.

The diagram of $\pi\in \fS_r$ is given by drawing $r$ lines in the
rectangle, connecting $r$ marked points on the top edge with $r$
marked points on the bottom edge.  For each $i\in [r]$ there is a
line with endpoints $i$ and $\pi(i)$.  Thus, an alternative way to
realise the composition of two morphisms $\pi$ and $\sigma$ is by
stacking the diagram of $\pi$ on top of the diagram of $\sigma$ so
that the two sets of $r$ endpoints on the middle edge match.

In general, there are many diagrams of a permutation, we consider two
diagrams to be equivalent if they represent the same permutation.  In
particular, we can require without loss of generality that the lines
in a diagram are drawn so that all intersections are transversal and
any two lines intersect at most once.

The category $\dc_\cP$ is generated as a monoidal category by the
single element
\begin{equation}\label{eq:perm}
\begin{tikzpicture}[line width = 2pt]
 \fill[color=blue!20] (0,0) rectangle  (2,1);
 \draw (0.5,0) .. controls (0.5,0.5) and (1.5,0.5) .. (1.5,1);
 \draw (1.5,0) .. controls (1.5,0.5) and (0.5,0.5) .. (0.5,1);
\end{tikzpicture}
\end{equation}

This example also illustrates the fundamental theorems. For any vector space
$V$ we have a symmetric monoidal functor from the permutation category
to vector spaces which on objects is given by $r\mapsto \otimes^rV$.
We will refer to this functor as an evaluation functor. On morphisms it
is given by the algebra homomorphisms
\begin{equation*}
 \ev_n\colon K\fS_r\rightarrow \End(\otimes^rV)
\end{equation*}

These homomorphisms are natural and so can be regarded as homomorphisms
\begin{equation*}
 \ev_n\colon K\fS_r\rightarrow \End_{\GL(V)}(\otimes^rV)
\end{equation*}
The first fundamental theorem is that these homomorphisms are surjective.

The second fundamental theorem describes the kernel and the image.
For $r\ge 0$, the antisymmetriser $E(r)\in K\fS_r$ is given by
\begin{equation*}
 E(r) = \frac1{r!}\sum_{\pi\in\fS_r} \varepsilon(\pi) \pi
\end{equation*}
This is the rank one central idempotent corresponding to the sign representation.

Put $n=\dim(V)$. Then $\ev_n(E(n+1))=0$ and so the evaluation homomorphisms
factor through the quotient by the ideal generated by $E(n+1)$ to give
\begin{equation*}
 \bev_n\colon K\fS_r/\left\langle E(n+1)\right\rangle\rightarrow \End_{\GL(V)}(\otimes^rV)
\end{equation*}

The second fundamental theorem is that these are isomorphisms.
Furthermore a basis of the image is given by the set
\begin{equation*}
  \{\bev_n(\pi) : \pi\in\fS_r, \lds(\pi) < n+1 \},
\end{equation*}
where $\lds(\pi)$ is the length of the longest decreasing subsequence
of $\pi$.  This can be proved using the Yang-Baxter elements
\begin{equation*}
 R_i(n)=\frac1{(n+1)}\Big( 1 + ns_i \Big)
\end{equation*}
which give a basis of $K\fS_r$ indexed by permutations.

\subsection{The Temperley-Lieb category}
The objects of the Temperley-Lieb category $\dc_\cT$ are again the
natural numbers.  For $r,s\in\bN$, the set of morphisms
$\Hom_{\dc_\cT}(r,s)$ is the free $K$-module whose basis is the set
of noncrossing perfect matchings of $[r]\amalg [s]$.  Thus
$\Hom_{\dc_\cT}(r,s)=\emptyset$ if $r+s$ is odd. If $r+s=2k$ then
$\dim \Hom_{\dc_\cT}(r,s)$ is the $k$-th Catalan number.

The diagram of a basis element $x\in \Hom_{\dc_\cT}(r,s)$, a
noncrossing perfect matching, consists of $(r+s)/2$ nonintersecting
arcs in the rectangle.  The composition of morphisms is defined on
diagrams and then extended bilinearly: let $x\in \Hom_{\dc_\cT}(r,s)$
and $y\in \Hom_{\dc_\cT}(s,t)$ be two diagrams.  Stack the diagram
for $x$ on top of the diagram for $y$ so that the two sets of $s$
endpoints on the middle edge match.  The resulting diagram possibly
has a number of loops, $c(x,y)$.  Removing these loops gives the
diagram $x\circ y$.  The composition of $x$ and $y$ is then defined
as
\begin{align*}
  x\cdot y = \delta^{c(x,y)} x\circ y.
\end{align*}

The Temperley-Lieb category is a strict pivotal category.  The tensor
product $x\otimes y$ of two diagrams is obtained by putting the
diagram of $x$ on the left of the diagram of $y$.  The dual of an
object is the object itself, the required isomorphism between
$\Hom_{\dc_\T}(x\otimes z,w)$ and $\Hom_{\dc_\T}(z,x\otimes w)$ is
given as follows.  Let $n$ be the diagram
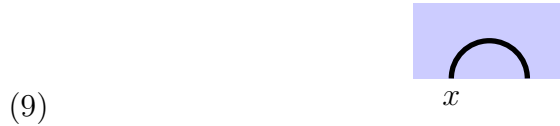
\begin{equation}\label{eq:TL-duals}
  \begin{tikzpicture}[line width = 2pt]
    \fill[color=blue!20] (0,0) rectangle  (2,1);
    \draw (0.5,0) node[anchor=north] {$x$} arc (180:0:0.5);
  \end{tikzpicture}
\end{equation}
that consists of $x$ nested arcs.  Then the map
\[
d \mapsto (n\otimes \id_z)\cdot (\id_x\otimes d)
\]
is an isomorphism.  However, $\dc_\T$ is not symmetric unless
$\delta=\pm 2$.  

For $\delta=-2$ the symmetric structure is given by the homomorphisms
$\fS_r\rightarrow \Hom(r,r)$ given on the standard generators by
$\sigma_i\mapsto 1+u_i$. This gives a representation of $\fS_r$ on
$\Hom(0,r)$ such that the long cycle is rotation. For $r=2k$, the
Frobenius character is the Schur function associated to the partition
$(2^k)$. The fake degree of this Schur function is
\begin{equation*}
 \frac{1}{[k+1]_q}\qbinom{2k}{k}
\end{equation*}
This is the example of the cyclic sieving phenomenon that
is described in \cite{MR3156682}.

Moreover, there is an antiinvolution $\ast$, which amounts to
rotating the diagram by a half turn.

The category $\dc_\T$ is generated as a monoidal category by the two
elements
\begin{equation}\label{eq:cupcap}
\begin{tikzpicture}[line width = 2pt]
 \fill[color=blue!20] (0,0) rectangle  (2,1);
 \draw (0.5,1) arc (180:360:0.5);
\end{tikzpicture} \qquad
\begin{tikzpicture}[line width = 2pt]
 \fill[color=blue!20] (0,0) rectangle  (2,1);
 \draw (0.5,0) arc (180:0:0.5);
\end{tikzpicture}
\end{equation}

\subsection{The general definition}
Let us first define an auxiliary family of categories, the
\Dfn{cobordism categories}.
\begin{defn}
  For $r,s\in\bN$ let $D(r,s)$ be a finite set, the set of
  \Dfn{diagrams}.  For $r,s,t\in\bN$, we require two maps:
  \begin{align}\label{eq:compose}
    \circ: D(r,s)\times D(s,t) &\rightarrow D(r,t)\\\notag%
    (x,y) &\mapsto x\circ y,
  \end{align}
  the \Dfn{composition} of two diagrams, and
  \begin{align}\label{eq:delta1}
    c: D(r,s)\times D(s,t) &\rightarrow \bN\\\notag%
    (x,y) &\mapsto c(x,y),
  \end{align}
  the \Dfn{number of loops} that arise when composing two diagrams.

  These maps define the \emph{cobordism category}, $\cb$, as follows:
  the set of objects of $\cb$ is $\bN$ and the morphisms between
  $r,s\in\bN$ are $\Hom_\cb(r,s)=\bN\times D(r,s)$.  The composition
  of two morphisms $(a,x)\in\Hom_\cb(r,s)$ and
  $(b,y)\in\Hom_\cb(s,t)$ is defined by
  \begin{equation}\label{eq:cobordism-compose}
    (a,x)\cdot(b,y) = \big(a+b+c(x,y), x\circ y\big).
  \end{equation}
  To ensure that $\cb$ is a category we have to insist that the maps
  $\circ$ and $c$ are such that the composition is associative, and
  that for each $p\in\bN$ there is a diagram $\id_p\in D(p,p)$ such
  that $(0,\id_p)$ is the identity morphism.

  Furthermore, we require associative maps
  \begin{align}\label{eq:delta2}
    \otimes: D(r_1,s_1)\times D(r_2,s_2) &\rightarrow
    D(r_1+r_2,s_1+s_2)\\\notag%
    (x,y) &\mapsto x\otimes y,
  \end{align}
  for all $r_1,s_1,r_2,s_2\in\bN$, such that the tensor product on
  $\cb$
  \begin{equation*}
    (a,x)\otimes (b,y) = (a+b, x\otimes y)
  \end{equation*}
  makes the cobordism category into a \Dfn{strict monoidal category}.

  Finally, we require a monoidal functor
  $\ast:\cb\to\cb^{\mathrm{op}}$ which is the identity on objects and
  satisfies $\left(x^\ast\right)^\ast = x$ for all morphisms $x$.
\end{defn}

\begin{rem}
  The composition of morphisms of $\cb$ in
  equation~\eqref{eq:cobordism-compose} is associative if and only if
  the composition $\circ$ of diagrams in equation~\eqref{eq:compose}
  is associative and for any three diagrams $x,y,z$ that can be
  composed we have
  \begin{equation*}
    c(x,y)+c(x\circ y, z) = c(x,y\circ z)+c(y, z).
  \end{equation*}
  Moreover, the morphisms $(0,\id_p)\in\Hom_{\cb}(p,p)$ are
  identities in $\cb$ if and only if the diagrams $\id_p$ are
  identities for the composition of diagrams and $c(\id_r,
  x)=c(x,\id_s) = 0$.  In this case, $c(\sigma, x) = c(x, \tau) = 0$
  whenever $(0,\sigma)$ and $(0,\tau)$ are isomorphisms.
\end{rem}

We are now ready to define \Dfn{diagram categories}.

\begin{defn}
  Let $K$ be a commutative ring.  Given a cobordism category and a
  fixed parameter $\delta\in K$, the \Dfn{diagram category} $\dc$ has
  objects $\bN$.  For $r,s\in\bN$, the set of morphisms
  $\Hom_{\dc}(r,s)$ is the free $K$-module with basis $D(r,s)$.  The
  composition in $\dc$, which we will refer to as multiplication
  henceforth, is defined as
  \begin{align}\label{eq:multiply}
    x\cdot y = \delta^{c(x,y)}x\circ y,
  \end{align}
  for basis elements $x\in D(r,s), y\in D(s,t)$, and extended
  $K$-bilinearly.

  Thus, $\dc$ is a $K$-linear (or pre-additive) category.  The tensor
  product from $\cb$ carries over to $\dc$ and makes the diagram
  category into a strict monoidal category. 
\end{defn}

\begin{defn}
  For any $r\ge 0$, the \emph{diagram algebra}, $D_r$ is the
  $K$-algebra $\Hom_{\dc}(r,r)$.
\end{defn}
The diagram algebra is a twisted semigroup algebra in the sense of
\cite{MR2301230}.

We require additionally that $\dc$ is symmetric. This symmetric structure
arises from a monoidal inclusion of the permutation category in $\cb$.
This then gives, for each $r\ge 0$, an inclusion of algebras
$K\fS_r\rightarrow D_r$.

Also, the functor $\ast$ yields an antiinvolution on $\dc$.  We
require additionally that for all $x$, $x^\ast$ is a dual of $x$,
so that $\dc$ is a strict pivotal category. Moreover in all our
examples the nondegenerate inner products are symmetric. This is
required for the cyclic sieving phenomenon.

As already hinted at above, in all our examples the sets $D(r,s)$,
the maps $\circ$ in equation \eqref{eq:compose}, $c$ in equation
\eqref{eq:delta1}, as well as the maps $\otimes$ and $\ast$ have
combinatorial meaning.

\begin{enumerate}
\item We draw a diagram for $x\in D(r,s)$ in a rectangle with two
  vertical sides and two horizontal sides, with $r$ marked points on
  the top edge and $s$ marked points on the bottom edge.  Depending
  on the specific example at hand, some of these points will be
  connected by strands or some other combinatorial structure.

\item To obtain $x\circ y$ for $x\in D(r,s)$ and $y\in D(s,t)$, we
  put the diagram for $x$ on top of the diagram for $y$ with the $s$
  marked points on the bottom edge of $x$ matching the $s$ marked
  points on the top edge of $y$. The resulting diagram will have a
  number of loops (or floating connected components).  Let $c(x,y)$
  be the number of loops. Then removing these loops gives $x\circ y$.

\item The diagram for $x\otimes y$ is obtained by putting the diagram
  for $x$ on the left of the diagram for $y$ identifying the two
  horizontal edges.

\item The diagram for $x^\ast$ is obtained by rotating the diagram
  for $x$ through a half turn.
\end{enumerate}

\begin{ex} The category of partitioned binary relations in \cite{MR3105542}
is a diagram category.
\end{ex}

\begin{ex} The diagram algebras in \cite{MR3210364} are endomorphism algebras
in a diagram category.
\end{ex}

\begin{ex} The walled Brauer-Clifford superalgebras in \cite{0443} are endomorphism algebras
in a diagram category.
\end{ex}

\subsection{Ideals}
We now define a statistic on diagrams, the propagating number, which
allows us to break up diagram categories into more manageable pieces:
\begin{defn}
  For $r,s\ge 0$ and $z\in D(r,s)$, let
  \begin{multline*}
    \pr(z) = \min \{ p\in\bN \colon\\
    z=x\circ y\text{ for some $x\in D(r,p)$ and some $y\in D(p,s)$} \}
  \end{multline*}
  be the \Dfn{propagating number} of $z$.  Let
  \begin{equation*}
    D(r,s;p) = \{ x\in D(r,s) | \pr(x)=p \}
  \end{equation*}
  be the set of diagrams with propagating number equal to $p$.
\end{defn}

For example, in the permutation category $\dc_\cP$ every element of
$\fS_r$ has propagating number $r$.  In the Temperley-Lieb category
$\dc_\cT$ the propagating number of a diagram is the number of through
strands, that is, the number of lines connecting a point on the top edge
of the rectangle to a point on the bottom edge.

\begin{lemma}\label{en:pr}
  The propagating number has the following properties:
  \begin{enumerate}
  \item $\pr(x)\le \min(r,s)$ for all $x\in D(r,s)$.
  \item $\pr(x\circ y) \le \min(\pr(x),\pr(y) )$ whenever $x\circ y$ is defined.
  \item $\pr(x\otimes y) \le \pr(x)+\pr(y)$ for all $x,y$.
  \item $\pr(x^\ast) = \pr(x)$ for all $x$.
  \end{enumerate}
\end{lemma}
\begin{proof}

  Claim (1) follows from the existence of identities, since then $x =
  \id_r\circ x=x\circ\id_s$.

  Claim (2) follows from the associativity of $\circ$: if $\pr(x)=p$
  then there are diagrams $a\in D(r,p)$ and $b\in D(p, s)$ with
  $x=a\circ b$ and therefore $x\circ y= a\circ (b\circ y)$.

  Claim (3) follows from the functoriality of the tensor product: let
  $x\in D(r,s)$ and $y\in D(m,n)$.  Suppose $\pr(x)=p$ and
  $\pr(y)=q$.  Then there exist diagrams $a\in D(r,p)$, $b\in
  D(p,s)$, $c\in D(m,q)$, $d\in D(q,n)$ with
  \begin{equation*}
    x = a\circ b\qquad \text{and}\qquad
    y = c\circ d.
  \end{equation*}
  By the functoriality of the tensor product, we have
  \begin{equation*}
    x\otimes y = (a\circ b)\otimes(c\circ d) = (a\otimes c)\circ
    (b\otimes d).
  \end{equation*}
  Therefore, $\pr(x\otimes y)\leq p+q = \pr(x)+\pr(y)$.

  Claim (4) follows from the functoriality of $\ast$: suppose that $x
  = a\circ b$, then $x^\ast = b^\ast \circ a^\ast$, so their
  propagating numbers are the same.
\end{proof}

\begin{defn}
  A \emph{$K$-ideal} $\cJ$ in a $K$-linear category $\dc$ is a
  $K$-submodule $\Hom_{\cJ}(r,s)\subseteq\Hom_{\dc}(r,s)$
  for all objects $r,s$ of $\dc$, such that
  for any morphism in the ideal the left and right composition with a
  morphism in $\dc$ is again in the ideal, whenever the composition is
  defined in $\dc$.
\end{defn}

\begin{defn}
  Let $\cJ(p)\vartriangleleft\dc$ be the $K$-ideal where
  $\Hom_{\cJ(p)}(r,s)$ is generated by the diagrams with propagating
  number at most $p$.

  Let $\cJ_r(p)$ be the ideal of the diagram algebra $D_r$ generated
  by the diagrams with propagating number at most $p$.
\end{defn}
This is an ideal by Lemma~\ref{en:pr}, Equation~(2).  Moreover, by
Lemma~\ref{en:pr}, Equation~(4), $\cJ(p)$ is closed under $\ast$,
that is, $\cJ(p)^\ast$ is an ideal in $\dc^{\mathrm{op}}$.

Let $\dc(p) = \dc/\cJ(p-1)$.  Thus, the quotient map $\dc \to
\dc/\cJ(p-1)$ sends the diagrams with propagating number strictly
less than $p$ to zero, and for $r,s\in\bN$ the diagrams in $D(r,s)$
with propagating number at least $p$ to a basis of
$\Hom_{\dc(p)}(r,s)$.

This has a natural extension.
\begin{defn}
Let $\dc$ be a monoidal category whose monoid
of objects is $\bN$. For $p\ge 0$, define an ideal $\cJ(p)$ by
\begin{equation*}
 \Hom_{\cJ(p)}(r,s) = \sum_{k\le p} \Hom_{\dc}(r,k)\otimes \Hom_{\dc}(k,s)
\subseteq \Hom_{\dc}(r,s)
\end{equation*}
\end{defn}

The properties in Lemma \ref{en:pr} extend to this setting.

\begin{ex}
Consider the cobordism category of planar trivalent graphs and the free $K$-linear
category on this cobordism category. The propagating number is defined using 
cut paths. A cut path is a path across a rectangular diagram such that each point
of intersection is transverse. The propagating number is the minimum of the number
of intersection points of a cut path.

This category has an evaluation functor to the category of representations of the
exceptional Lie group $G_2$. This maps the object $[r]$ to $\otimes^rV$ where $V$
is the seven dimensional fundamental representation. Alternatively, $G_2$ is the
automorphism group of the octonions and $V$ is the imaginary octonions.
This evaluation functor is full; generators for the kernel and the second fundamental
theorem are given in \cite{MR1403861}. The algebras $A_p$ are the Temperley-Lieb algebras.
\end{ex}

\section{Morita equivalence}
The original paper on Morita equivalence is \cite{MR0096700}. This paper and subsequent
treatments in text books, for example \cite[Chapter II]{MR0249491}, assume that algebras have units
and that the unit acts on a module as the identity. Here we present an extension of this theory.
This extension was originally given in \cite{MR687968} and is also presented in \cite{MR1610222}.
\subsection{Regular algebras}
Let $K$ be a commutative ring with a unit.  We will write $\otimes$
instead of $\otimes_K$ for brevity.

Let $A$ be an associative algebra over $K$, possibly without unit.  Let
$M$ be a left $A$-module with action $\mu_M\colon A\otimes M\rightarrow
M$ and $N$ a right $A$-module with action $\mu_N\colon N\otimes A\rightarrow N$.
Then the $K$-module homomorphism $N\otimes M \rightarrow N\otimes_A M$ is the coequaliser of
\begin{equation*}
\mu_N\otimes \id_M, \id_N\otimes \mu_M\colon N\otimes A\otimes M\rightarrow N\otimes M
\end{equation*}
Equivalently, we have the exact sequence
\begin{equation*}
 N\otimes A\otimes M\xrightarrow{\mu_N\otimes 1-1\otimes\mu_M} N\otimes M \rightarrow N\otimes_A M\rightarrow 0
\end{equation*}

For any left $A$-module $M$ the action map $\psi_M\colon A\otimes M\rightarrow M$
factors through $A\otimes_A M$ to give the action map $\psi_M\colon A\otimes_A M\rightarrow M$.

\begin{defn}
A left $A$-module $M$ is \emph{\hu} if the multiplication map $\mu\colon A\otimes_A M\rightarrow M$
is an isomorphism.  A similar notion applies to right $A$-modules.
An algebra $A$ is \emph{\hu} if it is {\hu} as a left or right $A$-module.
\end{defn}
Equivalently, we have the exact sequence
\begin{equation*}
 A\otimes A\otimes A\xrightarrow{\mu\otimes 1-1\otimes\mu} A\otimes A \xrightarrow{\mu} A\rightarrow 0
\end{equation*}

In particular, if $A$ has a unit then it is \hu.  In this case, an
$A$-module is {\hu} if and only if the unit of $A$ acts as the identity
map, as shown in \cite[Proposition 3.2]{MR1394505}.  It is sufficient
to assume that the action map is surjective. The identity is
idempotent and the range is the image of the action map.

Further examples of {\hu} algebras are $H$-unital (or homologically unital)
algebras defined in \cite{MR997314}.

If $A$ is {\hu} then the category of {\hu} $A$-modules is abelian,
see~\cite[Corollary 1.3]{MR1631300}.  However, there are examples in
\cite{MR2363140} where the category of {\hu} modules is not abelian.

\subsection{Morita contexts}
Let $A$ and $B$ be two {\hu} $K$-algebras.
The data for a Morita context is two {\hu} bimodules, $\prescript{}{A}{P}_B$ and $\prescript{}{B}{Q}_A$
together with two bimodule homomorphisms $f\colon P\otimes_B Q\rightarrow A$
and $g\colon Q\otimes_A P\rightarrow B$. These are required to satisfy the condition
that the following two diagrams commute:
\begin{equation}\label{comm}
\begin{CD}
 P\otimes_B Q\otimes_A P @>{\id_P\otimes g}>> P\otimes_B B \\
@V{f\otimes \id_P}VV @VVV \\
 A\otimes_A P @>>> P
\end{CD}\qquad
\begin{CD}
 Q\otimes_A P\otimes_B Q @>{\id_P\otimes f}>> Q\otimes_A A \\
@V{g\otimes \id_P}VV @VVV \\
 B\otimes_B Q @>>> Q
\end{CD}
\end{equation}
This Morita context is \emph{strict} if $f$ and $g$ are isomorphisms.

\begin{lemma}
Let $A$ and $B$ be two {\hu} $K$-algebras and $\prescript{}{A}P_B$ a {\hu} bimodule.
Then we have a functor $P\otimes_B - \colon M \mapsto P\otimes_B M$ from the
category of {\hu} left $B$-modules to {\hu} left $A$-modules.
\end{lemma}

\begin{proof} This is standard except for the statement that the left $A$-module
$P\otimes_B M$ is \hu. This follows from the commutativity of the square
\begin{equation*}
 \begin{CD}
  A\otimes_A (P\otimes_B M) @>>> P\otimes_B M \\
@VVV @| \\
(A\otimes_A P)\otimes_B M @>>> P\otimes_B M \\
 \end{CD}
\end{equation*}
The top edge is an isomorphism since the other three edges are known to be
isomorphisms and the diagram commutes.
\end{proof}

\begin{thm}\label{thm:inflation-equivalence}
  Given a strict Morita context then the functors $P\otimes_B - $
  and $Q\otimes_A - $ are inverse equivalences between the category
  of {\hu} $B$-modules and the category of {\hu} $A$-modules.
\end{thm}

\subsection{Inflation}\label{sec:inflation}
In this section we exhibit a strict Morita context for modules of
diagram algebras, depending on two assumptions which will be
demonstrated in each of the examples separately.  In particular, this
allows us to apply Theorem~\ref{thm:inflation-equivalence}.

Given a diagram algebra $D_r$, let $E = D_r/\cJ_r(p-1)$ be the
algebra of diagrams with propagating number at least $p$.  Let
$e_p\in E$ be an idempotent and set $A = A_p = e_p E e_p$ and $B =
B_p = E e_p E = \cJ_r(p) / \cJ_r(p-1)$.  Note that $A$ has unit $e_p$
and is therefore \hu, whereas we do not assume that $B$ has a unit.
\begin{ass}\label{ass:idempotent}
  For $0\leq p\leq r$ the idempotent $e_p$ is such that
  \begin{align*}
    \prescript{}{A}P_B &= e_p E = KD(p,r;p), \\
    \prescript{}{B}Q_A &= E e_p = KD(r,p;p),
  \end{align*}
  where $K$ is the ground field.
\end{ass}

\begin{ass}\label{ass:morita}
  For all $r,s\ge p\ge 0$, there are subsets $\bar D(r,p;p)\subseteq
  D(r,p;p)$ and $\bar D(p,r;p)\subseteq D(p,r;p)$ such that the
  composition map
  \[ 
  \circ: \bar D(r,p;p) \times D(p,p;p) \times \bar D(p,s;p)\to D(r,s;p) 
  \] 
  is a bijection
\end{ass}

\begin{thm}\label{thm:strict}
  Given Assumptions~\ref{ass:idempotent} and~\ref{ass:morita} the
  algebra $B$ and the modules $\prescript{}{A}P_B$ and
  $\prescript{}{B}Q_A$ are {\hu}.  Together with the multiplication
  maps $f:P\otimes_B Q\to A$ and $g:Q\otimes_A P\to B$ they form a
  strict Morita context.
\end{thm}
\begin{proof}
  Since the unit $e$ of $A$ acts on $P$ and $Q$ as the identity, $P$
  is a {\hu} left $A$-module and $Q$ is a {\hu} right $A$-module.  It
  is clear that the diagrams~\eqref{comm} commute.  We have to show
  that $B$ is a {\hu} algebra, $P$ is a {\hu} right $B$-module and
  $Q$ is a {\hu} left $B$-module.

  However, we first show that $f:P\otimes_B Q\to A$ is an
  isomorphism, which actually holds unconditionally.  It is clear
  that $f$ is surjective.  To show injectivity, we observe that
  $f(e\otimes e)=e$, which is the identity of $A$ and recall the
  commutativity of the diagrams~\eqref{comm}.  Assuming
  $f\left(\sum_i p_i\otimes q_i\right)=0$ we then obtain
  \begin{align*}
    \sum_i p_i\otimes q_i &= \sum_i p_i\otimes q_i\cdot f(e\otimes e) \\
    &= \sum_i p_i\otimes g(q_i\otimes e)\cdot e \\
    &= \sum_i f(p_i\otimes q_i)\otimes e = 0.
  \end{align*}
  
  Next, we show that $g:Q\otimes_A P\to B$ is an isomorphism.  Since
  the composition map is a bijection we have $\prescript{}{A}P\cong A
  \otimes K\bar D(p,r;p)$ and $Q_A\cong K\bar D(r,p;p)\otimes A$.
  Hence
  \begin{align*}
    Q\otimes_A P
    &\cong (K\bar D(r,p;p)\otimes A) \otimes_A (A \otimes K\bar D(p,r;p)) \\
    &\cong K\bar D(r,p;p)\otimes A \otimes K\bar D(p,r;p)\\
    &\cong B.
  \end{align*}
  
  Finally, using what we have shown so far, we obtain the regularity
  of $B$, $P$ and $Q$.  The map $B\otimes_B B\rightarrow B$ is an
  isomorphism because of the commutativity of the following square
  and the fact that we know that three of the four arrows are
  isomorphisms.
  \begin{equation*}
    \begin{CD}
      Q\otimes_A P\otimes_B Q\otimes_A P @>{\id_Q\otimes f \otimes\id_P} >> Q\otimes_A A\otimes_A P \\
      @V{g\otimes g}VV @VVV \\
      B\otimes_B B @>>> B \\
    \end{CD}
  \end{equation*}

  Similarly, to see that the maps $P\otimes_B B\rightarrow P$ and
  $B\otimes_B Q\rightarrow Q$ are isomorphisms we consider the
  following two commuting squares.
  \begin{equation*}
    \begin{CD}
      P\otimes_B Q\otimes_A P @>{f\otimes\id_P}>> A\otimes_A P \\
      @V{\id_P\otimes g}VV @VVV \\
      P\otimes_B B @>>> P \\
    \end{CD}
    \qquad
    \begin{CD}
      Q\otimes_A P\otimes_B Q @>{\id_Q\otimes f}>> Q\otimes_A A \\
      @V{g\otimes \id_Q}VV @VVV \\
      B\otimes_B Q @>>> Q \\
    \end{CD}
  \end{equation*}

\end{proof}

It will be useful to introduce the following notation for the functor
realising the equivalence from
Theorem~\ref{thm:inflation-equivalence}.
\begin{defn}
  The \Dfn{inflation} functor $\Inf_p^r$ is the equivalence from
  {\hu} $A_p$-modules to {\hu} $\cJ_r(p) / \cJ_r(p-1)$-modules
  \[
  V \mapsto V\otimes_{A_p} KD(p,r;p).
  \]
\end{defn}
\begin{rem}\label{rem:inflation}
  When $A_p = K\fS_p$
  the restriction of inflation to $\fS_r$ can be interpreted as an
  operation on species.  In this case we can consider $V$ as a
  species of sort $X$ and $D(p,r;p)$ as a bivariate species of sorts
  $X$ and $Y$.  Then we have
  \[
  (\Inf_p^r V)\downarrow^{D_r}_{\fS_r} = \langle V, D(p,r;p)\rangle_X.
  \]
\end{rem}

\begin{prop}\label{prop:equivalence-D_r}
  Suppose that, for $0\le p\le r$, the algebras $A_p$ are direct sums of
matrix algebras. Let $\mathcal A_p$ be the set of irreducible representations of $A_p$.
Assume that $D_r$ is semisimple and that Assumption~\ref{ass:morita} holds.  Then
$D_r$ is a direct sum of matrix algebras and
$\mathcal A = \{\Inf_p^r V: V\in \mathcal A_p, 0\leq p\leq r\}$
is a complete set of inequivalent irreducible representations of $D_r$.
\end{prop}
\begin{proof}
  The semisimplicity of $D_r$ is equivalent to the condition that for
  any $A_p$-module $U$ and any $A_q$-module $V$
  \[
  \Hom_{D_r}\left( \Inf_p^r V,\Inf_q^r U \right) = 0%
  \qquad\text{ if $p\ne q$}.
  \]
Thus the elements of $\mathcal A$ are all inequivalent.

Then we have a surjective algebra homomorphism
\begin{equation*}
 D_r \rightarrow \bigoplus_{p,V\in \mathcal A_p} \End( \Inf_p^r V )
\end{equation*}
Therefore, it is sufficient to show that these two algebras have the same
dimension.

Using Assumption~\ref{ass:morita},
  \begin{align*}
    \Inf_p^r V &\cong V\otimes_{A_p} KD(p,r;p)\\
    &\cong V\otimes_{A_p} A_p\otimes K\bar D(p,r;p)\\
    &\cong V\otimes K\bar D(p,r;p).
  \end{align*}
and so $\dim \Inf_p^r(V)=\dim V\cdot\dim K\bar D(p,r;p)$.

  Using Assumption~\ref{ass:morita} again, we obtain
  \begin{align*}
    \sum_{p=0}^r \sum_{V\in \mathcal A_p} (\dim \Inf_p^r V)^2 %
&= \sum_{p=0}^r \sum_{V\in \mathcal A_p} (\dim V)^2(\dim K\bar D(p,r;p))^2\\
    &= \sum_{p=0}^r \dim A_p(\dim K\bar D(p,r;p))^2\\
    &= \sum_{p=0}^r \dim K D(r,r;p))\\
    &= \dim D_r.
  \end{align*}
\end{proof}

\begin{rem} Assume that $A_p$ is a cellular algebra for all $p\ge 0$
as defined in \cite{MR1376244} then $D_r$ is a cellular algebra for all $r\ge 0$.
The proof is inductive with inductive step given by \cite{MR1753809}. 
\end{rem}

The class of cellularly stratified algebras is introduced in \cite{MR2658143}.
The diagram algebras are cellularly stratified since Assumption~\ref{ass:morita}
implies that they satisfy \cite[Definition 2.1]{MR2658143}.

\makeatletter
\renewcommand{\thethm}{\thesubsection.\arabic{thm}}
\@addtoreset{thm}{subsection}
\makeatother

\section{The Brauer category}\label{section:brauer}
In this section we discuss the representation theory of the symplectic groups
using the combinatorics of perfect matchings and the Brauer category.
The Brauer category was introduced for this purpose in~\cite{MR1503378}.
\subsection{Diagram category}
For $r,s\ge 0$, let $D_{\Br}(r,s)$ be the finite set of perfect
matchings on $[r]\amalg [s]$.  In particular, $D_{\Br}(r,s)=\emptyset$
if $r+s$ is odd.  An element of $D_{\Br}(r,s)$ is visualised as a set
of $(r+s)/2$ strands drawn in a rectangle, with $r$ endpoints on the
top edge of the rectangle and $s$ endpoints on the bottom edge.

The \Dfn{Brauer category} $\cb_{\Br}$ is the cobordism category with
\[
\Hom_{\cb_{\Br}}(r,s) = \bN\times D_{\Br}(r,s).
\]

The topological interpretation of this category is that it is
equivalent to the cobordism category whose objects are $0$-manifolds
and whose morphisms are $1$-manifolds.

There are inclusions $\cb_\cP\rightarrow\cb_{\Br}$ and
$\cb_\cT\rightarrow\cb_{\Br}$. Furthermore $\cb_{\Br}$ is generated
as a category by these two subcategories.  It follows that
$\cb_{\Br}$ is generated as a monoidal category by the morphisms
\eqref{eq:perm} and \eqref{eq:cupcap}.

There are three types of strands:
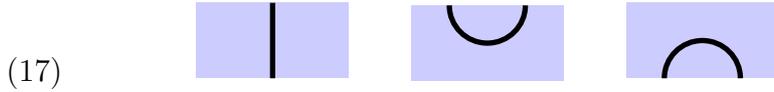
\begin{equation}\label{eq:type}
\begin{tikzpicture}[line width = 2pt]
 \fill[color=blue!20] (0,0) rectangle  (2,1);
 \draw (1,0) -- (1,1);
\end{tikzpicture} \qquad
\begin{tikzpicture}[line width = 2pt]
 \fill[color=blue!20] (0,0) rectangle  (2,1);
 \draw (0.5,1) arc (180:360:0.5);
\end{tikzpicture} \qquad
\begin{tikzpicture}[line width = 2pt]
 \fill[color=blue!20] (0,0) rectangle  (2,1);
 \draw (0.5,0) arc (180:0:0.5);
\end{tikzpicture}
\end{equation}

A strand of the first type is called a \Dfn{through strand}.  The
number of through strands in a diagram $x$ is $\pr(x)$, the
propagating number of $x$.

We denote the diagram category corresponding to the cobordism
category with $\dc_{\Br(\delta)}$.

\subsection{Frobenius characters for diagram
  algebras}\label{sec:Brauer-stable}

In this section we employ the results of Section~\ref{sec:inflation}
and elementary considerations on diagrams to compute the Frobenius
characters of the restriction of modules of the diagram algebras to
modules of the symmetric group.

In order to apply the results of Section~\ref{sec:inflation} we have
to provide a decomposition for the diagrams in $D(r,s;p)$.
\begin{lemma}
  Any Brauer diagram, $x\in D(r,s;p)$, has a unique decomposition as
  
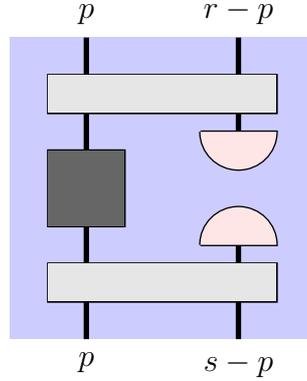
\begin{equation*}
    \begin{tikzpicture}[line width = 2pt]
      \fill[color=blue!20] (0,0) rectangle (4,4); %
      \draw (1,4) node[anchor=south] {$p$} -- (1,0)
      node[anchor=north] {$p$}; %
      \draw (3,4) node[anchor=south] {$r-p$} -- (3,2.5); %
      \draw (3,1.5) -- (3,0) node[anchor=north] {$s-p$}; %
      \draw[line width = 1pt] (0.5,1.5) rectangle (1.5,2.5); %
      \fill[color = black!60] (0.5,1.5) rectangle (1.5,2.5); %
      \draw[line width = 1pt] (0.5,3.0) rectangle (3.5,3.5); %
      \fill[color = black!10] (0.5,3.0) rectangle (3.5,3.5); %
      \draw[line width = 1pt] (0.5,0.5) rectangle (3.5,1.0); %
      \fill[color = black!10] (0.5,0.5) rectangle (3.5,1.0); %
      \draw[line width = 1pt] (2.5,2.75) -- (3.5,2.75) arc
      (0:-180:0.5); %
      \fill[color = red!10] (2.5,2.75) -- (3.5,2.75) arc
      (0:-180:0.5); %
      \draw[line width = 1pt] (2.5,1.25) -- (3.5,1.25) arc
      (0:180:0.5); %
      \fill[color = red!10] (2.5,1.25) -- (3.5,1.25) arc (0:180:0.5);
    \end{tikzpicture}
  \end{equation*}
  The two pale grey rectangles are a $(p,r-p)$-shuffle and a
  $(p,s-p)$-shuffle, the dark grey square is a permutation on $p$
  strands and the two pink half circles are perfect matchings on
  $r-p$ and on $s-p$ points.
\end{lemma}
\begin{proof}
  Take the strands of each type. The propagating strands give the
  permutation and the other two types give the two perfect
  matchings. Then the two shuffles arrange the boundary points in the
  prescribed order.
\end{proof}

\begin{defn}
  For $r\ge p\ge 0$, let $\bar D(p,r;p)$ be the set of diagrams of the form
  
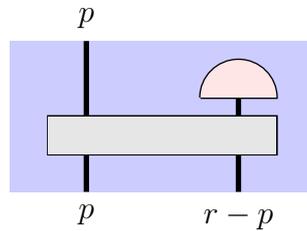
\begin{equation*}
    \begin{tikzpicture}[line width = 2pt]
      \fill[color=blue!20] (0,0) rectangle  (4,2);
      \draw (1,2) node[anchor=south] {$p$} -- (1,0) node[anchor=north] {$p$};
      \draw (3,1.5) -- (3,0) node[anchor=north] {$r-p$};
      \draw[line width = 1pt] (0.5,0.5) rectangle (3.5,1.0);
      \fill[color = black!10] (0.5,0.5) rectangle (3.5,1.0);
      \draw[line width = 1pt] (2.5,1.25) -- (3.5,1.25) arc (0:180:0.5);
      \fill[color = red!10] (2.5,1.25) -- (3.5,1.25) arc (0:180:0.5);
    \end{tikzpicture}
  \end{equation*}
  and define $\bar D(r,p;p)$ in an analogous way.
\end{defn}

\begin{cor}\label{cor:decomposition-Brauer}
  For all $r,s\ge p\ge 0$, the composition map
  \[
  \circ: \bar D(r,p;p) \times D(p,p;p) \times \bar D(p,s;p)\to
  D(r,s;p)
  \]
  is a bijection.
\end{cor}

Let $e_p\in D_r$ be the idempotent $\delta^{-(r-p)/2} \id_p\otimes u$,
where $\id_p\otimes u$ is the diagram
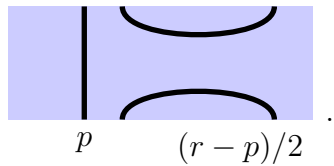
\begin{equation}
  \begin{tikzpicture}[line width = 2pt]
    \fill[color=blue!20] (0,0) rectangle  (4,1.5);
    \draw (1,0) node[anchor=north] {$p$} -- (1,1.5) {};
    \draw (1.5,0) .. controls (1.5,0.5) and (3.5,0.5) .. (3.5,0);
    \draw (4,0) node[anchor=west] {.};
    \draw (1.5,1.5) .. controls (1.5,1) and (3.5,1) .. (3.5,1.5);
    \draw (3.05,0) node[anchor=north] {$(r-p)/2$};
  \end{tikzpicture}
\end{equation}
Then, by Theorem~\ref{thm:strict} we have a strict Morita context and
inflation is an equivalence from $K\fS_p$-modules to {\hu}
$\cJ_r(p)/\cJ_r(p-1)$-modules.

We are now ready to prove the main theorem of this section:
\begin{thm}
  Let $V$ be an $\fS_p$-module.  Then $(\Inf_p^r
  V)\downarrow^{D_r}_{\fS_r}$ is isomorphic to $V\cdot M_{r-p}$,
  where $M_{r-p}$ is the species of perfect matchings of sets of
  cardinality $r-p$.
\end{thm}
\begin{proof}[First proof]
  Starting with the tensor product definition of induction, we have:
  \begin{align*}
    &\big(V \otimes K
    M_{r-p}\big)\downarrow^{\fS_r}_{\fS_p\otimes\fS_{r-p}} \\%
    &= V \otimes KM_{r-p}%
    \otimes_{K\fS_p\otimes K\fS_{r-p}} K\fS_r \\
    &\cong V\otimes_{K\fS_p} \left[ (K{\fS_p} \otimes KM_{r-p})%
      \otimes_{K\fS_p\otimes K\fS_{r-p}} K\fS_r \right] \\
    &\cong V\otimes_{K\fS_p} \left[ (K{\fS_p} \otimes KM_{r-p})%
      \otimes K[\fS_p\times \fS_{r-p} \backslash\fS_r] \right] \\
    &\cong V\otimes_{K\fS_p} KD(p,r;r)
  \end{align*}
\end{proof}
Let us now provide an alternative proof.  This proof also shows that
when $V$ is a permutation representation, the isomorphism is an
isomorphism of permutation representations.
\begin{proof}[Second proof]
  Consider the bivariate combinatorial species
  \[
  H(X\cdot Y)\cdot M(Y),
  \]
  where $H$ is the species of sets and $M$ is the species of perfect
  matchings.  Informally, this is the species that matches the set of
  labels of sort $X$ with a subset of the labels of sort $Y$, and
  puts a perfect matching on the remaining labels of sort $Y$.

  It follows immediately from
  Corollary~\ref{cor:decomposition-Brauer} (upon setting $r=p$ and
  $s=r$) that the restriction of this species to sets of sort $X$ is
  isomorphic to the $\fS_p\times\fS_r$-bimodule $KD(p,r;p)$ regarded
  as a bivariate species.

  As a consequence of the linearity of the scalar product and
  Lemma~\ref{lem:sets-reproducing-kernel} we have
  \[
  \big\langle V, H(X\cdot Y)\cdot M(Y)\big\rangle_X = V\cdot M.
  \]
  Restricting to the homogeneous component of degree $r$ we obtain
  the claim via remark~\ref{rem:inflation}.
\end{proof}

\subsection{Branching rules}\label{sec:Brauer-branching-rules}
In this section we determine branching rules for the diagram
algebras, assuming their semisimplicity.  This explains the relevance
of oscillating tableaux. This result is also proved in
\cite{MR951511}.

For $r\ge 0$ we have an inclusion $D_r\rightarrow D_{r+1}$.
Here we study the decomposition of the irreducible $D_{r+1}$-modules
restricted to a $D_r$-module.

\begin{prop}\label{prop:br} For $r\ge p\ge 0$ and $V$ any $A_p$-module,
 \begin{equation*}
   (\Inf_p^{r+1} V)\downarrow^{D_{r+1}}_{D_r} \cong
   \Inf_{p+1}^r(V\uparrow_{A_p}^{A_{p+1}}) \oplus \Inf_{p-1}^r(V\downarrow_{A_{p-1}}^{A_p})
 \end{equation*}
\end{prop}

\begin{proof}
Let $V$ be a representation of $A_p=K\fS_p$. Then we consider
$V\otimes_{A_p} D(p,r+1;p)$ restricted to $D_r$. This is given by
\begin{equation*}
 V\otimes_{A_p} \left(KD(p,r+1;p)\downarrow^{D_{r+1}}_{D_r}\right)
\end{equation*}

The first observation is that $KD(p,r+1;p)\downarrow^{D_{r+1}}_{D_r}$
has a decomposition as a $A_p$-$D_r$ bimodule into two summands
corresponding to a partition of the diagrams into two disjoint
sets. The diagrams are partitioned according to whether the last
point on the bottom row is matched with a point on the bottom row or
on the top row.

The $A_p$-$D_r$ bimodule generated by the diagrams whose last point
on the bottom row is matched with a point on the bottom row is
isomorphic to $KD(p+1,r;p+1)\downarrow^{\fS_{p+1}}_{\fS_p}$.
Informally, the isomorphism can be described as moving the last point
from the bottom row to the top row:
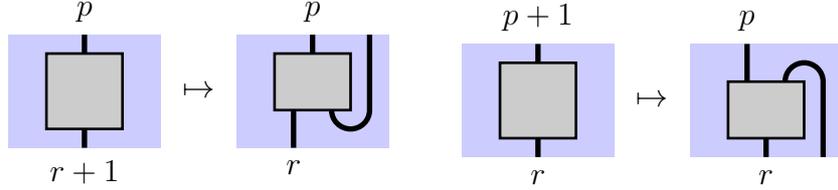
\begin{center}
  \begin{tikzpicture}[line width = 2pt]
    \fill[color=blue!20] (0,0) rectangle  (2,1.5);
    \draw (1,1.5) node[anchor=south] {$p$} -- (1,1.25) {};
    \draw (1,0) node[anchor=north] {$r+1$} -- (1,0.25) {};
    \fill[color=black!20] (0.5,0.25) rectangle (1.5,1.25);
    \draw[line width = 1pt] (0.5,0.25) rectangle (1.5,1.25);
    \draw (2.5,0.75) node {$\mapsto$};
    \fill[color=blue!20] (3,0) rectangle  (5,1.5);
    \draw (4,1.5) node[anchor=south] {$p$} -- (4,1.25) {};
    \draw (3.75,0) node[anchor=north] {$r$} -- (3.75,0.5) {};
    \fill[color=black!20] (3.5,0.5) rectangle (4.5,1.25);
    \draw[line width = 1pt] (3.5,0.5) rectangle (4.5,1.25);
    \draw (4.75,1.5) -- (4.75,0.5) arc (0:-180:0.25);
  \end{tikzpicture}
  \qquad
  \begin{tikzpicture}[line width = 2pt]
    \fill[color=blue!20] (0,0) rectangle  (2,1.5);
    \draw (1,1.5) node[anchor=south] {$p+1$} -- (1,1.25) {};
    \draw (1,0) node[anchor=north] {$r$} -- (1,0.25) {};
    \fill[color=black!20] (0.5,0.25) rectangle (1.5,1.25);
    \draw[line width = 1pt] (0.5,0.25) rectangle (1.5,1.25);
    \draw (2.5,0.75) node {$\mapsto$};
    \fill[color=blue!20] (3,0) rectangle  (5,1.5);
    \draw (3.75,1.5) node[anchor=south] {$p$} -- (3.75,1) {};
    \draw (4,0) node[anchor=north] {$r$} -- (4,0.5) {};
    \fill[color=black!20] (3.5,0.25) rectangle (4.5,1);
    \draw[line width = 1pt] (3.5,0.25) rectangle (4.5,1);
    \draw (4.75,0) -- (4.75,1) arc (0:180:0.25);
  \end{tikzpicture}
\end{center}

The $A_p$-$D_r$ bimodule generated by the diagrams whose last point
on the bottom row is matched with a point on the top row is
isomorphic to $KD(p-1,r;p-1)\uparrow^{\fS_{p}}_{\fS_{p-1}}$.  This
follows because every such diagram can be written uniquely as
follows, for some $p_1,p_2\ge 0$, $p_1+p_2=p-1$
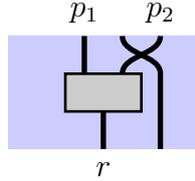
\begin{center}
  \begin{tikzpicture}[line width = 2pt]
    \fill[color=blue!20] (0,0) rectangle  (2.5,1.5);
    \draw (1,1.5) node[anchor=south] {$p_1$} -- (1,1) {};
    \draw (1.5,1.5) .. controls (1.5,1.25) and (2,1.25) .. (2,1) -- (2,0);
    \draw (2,1.5) node[anchor=south] {$p_2$} .. controls (2,1.25) and (1.5,1.25) .. (1.5,1) {};
    \draw (1.25,0) node[anchor=north] {$r$} -- (1.25,0.5) {};
    \fill[color=black!20] (0.75,0.5) rectangle (1.75,1);
    \draw[line width = 1pt] (0.75,0.5) rectangle (1.75,1);
  \end{tikzpicture}
\end{center}

Now we have
\begin{equation*}
  V\otimes_{A_p} (KD(p+1,r;p+1)\downarrow^{A_{p+1}}_{A_p}) %
  \cong (V\uparrow_{A_p}^{A_{p+1}}) \otimes_{A_{p+1}} KD(p+1,r;p+1)
\end{equation*}
as $D_r$-modules: apply associativity to the tensor product
\begin{equation*}
 V\otimes_{A_p} A_{p+1} \otimes_{A_{p+1}} KD(p+1,r;p+1)
\end{equation*}
and recall that induction is defined as
$V\uparrow_{A_p}^{A_{p+1}} = V\otimes_{A_p} A_{p+1}$.

Similarly we have
\begin{equation*}
  V\otimes_{A_p} (KD(p-1,r;p-1)\uparrow_{A_{p-1}}^{A_p}) %
  \cong (V\downarrow_{A_{p-1}}^{A_p})\otimes_{A_{p-1}} KD(p-1,r;p-1)
\end{equation*}
as $A_{p-1}$-modules by applying associativity to the tensor product
\begin{equation*}
  V\otimes_{A_p} A_{p} \otimes_{A_{p-1}} KD(p-1,r;p-1).
\end{equation*}
\end{proof}

\begin{cor}\label{cor:Brauer-branching-rules}
  Let $\lambda\vdash p$ and put
  $U(r,\lambda)=\Inf_p^r\left(S^\lambda\right)$, where $S^\lambda$ is
  the irreducible representation of $A_p$ corresponding to $\lambda$.
  Then the branching rules are given by
  \begin{equation*}
    U(r+1, \lambda)\downarrow^{D_{r+1}}_{D_r}\cong%
    \bigoplus_{\mu = \lambda\pm\square} U(r,\mu)
  \end{equation*}
  where the sum is over all partitions $\mu$ obtained from $\lambda$
  by adding or removing a cell.
\end{cor}

\begin{proof} 
  By Proposition~\ref{prop:equivalence-D_r} $\{U(r,\lambda):
  \lambda\vdash p, 0\leq p\leq r\}$ is a complete set of inequivalent
  irreducible representations of $D_r$.  From Proposition
  \ref{prop:br} it follows that
  \begin{multline*}
    \Hom_{D_r}\left( (\Inf_p^{r+1} V)\downarrow^{D_{r+1}}_{D_r}, \Inf_q^{r}U \right) \\
    \cong \Hom_{D_r}\left( \Inf_{p+1}^r(V\uparrow_{A_p}^{A_{p+1}}) \oplus%
      \Inf_{p-1}^r(V\downarrow_{A_{p-1}}^{A_p}) , \Inf_q^{r} U \right) \\
    \cong \begin{cases}
      \Hom_{A_{p+1}}\left( V\uparrow_{A_p}^{A_{p+1}},U\right) & \text{if $q=p+1$}\\
      \Hom_{A_{p-1}}\left( V\downarrow^{A_p}_{A_{p-1}},U\right) & \text{if $q=p-1$}\\
      0 & \text{otherwise},
    \end{cases}
  \end{multline*}
  where for the second isomorphism we used that inflation is a fully
  faithful functor.  Now let $U$ and $V$ be irreducible
  representations of $A_q\cong K\fS_q$ and $A_p\cong K\fS_p$
  respectively.  Then the statement follows from the branching rules
  for the symmetric group.
\end{proof}

\subsection{Fundamental theorems}
\label{sec:fundamental-theorems-Brauer}
The theorems of this section connect the representation theory of the
diagram algebras with the representation theory of the tensor powers
of the defining representation of the symplectic group by providing
an explicit isomorphism of categories.

Let $K$ be a field of characteristic zero and let $V$ be a vector
space over $K$ of dimension $2n$, equipped with a non-degenerate
skew-symmetric bilinear form $\langle\;,\;\rangle$.  The symplectic
group $\Sp(2n)$ is the group of linear transformations that preserve
this form.  However the categorical framework that we use requires that
a \emph{symmetric} form is preserved. Therefore we regard $V$ as an
\emph{odd} vector space (as a super vector space).
Let $\{b_1,\dots,b_{2n}\}$ be a basis of
$V$ and let $\{\bar b_1,\dots,\bar b_{2n}\}$ be the dual basis, that is,
$\langle \bar b_i,b_j\rangle = \delta_{i,j}$.

Let $\dc_{\Sp(2n)}$ be the diagram category corresponding to
$\cb_{\Br}$ with $\delta=-2n$ and let $\T_{\Sp(2n)}$ be the category
of invariant tensors of $\Sp(2n)$.

The connection between $\dc_{\Sp(2n)}$ and $\T_{\Sp(2n)}$ is
established by a functor we now define:
\begin{defn}
  Let $\ev_{\Sp(2n)}$ be the symmetric and pivotal monoidal functor
  $\dc_{\Sp(2n)}\rightarrow \T_{\Sp(2n)}$.  Explicitly,
  $\ev_{\Sp(2n)}$ sends the object $r\in\bN$ to $\otimes^rV$ and is
  defined on the generators by
  \begin{align*}
    &\ev_{\Sp(2n)}\left(
      \begin{tikzpicture}[scale=0.6, baseline={(0,0.2)}, line width = 2pt]
        \fill[color=blue!20] (0,0) rectangle  (2,1);
        \draw (0.5,0) .. controls (0.5,0.5) and (1.5,0.5) .. (1.5,1);
        \draw (1.5,0) .. controls (1.5,0.5) and (0.5,0.5) .. (0.5,1);
      \end{tikzpicture}
    \right) = u\otimes v\mapsto -v\otimes u\\
    &\ev_{\Sp(2n)}\left(
      \begin{tikzpicture}[scale=0.6, baseline={(0,0.2)}, line width = 2pt]
        \fill[color=blue!20] (0,0) rectangle (2,1); %
        \draw (0.5,0) arc (180:0:0.5);
      \end{tikzpicture}
    \right) = 1 \mapsto \sum_i b_i\otimes \bar b_i\\
    &\ev_{\Sp(2n)}\left(
      \begin{tikzpicture}[scale=0.6, baseline={(0,0.2)}, line width = 2pt]
        \fill[color=blue!20] (0,0) rectangle (2,1); %
        \draw (0.5,1) arc (180:360:0.5);
      \end{tikzpicture}
    \right) = u\otimes v \mapsto \langle u, v\rangle.
  \end{align*}
\end{defn}

We remark that the first equality is forced by the requirement that
$\ev_{\Sp(2n)}$ is a symmetric functor, given that $V$ is an odd
vector space.  The other two equalities are forced by the requirement
that $\ev_{\Sp(2n)}$ is pivotal.

The first fundamental theorem for the symplectic group
can now be stated as follows:
\begin{thm}\cite{MR1503378},\cite[Theorem (6.1A)]{MR1488158}
  \label{thm:FFT-SFT-Brauer}
  For all $n>0$ the functor
  $\ev_{\Sp(2n)}\colon\dc_{\Sp(2n)}\rightarrow \T_{\Sp(2n)}$ is full.
\end{thm}

In the remainder of this section we provide an explicit description
of the kernel of $\ev_{\Sp(2n)}$ as an ideal in the diagram category.
We will denote this ideal with $\Pf^{(n)}$ because of its intimate
connection to the Pfaffian.  Moreover we obtain a simple basis for
the vector space $\Hom_{\dc_{\Sp(2n)}/\Pf^{(n)}}(0, r)$.  This basis
is preserved by rotation, which we will use in
Section~\ref{sec:Brauer-CSP} to exhibit a cyclic sieving phenomenon.

The fundamental object involved is an idempotent $E(n+1)$ of the
diagram algebra $D_{n+1}$ for which we have
$\ev_{\Sp(2n)}\big(E(n+1)\big) = 0$.  This element can be
characterised as follows.

Let $\rho$ be the one dimensional representation of $D_{n+1}$ which on diagrams
is given by
\begin{equation*}
 \rho(x) = \begin{cases}
 1 & \text{if $\pr(x)=n+1$} \\
 0 & \text{if $\pr(x)<n+1$}
\end{cases}
\end{equation*}

The element $E(n+1)$ of $D_{n+1}$ is determined, up to scalar multiple, by the properties
\begin{equation}\label{eq:ch}
 xE(n+1)=\rho(x)E(n+1)=E(n+1)x
\end{equation}
It follows that $E(n+1)$ can be scaled so that it is idempotent and
these properties now determine $E(n+1)$.  It is clear that $E(n+1)$
is a central idempotent and that its rank equals one.

We will give two constructions of $E(n+1)$.  The first is as a simple
linear combination of diagrams.
\begin{defn}\label{defn:id}
 \begin{equation*}
 E(n+1) = \frac{1}{(n+1)!}\sum_{x\in D(n+1,n+1)}x
 \end{equation*}
\end{defn}

\begin{lemma} For $\delta=-2n$, the element $E(n+1)$ in Definition
  \ref{defn:id} satisfies the properties \eqref{eq:ch}.
\end{lemma}

In the proof of this lemma and also in the remainder of this section
we will make use of the following special diagrams.
\begin{defn}
  For $1\le i\le n$ define $u_i\in D(n+1,n+1)$ to consist of the
  pairs $(a,a')$ for $a\notin\{i,i+1\}$ together with the pairs
  $(i,i+1)$ and $(i',(i+1)')$ and define $s_i\in D(n+1,n+1)$ to consist
  of the pairs $(a,a')$ for $a\notin\{i,i+1\}$ together with the
  pairs $(i,(i+1)')$ and $(i',i+1)$:

  \begin{equation*}
    u_i = \raisebox{-1.25cm}
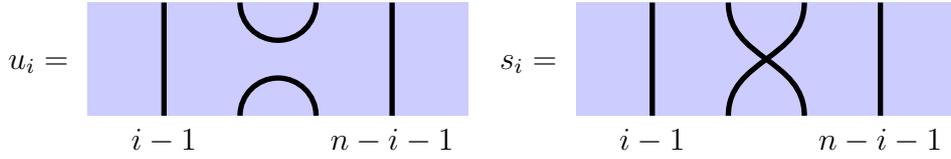
{
    \begin{tikzpicture}[line width = 2pt]
      \fill[color=blue!20] (0,0) rectangle  (5,1.5);
      \draw (1,0) node[anchor=north] {$i-1$} -- (1,1.5) {};
      \draw (2,0) arc (180:0:0.5);
      \draw (2,1.5) arc (180:360:0.5);
      \draw (4,0) node[anchor=north] {$n-i-1$} -- (4,1.5) {};
    \end{tikzpicture}}
  \quad
    s_i = \raisebox{-1.25cm}{
    \begin{tikzpicture}[line width = 2pt]
      \fill[color=blue!20] (0,0) rectangle  (5,1.5);
      \draw (1,0) node[anchor=north] {$i-1$} -- (1,1.5) {};
      \draw (2,0) .. controls (2,0.75) and (3,0.75) .. (3,1.5);
      \draw (3,0) .. controls (3,0.75) and (2,0.75) .. (2,1.5);
      \draw (4,0) node[anchor=north] {$n-i-1$} -- (4,1.5) {};
    \end{tikzpicture}}
  \end{equation*}
\end{defn}

\begin{proof} If $\pr(x)=n+1$ then $x$ acts as a permutation on the
  set of diagrams $D(n+1,n+1)$, so this case is clear.

The elements $u_i$, $1\le i\le n$ generate the ideal $\cJ_{n+1}(n)$
so it is sufficient to show that $u_i E(n+1)=0$ and $E(n+1)u_i=0$ for
$1\le i\le n$. We will now show that $u_i E(n+1)=0$. The case
$E(n+1)u_i=0$ is similar.

It is clear that $u_i E(n+1)$ is a linear combination of diagrams
which contain the pair $(i,i+1)$. Hence it is sufficient to show that
the coefficient of each of these diagrams is 0. Let $x$ be a diagram
which contains the pair $(i,i+1)$.  The set $\{y\in D(n+1,n+1): u_i y =
x\}$ has cardinality $2n$ and there is precisely one diagram $z$ with
$u_i z = \delta x$.  Therefore the coefficient of $x$ is $\delta+2n$
and so for $\delta=-2n$ this coefficient is 0.

The set of $2n$ diagrams is constructed as follows. The diagram $x$ contains
$n$ pairs other than $(i,i+1)$. For each of these pairs we construct two elements
of the set. Take the pair $(u,v)$, and replace the two pairs $(i,i+1)$ and $(u,v)$
by $(i,u)$ and $(i+1,v)$ and by $(i+1,u)$ and $(i,v)$ keeping all the remaining pairs.
\end{proof}

However it is not clear from this construction that
$\ev_{\Sp(2n)}\big(E(n+1)\big) = 0$.  We now give an alternative
construction which does make this clear.

\begin{defn} For $k\in \bZ$ and $1\le i\le n$ define $R_i(k)\in D_{n+1}$ by
\begin{equation*}
 R_i(k) = \frac1{k+1}\left( 1 + ks_i - \frac{2k}{\delta+2k-2}u_i \right)
\end{equation*}
\end{defn}

\begin{prop}\label{prop:yb} These elements satisfy
\begin{align*}
 R_i(h)R_{i+1}(h+k)R_i(k)&= R_{i+1}(k)R_i(h+k)R_{i+1}(h) \\
 R_i(h)R_j(k)&=R_j(k)R_i(h)\qquad\text{for $|i-j|>1$}
\end{align*}
\end{prop}

\begin{proof} The second relation is clear. The first relation is known as the Yang-Baxter
equation and is checked by a direct calculation. This calculation can be carried out using
generators and relations or by using a faithful representation.
\end{proof}

Next we construct an element of $D_{n+1}$ for each reduced word in the standard generators
of $\fS_{n+1}$. Let $s_{i_1}s_{i_2}\dots s_{i_l}$ be a reduced word. Then the associated
element of $D_{n+1}$ is of the form
\begin{equation*}
 R_{i_1}(k_1)R_{i_2}(k_2)\dots R_{i_l}(k_l)
\end{equation*}
The integers $k_1,\dotsc ,k_l$ are determined as follows.  First draw
the string diagram of the reduced word, where each letter of the
reduced word corresponds to a crossing of two adjacent strings.
Since the word is reduced, any two strings can cross at most once.
We number the crossings by their corresponding position in the
reduced word, and we label each string with the number of its
starting point on the top of the diagram.  For the $p$-th crossing
let the numbers of the two strings be $a_p$ and $b_p$.  Then we have
$a_p < b_p$ for a reduced diagram and we put $k_p=b_p-a_p$.

\begin{ex} The reduced word $s_3s_5s_4s_5s_3s_6s_5$ gives the element
 \begin{equation*}
  R_3(1)R_5(1)R_4(3)R_5(2)R_3(2)R_6(4)R_5(2)
 \end{equation*}
\end{ex}

The following is Matsumoto's theorem, \cite{MR0183818}.
\begin{prop}\label{prop:matsumoto}
  Two reduced words represent the same permutation if and only if
  they are related by a finite sequence of moves of the form
  $s_is_{i+1}s_i\rightarrow s_{i+1}s_is_{i+1}$, $
  s_{i+1}s_is_{i+1}\rightarrow s_is_{i+1}s_i$ and $s_is_j\rightarrow
  s_js_i$ for $|i-j|>1$.
\end{prop}

\begin{lemma} If two reduced words represent the same permutation
  then the associated elements of $D_{n+1}$ are the same.
\end{lemma}
\begin{proof} By Proposition \ref{prop:matsumoto} it is sufficient to
  show that the element is not changed when a reduced word is changed
  by one of these moves. This follows from Proposition \ref{prop:yb}.
\end{proof}

\begin{defn}
  Let $E^\delta(n+1)$ be the element of $D_{n+1}$ associated to the
  longest length permutation.
\end{defn}
\begin{lemma}
  $E^\delta(n+1)$ satisfies the properties 
  \begin{equation}\label{eq:ch2}
    xE^\delta(n+1)=\rho(x)E^\delta(n+1)=E^\delta(n+1)x.
  \end{equation}
\end{lemma}
  It follows that $E^\delta(n+1)$ is idempotent and for $\delta = -2n$
  we have $E^\delta(n+1) = E(n+1)$.
\begin{proof}
  Choose a reduced word for the longest length permutation that
  begins with $s_j$.  Then $E^\delta(n+1)$ regarded as a word in the
  $R_i(k)$ begins with $R_j(1)$, because the strings $i$ and $i+1$
  cross first.  Since $s_j R_j(1) = R_j(1)$ and $u_j R_j(1) = 0$ we
  have $xE^\delta(n+1)=\rho(x)E^\delta(n+1)$.  The second equation is
  proved similarly, noting that a reduced word for the longest length
  permutation that ends with $s_j$ leads to a word in the $R_i(k)$
  that ends with $R_j(1)$.
\end{proof}

\begin{prop}\label{prop:ker}
 \begin{equation*}
  \ev_{\Sp(2n)}(E(n+1))=0
 \end{equation*}
\end{prop}

\begin{proof}
  Because $\ev_{\Sp(2n)}$ is a functor, $\ev_{\Sp(2n)}(E(n+1))$ is an
  idempotent in $\T_{\Sp(2n)}$, too.  Moreover, $\ev_{\Sp(2n)}$ is
  pivotal and therefore preserves the trace of morphisms.  Recall
  that the (diagrammatic) trace $\tr_n$ of a diagram $\alpha$ in
  $D_n$ is defined as
  \[
  \eta_{2n} \cdot (\alpha\otimes\id_n) \cdot \eta^\ast_{2n}
  \]
  where $\eta_{2n}\in D(0,2n)$ is the diagram that consists of $n$
  nested arcs.  The following properties are easily verified.
  \begin{align*}
    \tr_{n+1}(\alpha\otimes\id_1) &= \delta \tr_n \alpha \\
    \tr_{n+1} \alpha s_n\beta &=\tr_n \alpha \beta\\
    \tr_{n+1} \alpha u_n \beta &= \tr_n \alpha \beta
  \end{align*}
  for $\alpha, \beta$ diagrams on $n$ strings.

  The rank of an idempotent is equal to its trace so it is sufficient
  to show that $\tr_{n+1} E^\delta(n+1)=0$ for $\delta=-2n$.  

  A particular choice of the reduced word for the longest length
  permutation yields
  \[
  E^\delta(n+1)=R_1(1)\dots R_{n}(n)E^\delta(n).
  \]
  We now compute the trace by expanding $R_{n}(n) = \frac{1}{n+1}(1+n
  s_n+n u_n)$:
  \[
  \tr_{n+1} E^\delta(n+1) %
  = \frac{1}{n+1}\left(\delta + n + n\right) %
  \tr_{n+1} R_1(1)\dots R_{n-1}(n-1) E^\delta(n).
  \]
  Substituting $-2n$ for $\delta$ gives $\tr_{n+1} E(n+1)=0$.
\end{proof}

We will now describe the kernel of $\ev_{\Sp(2n)}$ in terms of
diagrammatic Pfaffians.
\begin{defn}[\protect{\cite[Definition 3.4 (b)]{MR1670662}}]
  Let $\Set S$ be a $2(n+1)$-element subset of $[r]\amalg [s]$ and
  let $f$ be a perfect matching of $[r]\amalg [s]\setminus\Set S$.
  Then the \Dfn{diagrammatic Pfaffian of order $2(n+1)$ corresponding
    to $f$} is
  \[
  \Pf(f) = \sum_{\text{$s$ a perfect matching of $\Set S$}} (s\cup f).
  \]
\end{defn}

Note that for $r,s=n+1$ and $f=\emptyset$ we have, by Definition \ref{defn:id},
$E(n+1) = \frac1{(n+1)!} \Pf(\emptyset)$.

\begin{defn} 
  For $r,s\ge 0$, the subspace $\Pf^{(n)}(r,s)$ in the ideal
  $\Pf^{(n)}$ of $\dc_{\Sp(2n)}$ is spanned by the set
  \begin{multline*}
    \Pf^{(n)}(r,s) = \{ \Pf(f): \text{$f$ a perfect matching of a
      subset of $[r]\amalg [s]$}\\\text{ of cardinality $r+s-2(n+1)$
    } \}.
  \end{multline*}
\end{defn}

\begin{ex} For $n=1$ we have $2$-noncrossing diagrams which means noncrossing
or planar. The basic Pfaffian, $\Pf(\emptyset)$ is
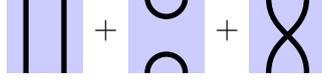
\begin{center}
\begin{tikzpicture}[line width = 2pt]
 \fill[color=blue!20] (0,0) rectangle  (1,1);
 \draw (0.25,0) -- (0.25,1);
 \draw (0.75,0) -- (0.75,1);
\end{tikzpicture}
\raisebox{0.5cm}{$+$}
\begin{tikzpicture}[line width = 2pt]
 \fill[color=blue!20] (0,0) rectangle  (1,1);
 \draw (0.75,0) arc (0:180:0.25);
 \draw (0.75,1) arc (0:-180:0.25);
\end{tikzpicture}
\raisebox{0.5cm}{$+$}
\begin{tikzpicture}[line width = 2pt]
 \fill[color=blue!20] (0,0) rectangle  (1,1);
 \draw (0.25,0) .. controls (0.25,0.5) and (0.75,0.5) .. (0.75,1);
 \draw (0.75,0) .. controls (0.75,0.5) and (0.25,0.5) .. (0.25,1);
\end{tikzpicture}
\end{center}
For $\delta=-2$, the quotient of the Brauer category by this relation is the Temperley-Lieb category.
\end{ex}

The ideal $\Pf^{(n)}$ is also characterised as the pivotal symmetric
ideal generated by $\Pf(\emptyset)$. It then follows from Proposition \ref{prop:ker}
that we have a pivotal symmetric quotient functor
\begin{equation*}
 \dc_{\Sp(2n)}\rightarrow \dc_{\Sp(2n)}/\Pf^{(n)}
\end{equation*}

\begin{defn}
  Let $\bev_{\Sp(2n)}\colon\dc_{\Sp(2n)}/\Pf^{(n)}\rightarrow
  \T_{\Sp(2n)}$ be the functor that factors $\ev_{\Sp(2n)}$ through
  the quotient.
\end{defn}

We can now state the main theorem of this section, also known as the
second fundamental theorem for the symplectic group:
\begin{thm}
  The functor $\bev_{\Sp(2n)}\colon\dc_{\Sp(2n)}/\Pf^{(n)}\rightarrow
  \T_{\Sp(2n)}$ is an isomorphism of categories.
\end{thm}
\begin{proof}
  Since $\bev_{\Sp(2n)}$ is obviously bijective on objects and full by
  the first fundamental theorem, it is sufficient to show
  \[
  \dim\Hom_{\dc_{\Sp(2n)}/{\Pf^{(n)}}}(r,s)\leq \dim
  \Hom_{\T_{\Sp(2n)}}(r,s).
  \]
  This is achieved by combining Lemma~\ref{lem:Brauer-iso},
  Lemma~\ref{lem:Brauer-basis-oscillating-tableaux},
  Theorem~\ref{thm:Brauer-basis} and Lemma~\ref{lem:Brauer-Sundaram}
  below.
\end{proof}
We first restrict our attention to the invariant tensors:
\begin{lemma}\label{lem:Brauer-iso} 
  We have isomorphisms of vector spaces
  \begin{align*}
    \Hom_{\dc_{\Sp(2n)}/{\Pf^{(n)}}}(r,s) &\cong \Hom_{\dc_{\Sp(2n)}/{\Pf^{(n)}}}(0,r+s)\\
    \Hom_{\T_{\Sp(2n)}}(r,s) &\cong \Hom_{\T_{\Sp(2n)}}(0,r+s).
  \end{align*}
\end{lemma}

Let us recall an indexing set for the basis of
$\Hom_{\T_{\Sp(2n)}}(0,r)$:
\begin{defn}
  An \Dfn{$n$-symplectic oscillating tableau} of length $r$ (and
  final shape $\emptyset$) is a sequence of partitions
  \[
  (\emptyset\!=\!\mu^0,\mu^1,\dots,\mu^r\!=\!\emptyset)
  \]
  such that the Ferrers diagrams of two consecutive partitions differ
  by exactly one cell and every partition $\mu^i$ has at most $n$
  non-zero parts.
\end{defn}

\begin{lemma}\label{lem:Brauer-basis-oscillating-tableaux}
  $\Hom_{\T_{\Sp(2n)}}(0,\otimes^r V)$ has a basis indexed by
  $n$-symplectic oscillating tableaux.
\end{lemma}
\begin{proof}
  This follows immediately from the branching rule for tensoring the
  defining representation with an irreducible representation of
  $\Sp(2n)$, see \cite[Theorem II]{MR0095209}.
\end{proof}

Next we exhibit a set of diagrams that span
$\Hom_{\dc_{\Sp(2n)}/\Pf^{(n)}}(0, r)$.  In fact, this is the key
observation.
\begin{defn}
  Let $d$ be a diagram in $D(0, r)$.  Then an \Dfn{$n$-crossing} in
  $d$ is a set of $n$ distinct strands such that every pair of
  strands crosses, that is, $d$ contains strands
  $(a_1,b_1),\dotsc,(a_n,b_n)$ with
  $a_1<a_2<\dotsb<a_n<b_1<b_2<\dotsb<b_n$.  The diagram is
  $(n+1)$-noncrossing if it contains no $(n+1)$-crossing.
\end{defn}

A bijection due to Sundaram~\cite[Lemma 8.3]{MR2941115} shows that
the cardinality of the set of $(n+1)$-noncrossing diagrams equals the
dimension of $\Hom_{\dc_{\Sp(2n)}/\Pf^{(n)}}(0, r)$:
\begin{lemma}\label{lem:Brauer-Sundaram} 
  For all $n$ and $r$ there is a bijection between the set of
  $n$-symplectic oscillating tableaux of length $r$ and the set of
  $(n+1)$-noncrossing diagrams in $D(0, r)$.
\end{lemma}

Finally, we can prove the main theorem of this section.
\begin{thm}\label{thm:Brauer-basis} 
  The set of $(n+1)$-noncrossing diagrams form a basis of
  $\Hom_{\dc_{\Sp(2n)}/\Pf^{(n)}}(0, r)$.
\end{thm}
\begin{proof}
  We only need to show that the set spans.  For each $f$ we write
  $\Pf(f)$ as a rewrite rule. The term that is singled out is the
  perfect matching of $\Set S$ in which every pair of strands
  crosses. The diagrams which cannot be simplified using these
  rewrite rules are the $(n+1)$-noncrossing diagrams.  The procedure
  terminates because the number of pairs of strands which cross
  decreases.
\end{proof}

\begin{cor}
  The diagram algebra $D_r=\Hom_{\dc_{\Sp(2n)}}(r,r)$ is semisimple
  for $n\geq r$.
\end{cor}
\begin{proof}
  When $n\geq r$ the ideal $\Pf^{(n)}(r,r)$ is empty.  Therefore
  $\ev_{\Sp(2n)}$ is an isomorphism of categories.  Since
  $\Hom_{\T_{\Sp(2n)}}(r,r)$ is semisimple, so is
  $\Hom_{\dc_{\Sp(2n)}}(r,r)$.
\end{proof}

\subsection{Frobenius characters for tensor algebras}
\label{sec:Brauer-tensor-algebra}
\begin{thm}\label{thm:Brauer-tensor-algebra}
  Let $\mu$ be a partition of length at most $n$ and let $W(\mu)$ be
  the irreducible representation of $\Sp(2n)$ with highest weight
  $(\mu_1-\mu_2,\dots,\mu_{n-1}-\mu_n,\mu_n)$.  Then for $n\geq r$
  there is a natural isomorphism
  \begin{equation*}
    \Hom_{\Sp(2n)} (W(\mu),\otimes^rV) \cong
    \Inf_p^r(S^\mu)
  \end{equation*}
  of $D_r$-modules.
\end{thm}
\begin{proof}
  The idea is to compare Corollary~\ref{cor:Brauer-branching-rules}
  with the classical branching rule for $V\otimes W(\mu)$, that is,
  \[
  W(\mu) \otimes V \cong 
  \bigoplus_{\lambda = \mu\pm\square} W(\lambda),
  \]
  the sum being over all partitions $\lambda$ obtained from $\mu$ by
  adding or removing a cell, see Littlewood~\cite{MR0095209}.
 
  Let $U(r, \mu) = \Hom_{\Sp(2n)}\big(W(\mu),\otimes^rV\big)$, that is,
  \begin{equation}
    \label{eq:Brauer-decomp1}
    \otimes^rV \cong \bigoplus_{\mu} W(\mu)\otimes U(r,\mu),    
  \end{equation}
  and let $\tilde U(r,\mu) = \Inf_p^r\left(S^\mu\right)$ and $\tilde
  W(r, \mu) = \Hom_{D_r}\big(\tilde U(r,\mu),\otimes^rV\big)$, that
  is,
  \begin{equation}
    \label{eq:Brauer-decomp2}
    \otimes^rV \cong \bigoplus_{\mu} \tilde W(r, \mu)\otimes \tilde
    U(r,\mu).    
  \end{equation}
  We then find, applying the classical branching rule and
  Corollary~\ref{cor:Brauer-branching-rules} respectively to
  $\otimes^{r+1}V \downarrow^{D_{r+1}}_{D_r}$,
  \begin{align*}
    \Hom_{D_r}\big(U(r,\mu),\otimes^{r+1}V
    \downarrow^{D_{r+1}}_{D_r}\big) %
    &\cong \bigoplus_{\lambda = \mu\pm\square} W(\lambda)\qquad\text{and}\\
    \Hom_{D_r}\big(\tilde U(r,\mu),\otimes^{r+1}V
    \downarrow^{D_{r+1}}_{D_r}\big) %
    &\cong \bigoplus_{\lambda = \mu\pm\square} \tilde W(r+1,
    \lambda).
  \end{align*}
  We now use induction on $r$.  For $r\leq 2$ and $|\mu|\leq r$ it
  can be checked directly that $U(r,\mu) \cong \tilde U(r,\mu)$ and
  $W(\mu) \cong \tilde W(r, \mu)$. %

  Suppose now that $r\geq 2$ and $U(r,\mu) \cong \tilde U(r,\mu)$ and
  $W(\mu) \cong \tilde W(r, \mu)$ for $|\mu|\leq r$.  Thus, for all
  partitions $\mu$ of $r$ we have
  \[
  \bigoplus_{\lambda = \mu\pm\square} W(\lambda) \cong %
  \bigoplus_{\lambda = \mu\pm\square} \tilde W(r+1, \lambda).
  \]
  Note that the set of partitions covered by a partition $\lambda$
  determines $\lambda$ when $|\lambda|>2$.  Furthermore, note that
  for $|\lambda| = r+1$, the modules $W(\lambda)$ and $\tilde
  W(r+1,\lambda)$ occur precisely in those equations corresponding to
  partitions $\mu$ that are covered by $\lambda$.  Therefore,
  $W(\lambda)$ and $\tilde W(r+1, \lambda)$ must be isomorphic.

  This in turn implies via Equations~\eqref{eq:Brauer-decomp1} and
  \eqref{eq:Brauer-decomp2} that $U(r+1,\lambda)$ and $\tilde U(r+1,
  \lambda)$ are isomorphic, too.
\end{proof}

\subsection{Cyclic sieving phenomenon}
\label{sec:Brauer-CSP}

We now use the results obtained so far to exhibit instances of the
cyclic sieving phenomenon.  Let $X=X(r,n)$ be the set of
$(n+1)$-noncrossing perfect matchings of $\{1,\dots,2r\}$ and let
$\rho:X\to X$ be the map that rotates a matching, that is, $\rho$ acts
on $\{1,\dots,2r\}$ as $\rho(i)=i\pmod{2r}+1$.

By Theorem~\ref{thm:Brauer-basis} $X$ is a basis for
$\Hom_{\dc_{\Sp(2n)}/\Pf^{(n)}}(0, 2r)$.  Since the functor
$\ev_{\Sp(2n)}$ is pivotal and symmetric, we can apply
Lemma~\ref{lem:rot} and Theorem~\ref{thm:csp}.

For $n\geq r$ the set $X$ coincides with $D(0,2r)$, that is, the set of
all perfect matchings of $\{1,\dots,2r\}$.  In this case the cyclic
sieving phenomenon follows directly from Corollary~\ref{cor:spr}:
\begin{cor}
  Let $X=X(r)$ be the set of perfect matchings on $\{1,\dots,2r\}$
  and let $\rho$ be the rotation map acting on $X$.  Let
  \[
  P(q) = \fd (h_{r}\circ h_2) = \sum_{\substack{\lambda\vdash
      2r\\\text{rows of even length}}} \fd s_\lambda.
  \]
  Then the triple $\big(X,\rho, P(q)\big)$ exhibits the cyclic
  sieving phenomenon.
\end{cor}
\begin{proof}
  The only observation to make is that the species of perfect
  matchings on $\{1,\dots,2r\}$ is the composition of the species of
  sets of cardinality $r$ with the species of sets of cardinality
  $2$.
\end{proof}

For $n<r$ we cannot realise $(n+1)$-noncrossing perfect matchings of
$\{1,\dots,2r\}$ as a combinatorial species, because the module
$\Hom_{\dc_{\Sp(2n)}/\Pf^{(n)}}(0, 2r)$ restricted to $\fS_{2r}$ is
not a permutation representation.  However, combinatorial
descriptions of the Frobenius character of
$\Hom_{\Sp(2n)}\big(W(\mu),\otimes^r V\big)$ were obtained in
\cite{MR2941115} (using combinatorics) and in \cite{MR933441} (using
representation theory).  For the special case of the invariant
tensors, a geometric proof can be found in Procesi~\cite[Equation
11.5.1.6]{MR2265844}.
\begin{lemma}\label{lem:it}
 \begin{equation*}
   \ch \Hom_{\T_{\Sp(2n)}}(0,2r) = \sum_{\substack{\lambda\vdash 2r\\
       \text{columns of even length}\\ \ell(\lambda)\le 2n}}
   s_{\lambda^t}
 \end{equation*}
\end{lemma}
\begin{rem}
  The partitions indexing the Schur functions appearing in the
  Frobenius character are all transposed, since we defined $V$ to be
  an \emph{odd} vector space.
\end{rem}

\begin{thm}\label{thm:pmcsp} Let $X=X(r,n)$ be the set of $(n+1)$-noncrossing perfect
  matchings on $\{1,\dots,2r\}$ and let $\rho$ be the rotation map
  acting on $X$.  Let
  \begin{equation*}
    P(q) = \sum_{\substack{\lambda\vdash 2r\\
        \text{columns of even length}\\ \ell(\lambda)\le 2n}}
    \fd s_{\lambda^t}
  \end{equation*}
  Then the triple $\big(X, \rho, P(q)\big)$ exhibits the cyclic
  sieving phenomenon.
\end{thm}

\begin{rem} Let $T$ be an oscillating tableau. The descent set of $T$ is defined in \cite{MR3226822}
and is denoted by $\mathbf{Des}(T)$. The major index of $T$ is
\begin{equation*}
 \mathbf{maj}(T) = \sum_{i\in \mathbf{Des}(T)}i
\end{equation*}
Then the main result of \cite{MR3226822} is that 
 \begin{equation*}
   P(q) = \sum_T q^{\mathbf{maj}(T)}
 \end{equation*}
where the sum is over $n$-symplectic oscillating tableaux of length $2r$ and weight $0$.
\end{rem}

The two extremes of Theorem~\ref{thm:pmcsp}, $n=1$ and $n\ge r$ are known.
Putting $n=1$, $X(r,1)$ is the set of non-crossing perfect matchings. The
corresponding cyclic sieving phenomenon can be found in \cite{MR2557880} and \cite{MR2519848}.
For $n\ge r$, $X(r,n)$ is the set of all perfect matchings. The Frobenius
character of this permutation representation was expanded by Littlewood
into Schur functions as
\begin{equation*}
h_r\circ h_2=\sum_{\substack{\lambda\vdash 2r\\ \text{rows of even length}}} s_{\lambda}.
\end{equation*}

Theorem~\ref{thm:pmcsp} can be generalised. Informally, we consider
the set of $k$-regular graphs (with loops prohibited but multiple
edges allowed) on $r$ vertices which are also $(n+1)$-noncrossing.
More precisely, consider $\{1,\dots kr\}$ as a cyclically ordered
set, partitioned into $r$ blocks of $k$ consecutive elements each.
Let $X(r,n,k)$ be the set of $(n+1)$-noncrossing perfect matchings
of $\{1,\dots,kr\}$ such that there is no pair contained in a block
and if two pairs cross then the four elements are in four distinct
blocks, see Figure~\ref{fig:symmetric-powers} for an example.
Finally, let $\rho$ be rotation by $k$ points, that is,
$\rho(i)=i+k-1\pmod{kr}+1$.  

\begin{figure}
  \centering
  \begin{tikzpicture}[line width=2pt]
    \node[draw=none,Xsize,regular polygon,regular polygon sides=8] (a) {};
    \foreach \x in {1,2,...,8} 
    \fill (a.corner \x) circle[radius=2pt];
    \node[fit=(a.corner 3)(a.corner 4), block] {};
    \node[fit=(a.corner 5)(a.corner 6), block] {};
    \node[fit=(a.corner 7)(a.corner 8), block] {};
    \node[fit=(a.corner 1)(a.corner 2), block] {};
    \draw (a.corner 1) to (a.corner 8);
    \draw (a.corner 2) to (a.corner 3);
    \draw (a.corner 4) to (a.corner 5);
    \draw (a.corner 6) to (a.corner 7);
  \end{tikzpicture}
  \quad
  \begin{tikzpicture}[line width=2pt]
    \node[draw=none,Xsize,regular polygon,regular polygon sides=8] (a) {};
    \foreach \x in {1,2,...,8}
    \fill (a.corner \x) circle[radius=2pt];
    \node[fit=(a.corner 3)(a.corner 4), block] {};
    \node[fit=(a.corner 5)(a.corner 6), block] {};
    \node[fit=(a.corner 7)(a.corner 8), block] {};
    \node[fit=(a.corner 1)(a.corner 2), block] {};
    \draw (a.corner 1) to (a.corner 8);
    \draw (a.corner 2) to (a.corner 7);
    \draw (a.corner 3) to (a.corner 6);
    \draw (a.corner 4) to (a.corner 5);
  \end{tikzpicture}  
  \quad
  \begin{tikzpicture}[line width=2pt]
    \node[draw=none,Xsize,regular polygon,regular polygon sides=8] (a) {};
    \foreach \x in {1,2,...,8}
    \fill (a.corner \x) circle[radius=2pt];
    \node[fit=(a.corner 3)(a.corner 4), block] {};
    \node[fit=(a.corner 5)(a.corner 6), block] {};
    \node[fit=(a.corner 7)(a.corner 8), block] {};
    \node[fit=(a.corner 1)(a.corner 2), block] {};
    \draw (a.corner 1) to (a.corner 4);
    \draw (a.corner 2) to (a.corner 3);
    \draw (a.corner 5) to (a.corner 8);
    \draw (a.corner 6) to (a.corner 7);
  \end{tikzpicture}  
  \quad
  \begin{tikzpicture}[line width=2pt]
    \node[draw=none,Xsize,regular polygon,regular polygon sides=8] (a) {};
    \foreach \x in {1,2,...,8}
    \fill (a.corner \x) circle[radius=2pt];
    \node[fit=(a.corner 3)(a.corner 4), block] {};
    \node[fit=(a.corner 5)(a.corner 6), block] {};
    \node[fit=(a.corner 7)(a.corner 8), block] {};
    \node[fit=(a.corner 1)(a.corner 2), block] {};
    \draw (a.corner 1) to (a.corner 6);
    \draw (a.corner 2) to (a.corner 5);
    \draw (a.corner 3) to (a.corner 8);
    \draw (a.corner 4) to (a.corner 7);
  \end{tikzpicture}  
  \quad
  \begin{tikzpicture}[line width=2pt]
    \node[draw=none,Xsize,regular polygon,regular polygon sides=8] (a) {};
    \foreach \x in {1,2,...,8}
    \fill (a.corner \x) circle[radius=2pt];
    \node[fit=(a.corner 3)(a.corner 4), block] {};
    \node[fit=(a.corner 5)(a.corner 6), block] {};
    \node[fit=(a.corner 7)(a.corner 8), block] {};
    \node[fit=(a.corner 1)(a.corner 2), block] {};
    \draw (a.corner 1) to (a.corner 8);
    \draw (a.corner 2) to (a.corner 6);
    \draw (a.corner 3) to (a.corner 7);
    \draw (a.corner 4) to (a.corner 5);
  \end{tikzpicture}  
  \quad
  \begin{tikzpicture}[line width=2pt]
    \node[draw=none,Xsize,regular polygon,regular polygon sides=8] (a) {};
    \foreach \x in {1,2,...,8}
    \fill (a.corner \x) circle[radius=2pt];
    \node[fit=(a.corner 3)(a.corner 4), block] {};
    \node[fit=(a.corner 5)(a.corner 6), block] {};
    \node[fit=(a.corner 7)(a.corner 8), block] {};
    \node[fit=(a.corner 1)(a.corner 2), block] {};
    \draw (a.corner 1) to (a.corner 5);
    \draw (a.corner 2) to (a.corner 3);
    \draw (a.corner 4) to (a.corner 8);
    \draw (a.corner 6) to (a.corner 7);
  \end{tikzpicture}  
  \caption{The six elements of $X(4,2,2)$.}
  \label{fig:symmetric-powers}
\end{figure}
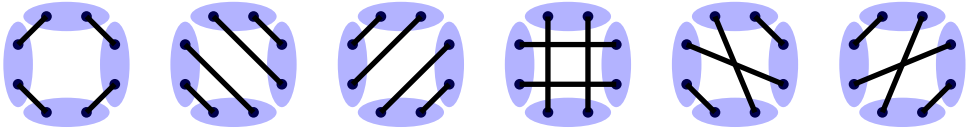

\begin{thm}
  The invariant tensors corresponding to the diagrams $X(r,n,k)$ form
  a basis of the space of $\Sp(2n)$-invariants in the $r$-th tensor
  power of the $k$-th symmetric power of the defining representation
  of $\Sp(2n)$.
\end{thm}
We prove this theorem in two lemmas.  First, we exhibit a projection
$\pi$ such that the image of $\pi$ equals $X(r,n,k)$.  Then we show
that the kernel of $\pi$ coincides with the kernel of the $r$-th
tensor power of the antisymmetriser.

\begin{lemma}
  Let $D=\Hom_{\dc_{\Sp(2n)}/\Pf^{(n)}}(0, kr)$ and let $X=X(kr,n)$
  be the set of $(n+1)$-noncrossing perfect matchings of
  $\{1,\dots,kr\}$, regarded as a basis of $D$.

  Define a linear map $\pi:D\to D$ on diagrams $\alpha\in X$ as
  follows:
  \begin{itemize}
  \item if $\alpha$ contains a pair in one block then $\pi(\alpha) =
    0$.
  \item otherwise, suppose that $\alpha$ contains $c(\alpha)$
    crossing pairs with two ends in one block.  Then $\pi(\alpha) =
    (-1)^{c(\alpha)} \tilde\alpha$, where $\tilde\alpha$ is obtained
    from $\alpha$ by untangling these crossings.
  \end{itemize}
  Then $\pi$ is idempotent and the kernel of $\pi$ is spanned by the
  diagrams which contain a pair in one block together with the sums
  $\alpha - (-1)^{c(\alpha)} \tilde\alpha$.

  Thus, the image of $\pi$, $\tilde X = X(r,n,k)$, is a basis of
  $D/\ker\pi$.
\end{lemma}
\begin{proof}
  We have
  \begin{equation*}
    \pi\left(\sum_{\alpha\in X} d_\alpha\;\alpha\right) %
    = \sum_{\alpha\in X} d_\alpha\;\pi(\alpha)
    = \sum_{\tilde\alpha\in \tilde X} %
    \sum_{\substack{\alpha\in X\\|\pi(\alpha)| = \tilde\alpha}}%
    (-1)^{c(\alpha)} d_{\alpha}\;\tilde\alpha. %
  \end{equation*}
  Suppose that $\pi(\sum_\alpha d_\alpha\;\alpha) = 0$.  Since the
  set $\tilde X$ is linearly independent in $D$, it follows that
  \[
  \sum_{\substack{\alpha\in X\\|\pi(\alpha)| = \tilde\alpha}}%
  (-1)^{c(\alpha)} d_{\alpha} = 0
  \]
  for all $\tilde\alpha\in \tilde X$.  This in turn implies that
  $\sum_{\alpha\in X} d_\alpha\;\alpha$ is a linear combination of
  diagrams that contain a pair in a block and sums $\alpha -
  (-1)^{c(\alpha)} \tilde\alpha$:
  \begin{align*}
    \sum_{\alpha\in X} d_\alpha\;\alpha%
    &=\sum_{\tilde\alpha\in\tilde X} %
    \bigg(d_{\tilde\alpha}\;\tilde\alpha%
    + \sum_{\substack{\alpha\in X\\%
        \alpha\neq\tilde\alpha\\%
        |\pi(\alpha)| = \tilde\alpha}}%
    d_{\alpha}\;\alpha\bigg)\\%
    &= \sum_{\tilde\alpha\in\tilde X} %
    \bigg(\Big(-\sum_{\substack{\alpha\in X\\%
        \alpha\neq\tilde\alpha\\%
        |\pi(\alpha)| = \tilde\alpha}}%
    (-1)^{c(\alpha)} d_{\alpha}\Big)\;%
    \tilde\alpha%
    + \sum_{\substack{\alpha\in X\\%
        \alpha\neq\tilde\alpha\\%
        |\pi(\alpha)| = \tilde\alpha}}%
    d_{\alpha}\;\alpha\bigg)\\%
    &= \sum_{\tilde\alpha\in\tilde X} %
    \sum_{\substack{\alpha\in X\\%
        \alpha\neq\tilde\alpha\\%
        |\pi(\alpha)| = \tilde\alpha}}%
    d_\alpha \bigg(\alpha - (-1)^{c(\alpha)} \tilde\alpha\bigg)%
  \end{align*}
\end{proof}

\begin{lemma}
  Let $\otimes^r \sigma:D\to D$ be the $r$-th tensor power of the
  $k$-antisymmetriser.  Then $\otimes^r \sigma$ is idempotent and the
  kernel $\otimes^r \sigma$ coincides with the kernel of $\pi$.

  Thus, $\tilde X = X(r,n,k)$, is a basis of $D/\ker\otimes^r
  \sigma$.
\end{lemma}
\begin{rem}
  We remark that the images of $\otimes^r \sigma$ and $\pi$ do not
  coincide.
\end{rem}
\begin{proof}
  The kernel of $\pi$ is a subset of the kernel of $\otimes^r
  \sigma$, since applying $\otimes^r \sigma$ to a diagram which
  contains a pair in a block gives zero, and moreover, using notation
  from the statement of the preceding lemma,
  \[
  \otimes^r \sigma(\alpha) = \otimes^r \sigma\left((-1)^{c(\alpha)}
    \tilde\alpha\right).
  \]

  Conversely, we have to show that the kernel of $\otimes^r \sigma$
  is a subset of the kernel of $\pi$.  We first note that the
  restriction of $\pi\circ \otimes^r \sigma$ to $\tilde X$ is the
  identity map, since, informally, every crossing introduced by a
  permutation $\sigma$ between two strands emanating from one block
  is untangled by $\pi$.  Thus, $\otimes^r \sigma$ is injective on
  $\tilde X$, the image of $\pi$.  Suppose now that $\otimes^r
  \sigma(\alpha)=0$.  Let $\alpha = \alpha_1 + \alpha_2$, where
  $\alpha_1$ is in the kernel of $\pi$ and $\alpha_2$ is in the image
  of $\pi$.  Then
  \begin{align*}
    0 = \otimes^r \sigma(\alpha) &= %
    \otimes^r \sigma(\alpha_1) + \otimes^r \sigma(\alpha_2)\\
    & = 0 + \otimes^r \sigma(\alpha_2),
  \end{align*}
  and injectivity of $\otimes^r \sigma$ on the image of $\pi$ implies
  $\alpha_2 = 0$.
\end{proof}

The Frobenius character of this representation is given by an
application of \cite[Theorem 1]{6111}.
\begin{lemma}
  The Frobenius character of the space of $\Sp(2n)$-invariants in the
  $r$-th tensor power of the $k$-th symmetric power of the defining
  representation of $\Sp(2n)$ is
  \begin{equation}\label{eqn:invs}
    \sum_{\substack{\lambda\vdash kr\\
        \text{columns of even length}\\ \ell(\lambda)\le 2n}}
    \big\langle  h_r\big(X\cdot e_k(Y)\big), %
    s_{\lambda^t}(Y)  \big\rangle_Y
  \end{equation}
  where $e_k$ is the $k$-th elementary symmetric function and
  $\langle\ ,\ \rangle_Y$ denotes the scalar product of symmetric
  functions with respect to the alphabet $Y$.
\end{lemma}

\begin{thm}
  Let $X=X(r,n,k)$ and let $\rho$ be the rotation map.  Define $P(q)$
  to be the fake degree of the symmetric function \eqref{eqn:invs}.
  Then $(X,\rho,P)$ exhibits the cyclic sieving phenomenon.
\end{thm}

The case $k=1$ is Theorem~\ref{thm:pmcsp} and the case $k=2$ is
related to the invariant tensors of the adjoint representation in
\cite{MR792707}. The case $n=1$ is implicit in \cite{MR1446615} and
the case $n=2$ is related to the $C_2$ webs in \cite{MR1403861}.

Instead of the $k$-th symmetric power of $V$ we could also consider
the $k$-th fundamental representation of $\Sp(2n)$.  Again using
\cite[Theorem 1]{6111} we obtain the following expression for the
Frobenius character of its tensor powers.
\begin{lemma}
  The Frobenius character of the space of $\Sp(2n)$-invariants in the
  $r$-th tensor power of the $k$-th fundamental representation of
  $\Sp(2n)$ is
  \begin{equation}\label{eqn:invf}
    \sum_{\substack{|\lambda|\le kr\\
        \text{columns of even length}\\ \ell(\lambda)\le 2n}}
    \big\langle  h_r\big(X\cdot (h_k-h_{k-2})(Y)\big), %
    s_{\lambda^t}(Y)  \big\rangle_Y
  \end{equation}
\end{lemma}
Note that in general, this space cannot have a basis invariant under
cyclic rotation.  For example, for $n=1$, $k=3$ and $r=2$ its
Frobenius character evaluates to $s_{1,1}$ with fake degree
polynomial equal to $q$, which is not a cyclic sieving polynomial.

However, let us compare the two Frobenius characters for $n>kr$.
Setting $H=1+h_1+h_2+\cdots$ we obtain for the symmetric powers
\begin{equation}\label{eqn:regs}
  \big\langle  h_r\big(X\cdot e_k(Y)\big), %
  H(h_2(Y))  \big\rangle_Y%
\end{equation}
whereas for the fundamental representation the expression becomes
\begin{equation}\label{eqn:regf}
  \big\langle  h_r\big(X\cdot (h_k-k_{k-2})(Y)\big), %
  H(h_2(Y))  \big\rangle_Y
\end{equation}

For $n>kr$, the set $X(r,n,k)$ is independent of $n$ and can be
identified with the set of $k$-regular graphs with multiple edges
allowed but loops prohibited.  This is a species whose Frobenius
character is given by equation~\eqref{eqn:regf}, see Travis~\cite{MR2698697}.

The action of the symmetric group in the first case is defined using
a sign and is therefore not a permutation representation, as
illustrated by the following example.
\begin{ex} Putting $k=2$ and $r=6$, the formula \eqref{eqn:regs} for
  invariant tensors gives
  \begin{equation}
    \frac{1}{72}(13p_{1^6} + 12p_{21^4} + 63p_{2^21^2} +
    54p_{2^3} + 4p_{31^3} - 12p_{321} + 28p_{3^2} + 18p_{41^2}
    + 36p_{42} + 36p_{6})
  \end{equation}
  and the formula \eqref{eqn:regf} for regular graphs gives
  \begin{equation*}
    \frac{1}{72}(13p_{1^6} + 24p_{21^4} + 63p_{2^21^2} +
    54p_{2^3} + 4p_{31^3} + 12p_{321} + 28p_{3^2} + 18p_{41^2}
    + 36p_{42} + 36p_{6})
 \end{equation*}
\end{ex}


\section{The partition category}\label{section:parts}
In this section we discuss the representation theory of the symmetric groups
using the combinatorics of set partitions and the partition category.
The partition category was introduced in \cite{MR1265453} and also in
\cite{MR1317365} where the relationship with the representation theory of
the symmetric groups was observed.

\subsection{Diagram category}
For $r,s\ge 0$, let $D_{\Par}(r,s)$ be the finite set of set
partitions of $[r]\amalg [s]$ and let $D_{\overline{\Par}}(r,s)$ be
the subset of set partitions with no singletons.

Let $x$ be a set partition of $[r]\amalg [s]$ and let $y$ be a set
partition of $[s]\amalg [t]$.  Consider these as equivalence
relations.  These generate an equivalence relation on $[r]\amalg
[s]\amalg [t]$.  For each block $\beta$ such that $\beta\cap
([r]\amalg [t])\ne\emptyset$, $\beta\cap ([r]\amalg [t])$ is a block
of $x\circ y$.  Define $c(x,y)$ to be the number of blocks that do
not contain any elements of $[r]\amalg [t]$.

The \Dfn{partition category} $\cb_{\Par}$ is the cobordism category
with
\[
\Hom_{\cb_{\Par}}(r,s) = \bN\times D_{\Par}(r,s).
\]
There is an inclusion $\cb_{\Br}\rightarrow \cb_{\overline{\Par}}$.
The category $\cb_{\Par}$ is generated, as a monoidal category, by
the morphisms
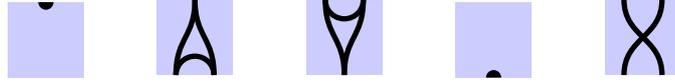
\begin{center}
\begin{tikzpicture}[line width = 2pt]
 \fill[color=blue!20] (0,0) rectangle  (1,1);
 \fill (0.4,1) -- (0.6,1) arc (0:-180:0.1);
\end{tikzpicture}
\qquad
\begin{tikzpicture}[line width = 2pt]
 \fill[color=blue!20] (0,0) rectangle  (1,1);
 \draw (0.25,0) .. controls (0.25,0.5) and (0.5,0.5) .. (0.5,1);
 \draw (0.75,0) .. controls (0.75,0.5) and (0.5,0.5) .. (0.5,1);
 \draw (0.75,0) arc (0:180:0.25);
\end{tikzpicture}
\qquad
\begin{tikzpicture}[line width = 2pt]
 \fill[color=blue!20] (0,0) rectangle  (1,1);
 \draw (0.5,0) .. controls (0.5,0.5) and (0.75,0.5) .. (0.75,1);
 \draw (0.5,0) .. controls (0.5,0.5) and (0.25,0.5) .. (0.25,1);
 \draw (0.75,1) arc (1:-180:0.25);
\end{tikzpicture}
\qquad
\begin{tikzpicture}[line width = 2pt]
 \fill[color=blue!20] (0,0) rectangle  (1,1);
 \fill (0.4,0) -- (0.6,0) arc (0:180:0.1);
\end{tikzpicture}
\qquad
\begin{tikzpicture}[line width = 2pt]
 \fill[color=blue!20] (0,0) rectangle  (1,1);
 \draw (0.25,0) .. controls (0.25,0.5) and (0.75,0.5) .. (0.75,1);
 \draw (0.75,0) .. controls (0.75,0.5) and (0.25,0.5) .. (0.25,1);
\end{tikzpicture}
\end{center}

A \Dfn{through block} is a block $\beta$ such that $\beta\cap [r]\ne
\emptyset$ and $\beta\cap [s]\ne \emptyset$.  The propagating number
of a diagram is the number of through blocks.

We denote the diagram category corresponding to the cobordism
category with $\dc_{\Par(\delta)}$.

Although we do not make use of this in this paper we remark that the
object $[1]$ in the partition category is a commutative Frobenius
algebra.
\subsection{Frobenius characters for diagram
  algebras}\label{sec:Partition-stable}

\begin{lemma}
  Any partition diagram, $x\in D(r,s;p)$, has a unique decomposition
  as
  
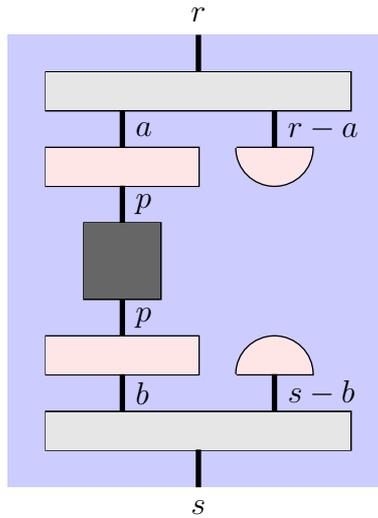
\begin{equation*}
    \begin{tikzpicture}[line width = 2pt]
      \fill[color=blue!20] (0,0) rectangle (5,6); %
      \draw (2.5,6) node[anchor=south] {$r$} -- (2.5,5.5); %
      \draw[line width = 1pt] (0.5,5.0) rectangle (4.5,5.5); %
      \fill[color = black!10] (0.5,5.0) rectangle (4.5,5.5); %
      \draw (1.5,5) -- (1.5,4.5) node[midway, anchor=west] {$a$}; %
      \draw[line width = 1pt] (0.5,4.5) rectangle (2.5,4); %
      \fill[color = red!10] (0.5,4.5) rectangle (2.5,4); %
      \draw (3.5,5) -- (3.5,4.5) node[midway, anchor=west] {$r-a$}; %
      \draw[line width = 1pt] (3,4.5) -- (4,4.5) arc (0:-180:0.5); %
      \fill[color = red!10] (3,4.5) -- (4,4.5) arc (0:-180:0.5); %
      \draw (1.5,4) -- (1.5,3.5) node[midway, anchor=west] {$p$}; %
      \draw[line width = 1pt] (1,3.5) rectangle (2,2.5); %
      \fill[color = black!60] (1,3.5) rectangle (2,2.5); %
      \draw (1.5,2) -- (1.5,2.5) node[midway, anchor=west] {$p$}; %
      \draw[line width = 1pt] (0.5,2) rectangle (2.5,1.5); %
      \fill[color = red!10] (0.5,2) rectangle (2.5,1.5); %
      \draw[line width = 1pt] (3,1.5) -- (4,1.5) arc (0:180:0.5); %
      \fill[color = red!10] (3,1.5) -- (4,1.5) arc (0:180:0.5); %
      \draw (1.5,1.5) -- (1.5,1) node[midway, anchor=west] {$b$}; %
      \draw (3.5,1.5) -- (3.5,1) node[midway, anchor=west] {$s-b$}; %
      \draw[line width = 1pt] (0.5,1) rectangle (4.5,0.5); %
      \fill[color = black!10] (0.5,1) rectangle (4.5,0.5); %
      \draw (2.5,0.5) -- (2.5,0) node[anchor=north] {$s$}; %
    \end{tikzpicture}
  \end{equation*}
  The two pale grey rectangles are a $(a,r-a)$-shuffle and a
  $(b,s-b)$-shuffle, the dark grey square is a permutation on $p$
  strands, the two pink half circles are set partitions on $r-a$ and
  on $s-b$ points and the two pink rectangles are set partitions on
  $a$ and on $b$ points into $p$ blocks.
\end{lemma}
\begin{proof}
  The upper shuffle sorts the points on the top edge according to
  whether they belong to a through block or not.  The $a$ points
  which belong to a through block form a set partition into $p$
  blocks.  These blocks are joined via a permutation to the $p$
  blocks of the set partition formed by the $b$ points on the bottom
  edge which belong to a through block.
\end{proof}

\begin{defn}
  For $r\geq p\geq 0$ let $\bar D(p,r;p)$ be the set of diagrams of
  the form
  
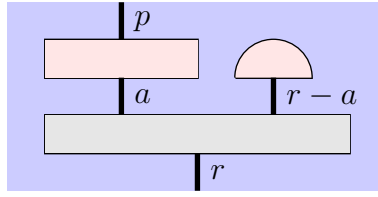
\begin{equation*}
    \begin{tikzpicture}[line width = 2pt]
      \fill[color=blue!20] (0,0) rectangle  (5,2.5);
      \draw (1.5,2) -- (1.5,2.5) node[midway, anchor=west] {$p$}; %
      \draw[line width = 1pt] (0.5,2) rectangle (2.5,1.5); %
      \fill[color = red!10] (0.5,2) rectangle (2.5,1.5); %
      \draw[line width = 1pt] (3,1.5) -- (4,1.5) arc (0:180:0.5); %
      \fill[color = red!10] (3,1.5) -- (4,1.5) arc (0:180:0.5); %
      \draw (1.5,1.5) -- (1.5,1) node[midway, anchor=west] {$a$}; %
      \draw (3.5,1.5) -- (3.5,1) node[midway, anchor=west] {$r-a$}; %
      \draw[line width = 1pt] (0.5,1) rectangle (4.5,0.5); %
      \fill[color = black!10] (0.5,1) rectangle (4.5,0.5); %
      \draw (2.5,0.5) -- (2.5,0) node[midway, anchor=west] {$r$}; %
    \end{tikzpicture}
  \end{equation*}
  and define $\bar D(r,p;p)$ in an analogous way.
\end{defn}

\begin{cor}\label{cor:decomposition-Partition}
  For all $r,s\geq p\geq 0$, the composition map
  \[
  \circ: \bar D(r,p;p) \times D(p,p;p) \times \bar D(p,s;p)\to
  D(r,s;p)
  \]
  is a bijection.
\end{cor}

Let $e_p\in D_r$ be the idempotent $\delta^{-(r-p)/2} \id_p\otimes
u$, where $\id_p\otimes u$ is the diagram
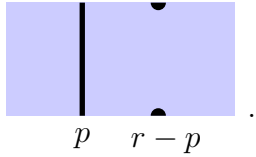
\begin{equation}
  \begin{tikzpicture}[line width = 2pt]
    \fill[color=blue!20] (0,0) rectangle  (3,1.5);
    \draw (1,0) node[anchor=north] {$p$} -- (1,1.5) {};
    \fill (1.9,1.5) -- (2.1,1.5) arc (0:-180:0.1);
    \fill (1.9,0) -- (2.1,0) node[anchor=north] {$r-p$} arc (0:180:0.1);
    \draw (3,0) node[anchor=west] {.};
  \end{tikzpicture}
\end{equation}
Then, by Theorem~\ref{thm:strict} we have a strict Morita context and
inflation is an equivalence from $K\fS_p$-modules to {\hu}
$\cJ_r(p)/\cJ_r(p-1)$-modules.

We can now prove the main theorem of this section:
\begin{thm}
  Let $V$ be an $\fS_p$-module.  Then $(\Inf_p^r
  V)\downarrow^{D_r}_{\fS_r}$ is isomorphic to the restriction of
  \[
  (V\circ H_+)\cdot P
  \]
  to the homogeneous component of degree $r$, where $H_+$ is the
  species of non-empty sets and $P$ is the species of set partitions.
\end{thm}
\begin{proof}
  Consider the bivariate combinatorial species
  \[
  H\circ\big(X\cdot H_+(Y)\big)\cdot P(Y),
  \]
  where $P=H\circ H_+$ is the species of set partitions.  It follows
  immediately from Corollary~\ref{cor:decomposition-Partition} that
  the restriction of this species to sets of sort $X$ of cardinality
  $p$ and sets of sort $Y$ of cardinality $r$ is isomorphic to the
  $\fS_r\times\fS_p$-bimodule $KD(r,p;p)$, regarded as a bivariate
  species.

  It remains to remark that we have
  \[
  \Big\langle V, H\big(X\cdot H_+(Y)\big)\cdot P(Y)\Big\rangle_X =
  (V\circ H_+)\cdot P
  \]
  as a consequence of the linearity of the scalar product and
  Lemma~\ref{lem:sets-reproducing-kernel}, where we substitute
  $H_+(Y)$ for $Y$.
\end{proof}

\subsection{Branching rules}
In this section we determine branching rules for the diagram
algebras, assuming their semisimplicity.  This explains the relevance
of vacillating tableaux.

In order to describe the branching rules we introduce intermediate
algebras, $D'_{r}\subset D_r\subset D'_{r+1}$ and describe the
branching rules for each inclusion separately. The algebra
$D'_{r+1}\subset D_{r+1}$ is the subalgebra with basis diagrams in
which the last points on the top and bottom row are in the same
block.  This gives algebras $A'_{p}\subset A_p\subset A'_{p+1}$
where the inclusion $A_p\subset A'_{p+1}$ is the identity map.

The inflation functor from $A'_{p}$-modules to
$D'_{r}$-modules is constructed using the bimodule which as a
subspace of $KD(p,r;p)$ has basis the set of diagrams in which
the last points on the top and bottom row are in the same block.
Denote this set by $D'(p,r;p)$. This inflation functor is denoted by
${}'\Inf_{p}^{r}$.

\begin{prop} For $r\ge p\ge
  0$ and $V$ any $A_p$-module, there is a short exact sequence
  \begin{equation*}
    0 \to {}'\Inf_{p}^{r}\left(V\downarrow^{A_{p}}_{A'_p}\right) %
    \to (\Inf_{p}^{r} V)\downarrow^{D_{r}}_{D'_r}\;%
    \to {}'\Inf_{p+1}^{r}\left(V\uparrow^{A'_{p+1}}_{A_p}\right)
    \to 0.
  \end{equation*}
\end{prop}

\begin{proof} Consider the $A_p$-$D'_r$ bimodule $KD(p,r;p)$.  The
  set of diagrams in $D(p,r;p)$ whose last point in the bottom row is
  in a through block generates a submodule isomorphic to
  $KD'(p,r;p)\uparrow^{A_p}_{A'_p}$.  

  The quotient of $KD(p,r;p)$ by this module is isomorphic to
  $KD'(p+1,r;p+1)$: the isomorphism is given on diagrams by adding a
  point to the top row and connecting it to the block containing the
  last point of the bottom row.  If the last point of the bottom row
  is in a through block the image under the isomorphism is zero
  because the propagating number of the resulting diagram is then
  strictly less than $p+1$.
 
  By the definition of inflation we have
  \begin{equation*}
    V \otimes_{A_p} KD'(p,r;p)\uparrow^{A_p}_{A'_p}\; \cong
    V \otimes_{A_p}A_p\otimes_{A'_p} KD'(p,r;p) \cong
    {}'\Inf_{p}^{r}\left(V\downarrow^{A_{p}}_{A'_p}\right)
  \end{equation*}
  and
  \begin{equation*}
    V \otimes_{A_p}KD'(p+1,r;p+1) \cong
    {}'\Inf_{p+1}^{r}\left(V\uparrow^{A'_{p+1}}_{A_p}\right).
  \end{equation*}
\end{proof}

\begin{cor}
  Assume that $D'_r$ is semisimple.  Then, for $r\ge p\ge 0$ and $V$
  any $A_p$-module,
  \begin{equation*}
    (\Inf_{p}^{r}V)\downarrow^{D_{r}}_{D'_r} \cong
    {}'\Inf_{p+1}^{r}\left(V\uparrow^{A'_{p+1}}_{A_p}\right) \oplus  
    {}'\Inf_{p}^{r}\left(V\downarrow^{A_{p}}_{A'_p}\right)           
  \end{equation*}
\end{cor}

\begin{prop} For $r\ge p\ge 0$ and $V$ any $A'_{p+1}$-module,
 \begin{equation*}
   ({}'\Inf_{p+1}^{r+1}V)\downarrow^{D'_{r+1}}_{D_r} \cong
\Inf_{p+1}^{r}\left(V\uparrow^{A_{p+1}}_{A'_{p+1}}\right) \oplus 
\Inf_{p}^{r}\left(V\downarrow^{A'_{p+1}}_{A_p}\right)            
 \end{equation*}
\end{prop}

\begin{proof} 
  Consider the $A'_{p+1}$-$D_r$ bimodule
  $KD'(p+1,r+1;p+1)\downarrow^{D'_{r+1}}_{D_{r}}$.  The set of
  diagrams in $D'(p+1,r+1;p+1)$ whose last points in the top and in
  the bottom row are a block generate a submodule isomorphic to
  $KD(p,r;p)$ where the isomorphism is given by removing the last
  points in the top and bottom row.

  The remaining diagrams generate a submodule isomorphic to
  $KD(p+1,r;p+1)\downarrow^{A_{p+1}}_{A'_{p+1}}$ where the
  isomorphism is given by removing the last point in the bottom row.
  Note that the block containing this point by construction contains
  another in the bottom row, so the resulting diagram has indeed
  propagating number $p+1$.

  Finally we have
  \begin{equation*}
    V\otimes_{A'_{p+1}}KD(p,r;p)\cong
    \Inf_{p}^{r}\left(V\downarrow^{A'_{p+1}}_{A_p}\right)
  \end{equation*}
  and
  \begin{equation*}
    V\otimes_{A'_{p+1}}KD(p+1,r;p+1)\cong
    \Inf_{p+1}^{r}\left(V\uparrow^{A_{p+1}}_{A'_{p+1}}\right).
  \end{equation*}
\end{proof}

\begin{cor}\label{cor:Partition-branching-rules} 
  Let $\lambda\vdash p$ and put
  \begin{equation*}
    U(r,\lambda)=\Inf_p^r\left(S^\lambda\right) \text{ and }
    U'(r,\lambda)={}'\Inf_{p+1}^r\left(S^\lambda\right).
  \end{equation*}
  Then the branching rules are given by
  \begin{align*}
    U(r, \lambda)\downarrow^{D_r}_{D'_r}&\cong%
    \bigoplus_{\substack{\mu = \lambda-\square\\%
        \text{or }\mu = \lambda}}%
    U'(r,\mu)\qquad\text{and}\\
    U'(r+1, \lambda)\downarrow^{D'_{r+1}}_{D_r}&\cong%
    \bigoplus_{\substack{\mu = \lambda+\square\\%
        \text{or }\mu = \lambda}}%
    U(r,\mu).
  \end{align*}
\end{cor}
\begin{proof}
  The proof is completely analogous to the proof of
  Proposition~\ref{cor:Brauer-branching-rules} and therefore omitted.
\end{proof}
\subsection{Fundamental theorems}
\label{sec:fundamental-theorems-Partitions}
Let $K$ be a field of characteristic zero and let $V$ be the defining
permutation representation of the symmetric group $\fS_n$.

Let $\dc_{\fS_n}$ be the diagram category corresponding to
$\cb_{\Par}$ with $\delta=n$ and let $\T_{\fS_n}$ be the category of
invariant tensors.  There is a functor from $\dc_{\fS_n}$ to
$\T_{\fS_n}$ that allows us to study the representations of
$\otimes^r V$ via diagrams and vice versa.  Let $\{v_1,\dots,v_n\}$
be the basis of $V$ as a permutation representation.

\begin{defn}[\protect{\cite{MR1317365},\cite[Equation~3.2]{MR2143201}, \cite[Chapter II
    A \S3]{MR1488158}}]
  Let $\ev_{\fS_n}\colon\dc_{\fS_n}\rightarrow \T_{\fS_n}$ be the
  symmetric and pivotal monoidal functor $\dc_{\fS_n}\to\T_{\fS_n}$.

  Explicitly, for a diagram $x$ let
  \[
  (x)^{i_1,\dots,i_r}_{i_{r+1},\dots,i_{r+s}}=%
  \begin{cases}
    1 & \text{if $i_k=i_\ell$ whenever $k$ and $\ell$}\\
    & \text{are in the same block of $x$}\\
    0 &\text{otherwise.}
  \end{cases}
  \]
  Then $\ev_{\fS_n}$ sends the object $r\in\bN$ to $V^{\otimes r}$
  and
  \begin{align*}
    \ev_{\fS_n}(x) = \Big(v_{i_1}\otimes\dots\otimes v_{i_r}
    \mapsto\sum_{1\leq i_{r+1},\dots,i_{r+s}\leq n}%
    (x)^{i_1,\dots,i_r}_{i_{r+1},\dots,i_{r+s}}%
    v_{i_{r+1}}\otimes\dots\otimes v_{i_{r+s}}\Big).
  \end{align*}

Since $\ev_{\fS_n}$ is a monoidal functor it is determined by its values
on the generators. These are:
  \begin{align*}
    &\ev_{\fS(n)}\left(
      \begin{tikzpicture}[scale=0.6, baseline={(0,0.2)}, line width = 2pt]
        \fill[color=blue!20] (0,0) rectangle  (2,1);
        \draw (0.5,0) .. controls (0.5,0.5) and (1.5,0.5) .. (1.5,1);
        \draw (1.5,0) .. controls (1.5,0.5) and (0.5,0.5) .. (0.5,1);
      \end{tikzpicture}
    \right) = u\otimes v\mapsto v\otimes u\\
    &\ev_{\fS(n)}\left(
      \begin{tikzpicture}[scale=0.6, baseline={(0,0.2)},line width = 2pt]
        \fill[color=blue!20] (0,0) rectangle  (1,1);
        \fill (0.4,1) -- (0.6,1) arc (0:-180:0.1);
      \end{tikzpicture}
    \right) = v_i \mapsto 1\\
    &\ev_{\fS(n)}\left(
      \begin{tikzpicture}[scale=0.6, baseline={(0,0.2)}, line width = 2pt]
        \fill[color=blue!20] (0,0) rectangle  (1,1);
        \draw (0.25,0) .. controls (0.25,0.5) and (0.5,0.5) .. (0.5,1);
        \draw (0.75,0) .. controls (0.75,0.5) and (0.5,0.5) .. (0.5,1);
        \draw (0.75,0) arc (0:180:0.25);
      \end{tikzpicture}
    \right) = v_i \mapsto v_i\otimes v_i,\\
    &\ev_{\fS(n)}\left(
      \begin{tikzpicture}[scale=0.6, baseline={(0,0.2)}, line width = 2pt]
        \fill[color=blue!20] (0,0) rectangle  (1,1);
        \draw (0.5,0) .. controls (0.5,0.5) and (0.75,0.5) .. (0.75,1);
        \draw (0.5,0) .. controls (0.5,0.5) and (0.25,0.5) .. (0.25,1);
        \draw (0.75,1) arc (1:-180:0.25);
      \end{tikzpicture}
    \right) = v_i\otimes v_j \mapsto \delta_{i,j} v_i,\\
    &\ev_{\fS(n)}\left(
      \begin{tikzpicture}[scale=0.6, baseline={(0,0.2)}, line width = 2pt]
        \fill[color=blue!20] (0,0) rectangle  (1,1);
        \fill (0.4,0) -- (0.6,0) arc (0:180:0.1);
      \end{tikzpicture}
    \right) = 1 \mapsto \sum_j v_j.
  \end{align*}

\end{defn}
The first fundamental theorem for the symmetric group can be stated
as follows:
\begin{thm}[\protect{\cite[Theorem~3.6]{MR2143201}}]
  For all $n>0$ the functor $\ev_{\fS_n}\colon\dc_{\fS_n}\rightarrow
  \T_{\fS_n}$ is full.
\end{thm}

A second fundamental theorem is also available.  For any $d\in
D(r,r)$ define an element $x_d\in D_r$ via
\[
d = \sum_{d'\leq d} x_{d'},
\]
where $d'\leq d$ means that every block of the partition $d$ is
contained in a block of the partition $d'$.
\begin{defn} 
  For $r,s\ge 0$, the subspace $\Par^{(n)}(r,s)$ in the ideal
  $\Par^{(n)}$ of $\dc_{\fS(n)}$ is spanned by the set
  \[
  \Par^{(n)}(r,s) = \{ x_d: \text{$d$ has more than $n$ blocks} \}.
  \]
\end{defn}

\begin{defn}
  Let $\bev_{\fS(n)}\colon\dc_{\fS(n)}/\Par^{(n)}\rightarrow
  \T_{\fS(n)}$ be the functor that factors $\ev_{\fS(n)}$ through the
  quotient.
\end{defn}

\begin{thm}[\protect{\cite[Theorem~3.6]{MR2143201}}]
    The functor
    $\bev_{\fS(n)}\colon\dc_{\fS(n)}/\Par^{(n)}\rightarrow
    \T_{\fS(n)}$ is an isomorphism of categories.
\end{thm}

\begin{thm}
  The set of partitions with less than $n+1$ blocks form a basis of
  $\Hom_{\dc_{\fS(n)}/\Par^{(n)}}(0,r)$.
\end{thm}

\begin{cor}
  The diagram algebra $D_r = \Hom_{D_{\fS_n}}(r,r)$ is semisimple for
  $n\geq 2r$.
\end{cor}
We note that it should be possible to derive a second fundamental
theorem also using the methods that succeeded in the case of the
Brauer category.  To this end we have the following conjecture.

For $r\ge 0$, let $\rho$ be the one dimensional representation of $D_{r}$
which on diagrams is given by
\begin{equation*}
  \rho(x)=\begin{cases}
    x &\text{if $\pr(x)=r$}\\
    0 &\text{if $\pr(x)<r$}
  \end{cases}
\end{equation*}
Note that $\pr(x)=r$ if and only if $x$ is a permutation.

The element $E(r)$ is determined, up to scale by the properties
\begin{equation*}
 xE(r)=\rho(x)E(r)=E(r)x
\end{equation*}
The scale is determined by the property that $E(r)$ is idempotent.
Then $E(r)$ is a rank one central idempotent. This central idempotent can be
constructed using \cite{MR1700480}. It is also, up to scale, the element 
$x_d$ associated to the element $d\in D(r,r)$ consisting of $2r$ singletons.

Similarly, the restriction of $\rho$ to $D'_r$ is a one dimensional representation
and we have a corresponding rank one central idempotent, $E'(r)$.

\begin{defn} The generators $h_i$, $s_i$ and $p_i$ are: 
  \begin{equation*}
    h_i = \raisebox{-1.25cm}{
    \begin{tikzpicture}[line width = 2pt]
      \fill[color=blue!20] (0,0) rectangle  (2.5,1.5);
      \draw (0.5,0) node[anchor=north] {$i-1$} -- (0.5,1.5) {};
      \draw (1,0) arc (180:0:0.25);
      \draw (1,1.5) arc (180:360:0.25);
      \draw (1,0) -- (1,1.5);
      \draw (1.5,0) -- (1.5,1.5);
      \draw (2,0) node[anchor=north] {$n-i-1$} -- (2,1.5) {};
    \end{tikzpicture}}
  \quad
    s_i = \raisebox{-1.25cm}{
    \begin{tikzpicture}[line width = 2pt]
      \fill[color=blue!20] (0,0) rectangle  (2.5,1.5);
      \draw (0.5,0) node[anchor=north] {$i-1$} -- (0.5,1.5) {};
      \draw (1,0) .. controls (1,0.75) and (1.5,0.75) .. (1.5,1.5);
      \draw (1.5,0) .. controls (1.5,0.75) and (1,0.75) .. (1,1.5);
      \draw (2,0) node[anchor=north] {$n-i-1$} -- (2,1.5) {};
    \end{tikzpicture}}
  \quad
    p_i = \raisebox{-1.25cm}{
    \begin{tikzpicture}[line width = 2pt]
      \fill[color=blue!20] (0,0) rectangle  (2,1.5);
      \draw (0.5,0) node[anchor=north] {$i-1$} -- (0.5,1.5) {};
      \fill (0.9,1.5) -- (1.1,1.5) arc (0:-180:0.1);
      \fill (0.9,0) -- (1.1,0) arc (0:180:0.1);
      \draw (1.5,0) node[anchor=north] {$n-i$} -- (1.5,1.5) {};
    \end{tikzpicture}}
  \end{equation*}
\end{defn}

\begin{conj} The idempotents $E(r)\in D_r$ and $E'(r)\in D'_r$ are
constructed recursively by $E'(1)=1$ and
\begin{align*}
 E(r) &= \frac1r E'(r)\left[ 1+(r-1)s_{r-1} -\frac1{\delta-2r+2}p_r \right] E'(r) \\
 E'(r+1) &= E(r)\left[ 1 - r\left(\frac{\delta-2r+2}{\delta-2r+1}\right)h_r \right] E(r)
\end{align*}
\end{conj}

\begin{ex}
 \begin{align*}
  E(1) &= (1-\frac1\delta p_1) \\
  E'(2) &= E(1)\left( 1 - \frac{\delta}{\delta-1} h_1 \right) E(1) \\
  E(2) &= \frac12 E'(2)\left( 1+s_1-\frac1{\delta-2}p_2 \right) E'(2) \\
  E'(3) &= E(2)\left( 1 - 2\frac{\delta-2}{\delta-3} h_2 \right) E(2) \\
  E(3) &= \frac13 E'(3)\left( 1+2s_2-\frac1{\delta-4}p_3 \right) E'(3)
 \end{align*}
\end{ex}

\begin{conj}
  The ideal $\Par^{(n)}$ of $\dc_{\fS(n)}$ is the pivotal symmetric ideal
  generated by $E(n+1)$.
\end{conj}

\subsection{Frobenius characters for tensor algebras}
\label{sec:Partition-tensor-algebra}
\begin{thm}\label{thm:Partition-tensor-algebra}
  Let $\mu$ be a partition of $n$ and let $S^\mu$ be the irreducible
  representation of $\fS_n$ corresponding to $\mu$.  Then for $n\geq
  2r$ there is a natural isomorphism
  \begin{equation*}
    \Hom_{\fS_n} (S^\mu,\otimes^rV) \cong
    \Inf_p^r\left(S^{\tilde\mu}\right)
  \end{equation*}
  of $D_r$-modules, where $\tilde\mu$ is the partition $\mu$ with the
  first part removed.
\end{thm}
\begin{proof}
  The proof is very similar to the proof of
  Theorem~\ref{thm:Brauer-tensor-algebra}.  Recall that $V\cong
  S^{(n)} \oplus S^{(n-1,1)}$ and the classical branching rule is
  \[
  S^\mu \otimes V \cong%
  \bigoplus_{\nu = \mu-\square}\bigoplus_{\lambda = \nu+\square}
  S^\lambda.
  \]

  Let $U(r, \mu) = \Hom_{\fS_n}\big(S^\mu,\otimes^rV\big)$, that is,
  \begin{equation}
    \label{eq:Partition-decomp1}
    \otimes^rV \cong \bigoplus_{\mu} S^\mu \otimes U(r,\mu),    
  \end{equation}
  let $\tilde U(r,\mu) = \Inf_p^r\left(S^{\tilde\mu}\right)$ (where
  $\tilde\mu\vdash p$ and let $\tilde W(r, \mu) =
  \Hom_{D_r}\big(\tilde U(r,\mu),\otimes^rV\big)$, that is,
  \begin{equation}
    \label{eq:Partition-decomp2}
    \otimes^rV \cong \bigoplus_{\mu} \tilde W(r, \mu)\otimes \tilde
    U(r,\mu).    
  \end{equation}
  We then find, using the classical branching rule and
  Corollary~\ref{cor:Partition-branching-rules} respectively,
  \begin{align*}
    \Hom_{D_r}\big(U(r,\mu),\otimes^{r+1}V
    \downarrow^{D_{r+1}}_{D_r}\big) %
    &\cong%
    \bigoplus_{\nu = \mu-\square}\bigoplus_{\lambda = \nu+\square}
    S^\lambda%
    \qquad\text{and}\\
    \Hom_{D_r}\big(\tilde U(r,\mu),\otimes^{r+1}V
    \downarrow^{D_{r+1}}_{D_r}\big) %
    &\cong%
    \bigoplus_{\nu = \mu-\square}\bigoplus_{\lambda = \nu+\square}
    \tilde W(r+1, \lambda).
  \end{align*}
  The proof now proceeds exactly as the proof of
  Theorem~\ref{thm:Brauer-tensor-algebra}: for $r\leq 2$ we have to
  check directly that $U(r,\mu) \cong \tilde U(r,\mu)$ and $S^\mu
  \cong \tilde W(r, \mu)$, whereas for $r>2$ we use induction.  To
  this end, note that for a given $\lambda$ the multiset of
  partitions obtained by first removing a cell from $\lambda$ and
  then adding a cell uniquely determines $\lambda$, see for example
  \cite[Corollary~4.2]{MR941434}.
\end{proof}

\subsection{Cyclic sieving phenomenon}
\label{sec:Partition-CSP}

As in Section~\ref{sec:Brauer-CSP} we will now exhibit an instance of
the cyclic sieving phenomenon.  Let $X$ be the set of set
partitions $\{1,\dots,r\}$ into at most $n$ blocks, and let
$\rho:X\to X$ be the rotation map, that is, $\rho$ acts on
$\{1,\dots,r\}$ as $\rho(i)=i\pmod r+1$.

In contrast to the situation with $(n+1)$-noncrossing perfect
matchings, the set partitions into at most $n$ blocks can be
understood as a combinatorial species.

Let $H_n=\sum_{k=0}^n h_k$ be the species of sets of cardinality
at most $n$ and let $(H_n \circ H_+)_r$ be the homogeneous part of degree $r$.
\begin{thm}\label{thm:partitions_csp} Let $X=X(r,n)$ be the set of set partitions on
  $\{1,\dots,r\}$ into at most $n$ blocks and let $\rho$ be the
  rotation map acting on $X$.  Let
  \begin{equation*}
    P(q) = \fd \ch (H_n\circ H_+)_r,
  \end{equation*}
Then the triple $\big(X, \rho, P(q)\big)$ exhibits the cyclic sieving phenomenon.
\end{thm}

For completeness, we remark that as in the case of the symplectic
group, the Frobenius characters of the isotypical components are known:
\begin{thm}[\protect{\cite[Theorem 5.1]{MR1272068} and \cite[Exercise
    7.74]{MR1676282}}]
  For any $\mu\vdash n$,
  \begin{equation*}
    \sum_{r\ge 0}\ch \Hom_{\fS_n} (S^\mu,\otimes^rV) = s_\mu\circ H.
  \end{equation*}
\end{thm}

As in Section~\ref{sec:Brauer-CSP}, we can generalise
Theorem~\ref{thm:partitions_csp}.  Let $X = X(r,n,k)$ be the set of
multiset partitions of $\{1,\dots,1,2,\dots,2,\dots,r,\dots,r\}$,
where each label occurs $k$ times, into at most $n$ blocks.  Let
$\rho$ be the rotation action, that is, $\rho$ acts on the labels as
$\rho(i)=i\pmod{r}+1$.  The symmetric group $\fS_r$ acts on this set
by permuting the labels.  The corresponding Frobenius character can
be obtained using \cite[Theorem 1]{6111} or alternatively with the
calculus of species and equals
\[
\big\langle h_r\big(X\cdot h_k(Y)\big), %
(H_n\circ H_+)_{kr}(Y) \big\rangle_Y.
\]
Let $P=P(r,n,k)$ be the fake degree of this symmetric function.
Then $(X,\rho,P)$ exhibits the cyclic sieving phenomenon.

A more interesting generalisation of Theorem~\ref{thm:partitions_csp}
is obtained by considering the exterior powers of the defining
representation of the symmetric group.  In this case, $X=X(r,n,k)$ is
the set of multiset partitions of
$\{1,\dots,1,2,\dots,2,\dots,r,\dots,r\}$, each label occurring $k$
times, into at most $n$ blocks, with the additional requirement that
the blocks are proper sets and blocks of odd cardinality have
multiplicity one.  The corresponding Frobenius character can be
obtained using \cite[Theorem 1]{6111}
\[
\big\langle h_r\big(X\cdot e_k(Y)\big), %
(H_n\circ H_+)_{kr}(Y) \big\rangle_Y.
\]
When $P=P(r,n,k)$ is the fake degree of this symmetric function and
$\rho$ is the rotation action as above $(X,\rho,P)$ exhibits the
cyclic sieving phenomenon.

However, in general this Frobenius character does not come from a
permutation representation.  For example, for $n = 3$, $k = 2$, $r =
3$ it is $2/3p_{1, 1, 1} + 1/3p_3$.

\section{The directed Brauer category}
In this section we discuss the invariant theory of the adjoint representation
of the general linear groups using the combinatorics of directed perfect matchings,
or bijections, and the directed Brauer category; alternatively, the
combinatorics of walled Brauer diagrams and the walled Brauer
category. This invariant theory was developed in
\cite{MR768993}, \cite{MR792707}, \cite{MR899903}, \cite{MR1280591}.
The use of diagrams in this context emerged from the skein relation approach
to the HOMFLYPT knot polynomial.

\subsection{Diagram category}

For $r,s\ge 0$, let $D_{\vec\Br}(r,s)$ be the set of directed perfect
matchings of $[r]\amalg [s]$, that is, the set of partitions of
$[r]\amalg [s]$ into ordered pairs.  In terms of diagrams an element
of $D_{\vec\Br}(r,s)$ is a Brauer diagram in $D_{\Br}(r,s)$ together
with an orientation of each strand.

The directed cobordism category $\cb_{\vec\Br}$ has as objects the
set of words in the alphabet $\{+,-\}$, that is, $\bN\ast\bN$.  The set
of morphisms is $\bN\times\bN\times D_{\vec\Br}(r,s)$, since there
are two kinds of oriented loops. For an element $x\in
\bN\times\bN\times D_{\vec\Br}(r,s)$ the domain of $x$ is the word
$u_1,\dotsc,u_r$ with
\begin{equation*}
  u_i= \begin{cases}
    + &\text{if $i$ is the initial point of a pair}\\
    - &\text{if $i$ is the final point of a pair}\end{cases}
\end{equation*}
and the codomain of $x$ is the word $v_1,\dotsc,v_s$ with
\begin{equation*}
  v_i = \begin{cases}
    + &\text{if $i$ is the final point of a pair}\\
    - &\text{if $i$ is the initial point of a pair.}\end{cases}
\end{equation*}
Composition is defined in the same way as for the Brauer category.

The topological interpretation of this category is that it is
equivalent to the cobordism category whose objects are oriented $0$-manifolds
and whose morphisms are oriented $1$-manifolds.

We denote with $\dc_{\vec\Br(\delta)}$ the corresponding diagram
category. Here both oriented loops give a factor of $\delta$.

Then $\dc_{\vec\Br(\delta)}$ is generated as a monoidal category by
the diagrams obtained from \eqref{eq:perm} and \eqref{eq:cupcap}
by putting orientations on the strands.

There are now six types of strand since each type of strand in
\eqref{eq:type} can be directed in two possible ways.  In particular
there are two types of through strand.  The two types of through
strand determine two propagating numbers each of which satisfies the
conditions in Lemma~\ref{en:pr}.

We will use boldface letters to denote pairs of numbers. This means
that objects will typically be denoted by $\mathbf r = (r_1,r_2)$ and
propagating numbers by $\mathbf p = (p_1,p_2)$.

Any object of the directed Brauer category is isomorphic to $+^n -^m$
for a unique pair $(n,m)$ with $n,m\in\bN$.  The full subcategory on
these objects is called the walled Brauer category.  By construction,
the inclusion of the walled Brauer category in the directed Brauer
category is fully faithful and essentially surjective and hence is an
equivalence.

For any $\mathbf p$, the algebra $A_{\mathbf p}$ is $K\fS_{\mathbf p}=K(\fS_{p_1}\times\fS_{p_2})$.
Also for any $\mathbf r$ we have an inclusion $K\fS_{\mathbf r}\subseteq D_{\mathbf r}$.

The directed Brauer category has a decomposition as a direct sum of subcategories.
The components are indexed by $\bZ$. A sign sequence with $r$ plus
signs and $s$ minus signs
is an object of block $k$ where $k=r-s$. The zero component is a symmetric monoidal
subcategory. The intersection of the zero component with the walled Brauer category
has objects $(r,r)$. Another subcategory is the full subcategory with objects
$(+,-)^r$ for $r\ge 0$. These two inclusions are both fully faithful and essentially
surjective and so are equivalences. In fact both subcategories are isomorphic.
This category is denoted $\dc_A$.
This a diagram category in the sense of section \ref{section:dc}.
An alternative description of this as a diagram category is given in
\cite{hexagons}.

Since $\dc_A$ is a symmetric monoidal category we have inclusions
$\fS_r\rightarrow D_{(r,r)}$. These inclusions are given by the
composition of the diagonal map $\fS_r \rightarrow \fS_r\times\fS_r$
with the inclusions above $\fS_r\times\fS_r \rightarrow D_{r,r}$.

Although we do not make use of this in this paper we remark that the
object $[1]$ in the category $\dc_A$ is a symmetric Frobenius
algebra.
\subsection{Frobenius characters for diagram
  algebras}\label{sec:mixed-stable}

In this section we employ the results of Section~\ref{sec:inflation},
adapted in the obvious way, to compute the Frobenius characters of
the restriction of modules of the diagram algebras $D_{\mathbf r}$ to
$K(\fS_{r_1}\times\fS_{r_2})$-modules.

Let us first provide a decomposition for the diagrams in $D(\mathbf
r,\mathbf s;\mathbf p)$.

\begin{lemma} Any walled Brauer diagram in $D(\mathbf r,\mathbf
  s;\mathbf p)$ has a unique decomposition as
  
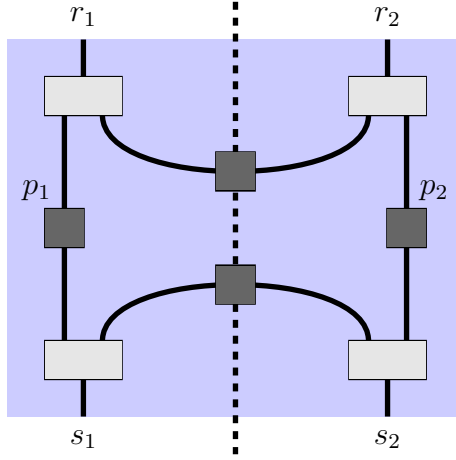
\begin{equation*}
    \begin{tikzpicture}[line width = 2pt]
      \fill[color=blue!20] (0,0) rectangle  (6,5);
      \draw (1,5) node[anchor=south] {$r_1$} -- (1,4);
      \draw (0.75,4) -- (0.75,3) node[anchor=east] {$p_1$} -- (0.75,1);
      \draw (1,1) -- (1,0) node[anchor=north] {$s_1$};
      \draw[dashed] (3,5.5) -- (3,-0.5);
      \draw (5,5) node[anchor=south] {$r_2$} -- (5,4);
      \draw (5.25,4) -- (5.25,3)  node[anchor=west] {$p_2$} -- (5.25,1);
      \draw (5,1) -- (5,0) node[anchor=north] {$s_2$};

      \draw (1.25,1) .. controls (1.25,2) and (4.75,2) .. (4.75,1);
      \draw (1.25,4) .. controls (1.25,3) and (4.75,3) .. (4.75,4);

      \draw[line width = 1pt] (0.5,2.25) rectangle (1,2.75);
      \fill[color = black!60] (0.5,2.25) rectangle (1,2.75);
      \draw[line width = 1pt] (5,2.25) rectangle (5.5,2.75);
      \fill[color = black!60] (5,2.25) rectangle (5.5,2.75);

      \draw[line width = 1pt] (2.75,3.00) rectangle (3.25,3.5);
      \fill[color = black!60] (2.75,3.00) rectangle (3.25,3.5);
      \draw[line width = 1pt] (2.75,1.5) rectangle (3.25,2);
      \fill[color = black!60] (2.75,1.5) rectangle (3.25,2);

      \draw[line width = 1pt] (0.5,4.0) rectangle (1.5,4.5);
      \fill[color = black!10] (0.5,4.0) rectangle (1.5,4.5);
      \draw[line width = 1pt] (0.5,0.5) rectangle (1.5,1.0);
      \fill[color = black!10] (0.5,0.5) rectangle (1.5,1.0);
      \draw[line width = 1pt] (4.5,4.0) rectangle (5.5,4.5);
      \fill[color = black!10] (4.5,4.0) rectangle (5.5,4.5);
      \draw[line width = 1pt] (4.5,0.5) rectangle (5.5,1.0);
      \fill[color = black!10] (4.5,0.5) rectangle (5.5,1.0);

    \end{tikzpicture}
  \end{equation*}
  The four pale grey rectangles are shuffles and the four dark grey
  squares are permutations.
\end{lemma}

\begin{defn}
  For $r_1 \geq p_1\geq 0$ and $r_2 \geq p_2\geq 0$ let $\bar
  D(\mathbf p,\mathbf r;\mathbf p)$ be the set of walled Brauer
  diagrams of the form
  
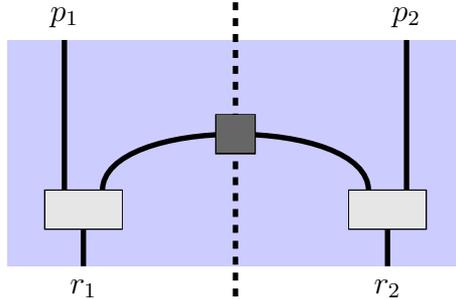
\begin{equation*}
    \begin{tikzpicture}[line width = 2pt]
      \fill[color=blue!20] (0,0) rectangle  (6,3);
      \draw (0.75,3) node[anchor=south] {$p_1$} -- (0.75,2.5) -- (0.75,1.5) -- (0.75,1);
      \draw (1,1) -- (1,0) node[anchor=north] {$r_1$};
      \draw[dashed] (3,3.5) -- (3,-0.5);
      \draw (5.25,3) node[anchor=south] {$p_2$} -- (5.25,2.5) -- (5.25,1.5) -- (5.25,1);
      \draw (5,1) -- (5,0) node[anchor=north] {$r_2$};

      \draw (1.25,1) .. controls (1.25,2) and (4.75,2) .. (4.75,1);

      \draw[line width = 1pt] (2.75,1.5) rectangle (3.25,2);
      \fill[color = black!60] (2.75,1.5) rectangle (3.25,2);

      \draw[line width = 1pt] (0.5,0.5) rectangle (1.5,1.0);
      \fill[color = black!10] (0.5,0.5) rectangle (1.5,1.0);
      \draw[line width = 1pt] (4.5,0.5) rectangle (5.5,1.0);
      \fill[color = black!10] (4.5,0.5) rectangle (5.5,1.0);
    \end{tikzpicture}
  \end{equation*}
  and define $\bar D(\mathbf r,\mathbf p;\mathbf p)$ in an analogous
  way.
\end{defn}

\begin{cor}\label{cor:decomposition-Permutation}
  For $r_1, s_1 \geq p_1\geq 0$ and $r_2,s_2 \geq p_2\geq 0$ the
  composition map
  \[
  \circ: \bar D(\mathbf r,\mathbf p;\mathbf p) \times D(\mathbf
  p,\mathbf p; \mathbf p) \times \bar D(\mathbf p,\mathbf s;\mathbf
  p)\to D(\mathbf r,\mathbf s;\mathbf p)
  \]
  is a bijection.
\end{cor}

Let us now proceed in obvious analogy to Section~\ref{sec:inflation}.
Let $E = D_{\mathbf r}/\big(\cJ_{\mathbf r}(p_1-1,p_2)\oplus
\cJ_{\mathbf r}(p_1,p_2-1)\big)$.  Note that for any diagram in
$D_{\mathbf r}$ with propagating numbers $\mathbf p$ we have $r_1-p_1
= r_2-p_2$.  Therefore, the element of $D_{\mathbf r}$ corresponding
to the diagram
\begin{equation}
  \begin{tikzpicture}[line width = 2pt]
    \fill[color=blue!20] (0,0) rectangle  (6,1.5);
    \draw[dashed] (3,-0.25) -- (3,1.75);
    \draw (1,0) node[anchor=north] {$p_1$} -- (1,1.5) {};
    \draw (5,0) node[anchor=north] {$p_2$} -- (5,1.5) {};
    \draw (2,1.5) .. controls (2,1) and (4,1) .. (4,1.5); 
    \draw (2,0) node[anchor=north] {$r_1-p_1$} .. controls (2,0.5) and (4,0.5) .. (4,0);
    \draw (6,0) node[anchor=west] {,};
  \end{tikzpicture}
\end{equation}
multiplied with $\delta^{-(r_1-p_1)}$ is an idempotent, $e_{\mathbf
  p}$.  Then, analogous to Theorem~\ref{thm:strict}, we have a strict
Morita context and inflation is an equivalence from $K\fS_{p_1}\times
\fS_{p_2}$-modules to {\hu} $\cJ_{\mathbf r}(\mathbf
p)/\big(\cJ_{\mathbf r}(p_1-1,p_2)\oplus \cJ_{\mathbf
  r}(p_1,p_2-1)\big)$-modules.

The following are the main theorems of this section.

Theorem~\ref{thm:dirdg} is implicit in the results of \cite{MR792707}
and the corresponding character formula is \cite[Corollary
7.24]{MR1405593}. 
\begin{thm}\label{thm:dirdg}
  Let $V=V_1\otimes V_2$ be an $\fS_{p_1}\times\fS_{p_2}$-module.
  Then the $\fS_{r_1}\times\fS_{r_2}$-module $\big(\Inf^{\mathbf
    r}_{\mathbf p} V\big)\downarrow^{D_{\mathbf
      r}}_{\fS_{r_1}\times\fS_{r_2}}$ is isomorphic to
  \begin{equation}
    \label{eq:dirdg}
    \bigoplus_{\lambda\vdash k} V_1 \cdot S^\lambda
    \otimes V_2 \cdot S^\lambda,
  \end{equation}
  where we denote $r_1-p_1 = r_2-p_2$ by $k$.
\end{thm}
\begin{rem}
  Using the notation for species, formula~\eqref{eq:dirdg} can be
  rewritten as
  \[
  \sum_{\lambda\vdash k} V_1(Y_1)\cdot S^\lambda(Y_1)\cdot
  V_2(Y_2)\cdot S^\lambda(Y_2).
  \]
\end{rem}
\begin{proof}
  Assume $K$ is a field of characteristic zero.  There is a natural
  isomorphism of $K\fS_{\mathbf p}$-$K\fS_{\mathbf r}$ bimodules
  \begin{equation*}
    D(\mathbf r,\mathbf p;\mathbf p)\cong K\fS_{\mathbf p} \otimes_{K\fS_k} K\fS_{\mathbf r}
  \end{equation*}
  This is then isomorphic to
  \begin{equation*}
    \bigoplus_{\lambda\vdash k} K\fS_{p_1} \cdot S^\lambda
    \otimes K\fS_{p_2} \cdot S^\lambda.
  \end{equation*}
\end{proof}

The version using combinatorial species looks quite different:
\begin{thm}\label{thm:dirdg-species}
  For an arbitrary $\fS_{p_1}\times\fS_{p_2}$-module $V(Y_1, Y_2)$,
  the $\fS_{r_1}\times\fS_{r_2}$-module $\big(\Inf^{\mathbf
    r}_{\mathbf p} V\big)\downarrow^{D_{\mathbf
      r}}_{\fS_{r_1}\times\fS_{r_2}}$ is isomorphic, as a species,
to the homogeneous part of degree $r_1$ in sort $Y_1$ and $r_2$ in sort
  $Y_2$ in the bivariate species
  \[
  V(Y_1, Y_2)\cdot H(Y_1\cdot Y_2).
  \]

  When $V(Y_1, Y_2)$ is a permutation representation, the isomorphism
  is an isomorphism of permutation representations.
\end{thm}
\begin{proof}
  Consider the combinatorial species in four variables
  \[
  H\big(X_1\cdot Y_1 + X_2\cdot Y_2 + Y_1\cdot Y_2\big).
  \]
  It follows immediately from
  Corollary~\ref{cor:decomposition-Permutation} that the restriction
  of this species to sets of sort $X_i$ of cardinality $p_i$ and sets
  of sort $Y_i$ of cardinality $r_i$ is isomorphic to the
  $\fS_{p_1}\times\fS_{p_2}\times\fS_{r_1}\times\fS_{r_2}$-module
  $KD(\mathbf p, \mathbf r; \mathbf p)$, regarded as a species.
  
  As a consequence of the linearity of the scalar product and
  Lemma~\ref{lem:sets-reproducing-kernel} we have
  \[
  \big\langle V(X_1,X_2), H(X_1\cdot Y_1 + X_2\cdot Y_2 %
  + Y_1\cdot Y_2)\big\rangle_{X_1, X_2} %
  = V(Y_1,Y_2) \cdot H(Y_1\cdot Y_2),
  \]
  which implies the claim by remark~\ref{rem:inflation}.
\end{proof}

Theorems~\ref{thm:dirdg} and \ref{thm:dirdg-species} are essentially
the same result. The character identity that arises from comparing
the results is the Cauchy identity.

Let us now specialise to the diagram algebra $D_r$ in the category
$\dc_A$.  This is the diagram algebra $D_{r,r}$ in the category
$\dc_{\vec\Br(\delta)}$ with the diagonal action of the symmetric
group $\fS_r$.  We first record a direct consequence of
Theorem~\ref{thm:dirdg}.
\begin{thm}\label{thm:dirdg-diagonal}
  Let $V=V_1\otimes V_2$ be an $\fS_{p}\times\fS_{p}$-module.
  Then the $\fS_{r}$-module $\big(\Inf^{r}_{p} V\big)\downarrow^{D_{r}}_{\fS_{r}}$
 is isomorphic to
 \begin{equation}
   \label{eq:dirdg-diagonal}
   \bigoplus_{\lambda\vdash k} V_1 \cdot S^\lambda
   \ast V_2 \cdot S^\lambda.
 \end{equation}
\end{thm}
\begin{rem}
  Using the notation for species, formula~\eqref{eq:dirdg-diagonal}
  can be rewritten as
  \[
  \sum_{\lambda\vdash k} V_1\cdot S^\lambda\ast V_2\cdot S^\lambda.
  \]
\end{rem}

Taking the diagonal of the result of Theorem~\ref{thm:dirdg-species}
requires significantly more work.  The following theorem is derived
by means of symmetric functions by
Stanley~\cite[Theorem~6.2]{MR768993}, see also~\cite[Proposition 4.8,
Proposition 7.8 and Corollary 8.5]{MR899903}.

Recall that $L_+$ is the species of non-empty linear orders and
$S=H\circ C_+$ is the species of permutations.
\begin{thm}\label{thm:dirdg-diagonal-species}
  Let $V(Y_1,Y_2)$ be an arbitrary $\fS_{p}\times\fS_{p}$-module.
  Then the $\fS_{r}$-module $\big(\Inf^{r}_{p}
  V\big)\downarrow^{D_{r}}_{\fS_{r}}$ is isomorphic to
  \[
  S \cdot \big(\nabla V(Y_1,Y_2)\big)(L_+).
  \]
\end{thm}

\begin{proof}
  Given Theorem~\ref{thm:dirdg-species} it remains to show that
  \[
  \nabla \big(V(Y_1,Y_2) \cdot H(Y_1\cdot Y_2)\big) %
  = S\cdot\big(\nabla V(Y_1,Y_2)\big)(L_+),
  \]
  where $L_+$ is the species of non-empty linear orders and $S$ is
  the species of permutations.

  Let $U$ be a finite set and consider an element $w$ of 
  \[
  V(Y_1,Y_2) \cdot H(Y_1\cdot Y_2)[U,U].
  \]
  By the definition of \lq$\cdot$\rq, this is an element of
  \[
  V[U_1, U_3]\times H(Y_1\cdot Y_2)[U_2,U_4],
  \]
  for some sets $U_1$--$U_4$ with $U = U_1\cup U_2 = U_3\cup U_4$ and
  $U_1\cap U_2 = U_3\cap U_4 = \emptyset$.  We can identify $w$ with
  an element of $S[V_1]\times\big(\nabla V(Y_1,Y_2)\big)(L_+)[V_2]$
  for two disjoint sets $V_1$ and $V_2$ with union $U$ as follows:

  Let $w'$ be the projection of $w$ to $H(Y_1\cdot Y_2)[U_2,U_4]$.
  Thus, $w'$ can be interpreted as a bijection between $U_2$ and
  $U_4$.  Let $V_1$ be the maximal subset of $U_2\cap U_4$ which is
  permuted by $w'$.  The restriction of $w'$ to $H(Y_1\cdot
  Y_2)[V_1,V_1]$ can then be identified with an element of $S[V_1]$.

  Let $V_2=U\setminus V_1$ and consider the set $\big(\nabla
  V(Y_1,Y_2)\big)(L_+)[V_2]$.  By the definition of the composition
  of species this set equals
  \[
  \nabla V(Y_1,Y_2)[\pi]\times L_+[\pi_1] \times\dots\times
  L_+[\pi_k]
  \]
  for some set partition $\pi=(\pi_1,\dots,\pi_k)$ of $V_2$.  We
  regard the elements of $U_1\cap U_3$ as singletons of $\pi$.  The
  remaining linearly ordered blocks of the set partition are obtained
  by choosing an element $a$ in $U_2\cap U_3$, and then repeatedly
  applying $w'$ until we obtain an element in $U_1\cap U_4$:
  \[
  a,w'(a),w'\big(w'(a)\big),\dots
  \]
  Thus, $U_1$ consists precisely of the final elements of the blocks
  (including singletons), while $U_3$ is the collection of initial
  elements of the blocks (including singletons).  Therefore, there is
  a natural bijection between $\nabla V(Y_1,Y_2)[\pi]$ and
  $V(Y_1,Y_2)[U_1,U_3]$, and a natural bijection between
  $S[V_1]\times L_+[\pi_1] \times\dots\times L_+[\pi_k]$ and
  bijections $w'$.
\end{proof}

Theorems~\ref{thm:dirdg-diagonal} and \ref{thm:dirdg-diagonal-species} are essentially
the same result. Hence comparing the answers gives the character
identity
\begin{equation*}
  \sum_\lambda (s_\alpha s_\lambda)\ast(s_\beta s_\lambda )
  =   \left(\prod_{k\geq 1} (1-p_k)^{-1}\right)\cdot \left( s_\alpha \ast s_\beta  \circ \frac{s_1}{1-s_1}\right).
\end{equation*}
where we have used
\begin{equation*}
 H\circ C_+ = \sum_\lambda p_\lambda = \left(\prod_{k\geq 1} (1-p_k)^{-1}\right)
\end{equation*}

\subsection{Branching rules}

For $r\ge 0$ we have an inclusion $D_{(r_1,r_2)}\rightarrow D_{(r_1+1,r_2)}$.
Here we study the decomposition of the irreducible $D_{(r_1,r_2+1)}$-modules
restricted to a $D_{(r_1,r_2)}$-module.

\begin{prop}[\protect{\cite[Theorem~3.16]{MR1405593}}]
  \label{prop:dir} For $\mathbf r\ge \mathbf p\ge 0$ and $V$ any
  $A_{\mathbf p}$-module,
  \begin{multline*}
    \big(\Inf_{\mathbf p}^{(r_1,r_2+1)}V\big)\downarrow^{D_{(r_1,r_2+1)}}_{D_{(r_1,r_2)}} \\
    \cong \Inf_{(p_1+1,p_2)}^{(r_1,r_2)}\left(V\uparrow_{A_{\mathbf
        p}}^{A_{(p_1+1,p_2)}}\right) \oplus
    \Inf_{(p_1,p_2-1)}^r\left(V\downarrow_{A_{(p_1,p_2-1)}}^{A_{\mathbf
        p}}\right).
 \end{multline*}
\end{prop}

\begin{proof}
  Let $V$ be a representation of $A_{\mathbf p}=K(\fS_{p_1}\times
  \fS_{p_2})$.  We decompose the $A_{\mathbf p}$-$D_{(r_1,r_2+1)}$
  bimodule $KD({\mathbf p},(r_1,r_2+1);{\mathbf p})$ restricted to
  $D_{\mathbf r}$ into a direct sum.  

  The set of diagrams in $D({\mathbf p},(r_1,r_2+1);{\mathbf p})$
  whose last point in the bottom row to the right of the wall is
  matched with a point in the bottom row (necessarily to the left of
  the wall) generates a submodule isomorphic to $KD((p_1+1,
  p_2),{\mathbf r};(p_1+1,
  p_2))\downarrow^{A_{(p_1+1,p_2)}}_{A_{\mathbf p}}$.  Informally,
  the isomorphism is given by moving the last point from the bottom
  row just to the left of the wall in the top row.

  The remaining diagrams in $D({\mathbf p},(r_1,r_2+1);{\mathbf p})$,
  whose last point in the bottom row is matched with a point in the
  top row generate a submodule isomorphic to $KD((p_1,
  p_2-1),{\mathbf r};(p_1,
  p_2-1))\uparrow^{A_{\mathbf p}}_{A_{(p_1,p_2-1)}}$.

  Tensoring with $V$ and applying the definition of induction we
  obtain
  \begin{multline*}
    V\otimes_{A_{\mathbf p}} KD((p_1+1,
    p_2),{\mathbf r};(p_1+1,
    p_2))\downarrow^{A_{(p_1+1,p_2)}}_{A_{\mathbf p}}\\
    \cong V\uparrow^{A_{(p_1+1,p_2)}}_{A_{\mathbf p}}%
    \otimes_{A_{(p_1+1,p_2)}} %
    KD((p_1+1, p_2),{\mathbf r};(p_1+1, p_2))
  \end{multline*}
  and
  \begin{multline*}
    V\otimes_{A_{\mathbf p}} KD((p_1, p_2-1),{\mathbf r};(p_1,
    p_2-1))\uparrow^{A_{\mathbf p}}_{A_{(p_1,p_2-1)}}\\
    \cong V\downarrow^{A_{\mathbf p}}_{A_{(p_1,p_2-1)}}%
    \otimes_{A_{(p_1,p_2-1)}} %
    KD((p_1, p_2-1),{\mathbf r};(p_1, p_2-1)).
  \end{multline*}
\end{proof}


\begin{cor}\label{cor:directed-branching-rules}  
  For $\mathbf{r}\ge \mathbf{p}\ge 0$ and
  $\boldsymbol{\lambda}\vdash\mathbf{p}$ put $U(\boldsymbol{\lambda},
  \mathbf{r})=\Inf_p^r\left(S^{\lambda_1}\otimes
    S^{\lambda_2}\right)$.  Then the branching rules are given by
  \begin{equation*}
    U\big((r_1,r_2+1),
    \boldsymbol{\lambda}\big)\downarrow^{(r_1,r_2+1)}_{\mathbf r}%
    \cong 
    \bigoplus_{\substack{%
        \boldsymbol{\mu} = (\lambda_1+\square, \lambda_2)\\
        \text{ or } \boldsymbol{\lambda} = (\mu_1, \mu_2+\square)}}
    U(\mathbf{r}, \boldsymbol{\mu}).
  \end{equation*}
\end{cor}
\begin{proof} 
  The proof is completely analogous to the proof of
  Proposition~\ref{cor:Brauer-branching-rules} and therefore omitted.
\end{proof}

\subsection{Fundamental theorems}
\label{sec:fundamental-theorems-GL}
Let $\T_{M(n)}$ be the category of mixed tensors, that is, the monoidal
category generated by the vector spaces $V$ and $V^\ast$, where $V$
is the defining representation of $\GL(n)$ and $V^\ast$ the dual
representation.

Let $\dc_{M(n)}$ be the diagram category $\dc_{\vec\Br(\delta)}$ with
$\delta=n$.

\begin{defn} For each $n>0$, there is a functor
  $\ev_{M(n)}\colon D_{M(n)}\rightarrow \T_{M(n)}$ uniquely determined
by the properties that
  $\ev_{M(n)}(+)=V$ and $\ev_{M(n)}(-)=V^\ast$
and that the functor is pivotal symmetric monoidal.
\end{defn}
This functor is also determined by the condition that it is monoidal and its
values on the monoidal generators. The monoidal generators are the diagrams
obtained from \eqref{eq:cupcap} and \eqref{eq:perm} by taking all
orientations.

Let $\{b_1,\dotsc,b_n\}$ be a basis of $V$ and $\{\bar b_1,\dots,\bar
b_n\}$ be the dual basis.  Then the values on the four diagrams
obtained by orienting the diagrams in \eqref{eq:cupcap} are
$\phi\otimes v\mapsto \langle \phi,v\rangle$ and $v\otimes
\phi\mapsto \langle v,\phi\rangle$ for $v\in V$ and $\phi\in V^*$
together with $1\mapsto \sum_i b_i \otimes \bar b_i$ and $1\mapsto
\sum_i \bar b_i\otimes b_i$.

This functor restricts to the walled Brauer category. This restriction is given on objects
by $(r,s)\mapsto (\otimes^rV)\otimes(\otimes^sV^*)$.

The first fundamental theorem for the mixed tensors can be stated as
follows: see \cite{MR2058514},
\begin{thm}\cite[Theorem (2.6A)]{MR1488158}
  \label{thm:FFT-SFT-mixed}
  For $n>0$, the evaluation functor $\ev_{M(n)}\colon\dc_{M(n)}\rightarrow
  \T_{M(n)}$ is full.
\end{thm}

\begin{thm}[\protect{\cite{MR792707},\cite[Theorem 5.8]{MR1280591}}]
  Let $\ev = \ev_{M(n)}$ be the functor described above.  Then the
  map $\ev_{(r_1,r_2)}^{(s_1,s_2)}$ is injective for $2n\geq
  r_1+r_2+s_1+s_2$.
\end{thm}

The rank one central idempotents are constructed in \cite{MR3210415}.

\subsection{Frobenius characters for tensor algebras}
\label{sec:Mixed-tensor-algebra}

Recall that the irreducible rational representations of $\GL(n)$ are
indexed by \Dfn{staircases} of length $n$, that is, sequences
$\boldsymbol{\mu}$ of $n$ integers satisfying
$\mu_1\geq\dots\geq\mu_n$.

There is a staircase $[\alpha,\beta]_n$ for each pair of partitions
$[\alpha,\beta]$ with $n\ge |\alpha|+|\beta|$. This is given by
\begin{equation*}
 [\alpha,\beta]_n = \left(\sum_j \alpha_je_j\right) - \left(\sum_j\beta_je_{n-j+1}\right) 
\end{equation*}

\begin{thm}\label{thm:Mixed-tensor-algebra}
  Let $\boldsymbol{\mu}$ be a staircase of length $n$ and let
  $W(\boldsymbol{\mu})$ be the irreducible rational representation of
  $\GL(n)$ with highest weight $\boldsymbol{\mu}$.  Then for $n\geq
  r_1+r_2$ there is a natural isomorphism
  \begin{equation*}
    \Hom_{\GL(n)} (W(\boldsymbol{\mu}),\otimes^{r_1}V\otimes\otimes^{r_2}V^\ast)%
    \cong
    \Inf_{\mathbf p}^{\mathbf r}\left(S^{\boldsymbol{\mu}}\right)
  \end{equation*}
  of $D_{\mathbf r}$-modules.
\end{thm}
\begin{proof}
  The classical branching rules are
  \begin{align*}
    W(\boldsymbol{\mu}) \otimes V &\cong %
    \bigoplus_{\substack{\boldsymbol{\lambda} = (\mu_1-\square, \mu_2)\\
        \text{ or } \boldsymbol{\lambda} = (\mu_1, \mu_2+\square)}}
    W(\boldsymbol{\lambda})\\
    W(\boldsymbol{\mu}) \otimes V^\ast &\cong %
    \bigoplus_{\substack{\boldsymbol{\lambda} = (\mu_1, \mu_2+\square)\\
        \text{ or } \boldsymbol{\lambda} = (\mu_1+\square, \mu_2)}}
    W(\boldsymbol{\lambda}),
  \end{align*}
  the remainder of the proof is as in the other sections.
\end{proof}

\subsection{Cyclic sieving phenomenon}
\label{sec:Permutations-CSP}

In contrast to the results in Section~\ref{sec:Brauer-CSP} and
Section~\ref{sec:Partition-CSP} we can only exhibit an instance of
the cyclic sieving phenomenon for the case when $n\geq 2r$.

 Let $X=X(r)$ be the set of permutations of $\{1,\dots,r\}$ and
  let $\rho$ be the long cycle, that is, $\rho$ acts on
  $\{1,\dots,r\}$ as $\rho(i)=i\pmod r +1$. Then conjugation
by $\rho$ generates an action of the cyclic group of order $r$
on $X$.

Take a disc with $2r$ boundary points labelled clockwise by
${1,1',2,2',\dotsc ,r,r'}$. Let $X$ be the set of directed perfect
matchings directed from an unprimed label to a primed label;
equivalently, $X$ is a bijection from the set of unprimed boundary
points to the set of primed boundary points.
Let $\rho$ be the rotation map which moves each boundary point
clockwise two places.

Then the map which sends a permutation $\pi$ to the directed matching
\begin{equation*}
 \{ (1,\pi(1)'), (2,\pi(2)') , \dotsc , (r,\pi(r)') \}
\end{equation*}
is a bijection compatible with the two cyclic actions.

These both have an action of $\fS_r$ and the bijection is compatible with
these actions. The Frobenius character of this permutation representation is
\begin{equation*}
 \sum_{\lambda\vdash r} p_\lambda = \sum_\lambda s_\lambda\ast s_\lambda
\end{equation*}
where $p_\lambda$ is a power sum function.

\begin{thm} Let $X$ be the set of permutations of $\{1,\dots,r\}$ and
  let $\rho:X\to X$ be the rotation map, that is, $\rho$ acts on
  $\{1,\dots,r\}$ as $\rho(i)=i\pmod r+1$.  Let
  \begin{equation*}
    P(q) = \fd \ch S_r = \fd \sum_{\lambda\vdash r} p_\lambda,
  \end{equation*}
  where $S_r$ is the species of permutations of sets of cardinality
  $r$ and $p_\lambda$ are the power sum symmetric functions.  Then
  the triple $\big(X, \rho, P(q)\big)$ exhibits the cyclic sieving
  phenomenon.
\end{thm}

When $n<2r$, we can still compute the Frobenius character of
$\Hom_{\T_{\GL(n)}}(0, r)$.  Note that it is clear that
$\Hom_{\T_{M(n)}}\big((0,0),\mathbf r\big)$ is zero if $r_1\neq r_2$.
\begin{thm}
  As a representation of $\fS_r\times\fS_r$,
  \[
  \Hom_{\T_{M(n)}}\big((0,0),(r,r)\big)\cong
  \bigoplus_{\substack{\lambda\vdash r\\l(\lambda)\leq n}}
  S^\lambda\otimes S^\lambda
  \]
\end{thm}
\begin{proof} 
  This uses the results in \S \ref{sec:perm}.

  Consider $\Hom_{\T_{M(n)}}\big((0,0),(r,r)\big)$ as a
  $K\fS_r$-$K\fS_r$ bimodule. This bimodule is isomorphic to the
  quotient of the regular representation of $K\fS_r$.
  
  This quotient representation is isomorphic to the bimodule
  \begin{equation*}
    \bigoplus_{\substack{\lambda\vdash r\\l(\lambda)\leq n}} S^\lambda\otimes S^\lambda
  \end{equation*}
\end{proof}

\begin{thm}\label{thm:inv} 
  There is an isomorphism of $\fS_r$-modules
 \begin{equation*}
   U(r, [\emptyset,\emptyset])\cong \sum_{\substack{\lambda\vdash r\\ \ell(\lambda)\le n}}
   S^\lambda\ast S^\lambda
 \end{equation*}
\end{thm}

There is a cyclic sieving phenomenon associated to each tensor power of the adjoint
representation in \cite{1512}. In order to have a combinatorial interpretation of these
cyclic sieving phenomena we require a basis of the invariant
tensors which is invariant under rotation. For $n=2$ such a basis can be constructed
using Temperley-Lieb diagrams, see \cite{MR1446615}. For $n=3$ such a basis is constructed in \cite{MR1403861}.
A basis invariant under rotation is constructed for all $n$ in
\cite{MR1227098} but this basis and the rotation map are not explicit.

\printbibliography
\end{document}